\documentclass[11pt, a4paper,leqno]{amsart}
\usepackage{amsmath,amsthm,amscd,amssymb,amsfonts, amsbsy}
\usepackage{latexsym}
\usepackage{txfonts}
\usepackage{exscale}

\usepackage[colorlinks,citecolor=red,pagebackref,hypertexnames=false]{hyperref}

\usepackage{pgf}

\calclayout
\allowdisplaybreaks

\theoremstyle{plain}
\newtheorem{theorem}[equation]{Theorem}
\newtheorem{lemma}[equation]{Lemma}

\theoremstyle{definition}
\newtheorem{definition}[equation]{Definition}

\theoremstyle{remark}
\newtheorem{remark}[equation]{Remark}

\numberwithin{equation}{section}

\newcommand{\R}{\mathcal{R}}

\renewcommand{\P}{\mathcal{P}}

\def\mean#1{\mathchoice%
          {\mathop{\kern 0.2em\vrule width 0.6em height 0.69678ex depth -0.58065ex
                  \kern -0.8em \intop}\nolimits_{\kern -0.4em#1}}%
          {\mathop{\kern 0.1em\vrule width 0.5em height 0.69678ex depth -0.60387ex
                  \kern -0.6em \intop}\nolimits_{#1}}%
          {\mathop{\kern 0.1em\vrule width 0.5em height 0.69678ex
              depth -0.60387ex
                  \kern -0.6em \intop}\nolimits_{#1}}%
          {\mathop{\kern 0.1em\vrule width 0.5em height 0.69678ex depth -0.60387ex
                  \kern -0.6em \intop}\nolimits_{#1}}}

\renewcommand{\emptyset}{\mbox{\textup{\O}}}

\def\div{\mathop{\operatorname{div}}\nolimits}

\begin{document}
\title[Boundedness of parabolic layer potentials]{Boundedness of single layer potentials associated to  divergence form parabolic equations with complex coefficients}

\address{Alejandro J. Castro\\Department of Mathematics, Uppsala University\\
S-751 06 Uppsala, Sweden}
\email{alejandro.castro@math.uu.se}
\address{Kaj Nystr\"{o}m\\Department of Mathematics, Uppsala University\\
S-751 06 Uppsala, Sweden}
\email{kaj.nystrom@math.uu.se}
\address{Olow Sande\\Department of Mathematics, Uppsala University\\
S-751 06 Uppsala, Sweden}
\email{olow.sande@math.uu.se}

\author{A. J.  Castro, K. Nystr{\"o}m, O. Sande}

\maketitle
\begin{abstract}
\noindent\medskip
We consider parabolic operators of the form $$\partial_t+\mathcal{L},\ \mathcal{L}:=-\mbox{div}\, A(X,t)\nabla,$$ in
$\mathbb R_+^{n+2}:=\{(X,t)=(x,x_{n+1},t)\in \mathbb R^{n}\times \mathbb R\times \mathbb R:\ x_{n+1}>0\}$, $n\geq 1$. We assume that $A$ is a $(n+1)\times (n+1)$-dimensional matrix which is bounded, measurable, uniformly elliptic and complex, and we assume, in addition, that the entries of A are independent of the spatial coordinate $x_{n+1}$ as well as of the time coordinate $t$. We prove that the boundedness of associated single layer potentials, with data in $L^2$, can be reduced to two crucial estimates (Theorem \ref{th0}), one being a square function estimate involving the single layer potential. By establishing a local parabolic Tb-theorem for square functions we are then able to verify the two  crucial estimates in the case of real, symmetric operators (Theorem \ref{th2}). As part of this argument we establish a scale-invariant reverse H{\"o}lder inequality for the parabolic Poisson kernel (Theorem \ref{parabolicm}).  Our results are important when addressing the solvability of the classical Dirichlet, Neumann and Regularity
problems for the operator $\partial_t+\mathcal{L}$ in $\mathbb R_+^{n+2}$, with $L^2$-data on $\mathbb R^{n+1}=\partial\mathbb R_+^{n+2}$, and by way of layer potentials.\\

\noindent
2010  {\em Mathematics Subject Classification: 35K20, 31B10}
\noindent

\medskip

\noindent
{\it Keywords and phrases: second order parabolic operator, complex coefficients, boundary value problems, layer potentials.}
\end{abstract}

    \section{Introduction and statement of main results}

    In this paper  we establish certain estimates related to the solvability of the Dirichlet, Neumann and Regularity
problems with data in $L^2$, in the following these problems are referred to as $(D2)$, $(N2)$ and $(R2)$,  by way of layer potentials and for second order parabolic equations of the form
    \begin{eqnarray}\label{eq1}
    \mathcal{H}u:=(\partial_t+\mathcal{L})u = 0,
    \end{eqnarray}
     where
    \begin{eqnarray*}
    \mathcal{L}:=-\mbox{div }A(X,t)\nabla =-\sum_{i,j=1}^{n+1}\partial_{x_i}(A_{i,j}(X,t)\partial_{x_j})
    \end{eqnarray*}
    is defined in $\mathbb R^{n+2}=\{(X,t)=(x_1,..,x_{n+1},t)\in \mathbb R^{n+1}\times\mathbb R\}$, $n\geq 1$. $A=A(X,t)=\{A_{i,j}(X,t)\}_{i,j=1}^{n+1}$ is assumed to be a $(n+1)\times (n+1)$-dimensional matrix with complex coefficients
    satisfying the uniform ellipticity condition
     \begin{eqnarray}\label{eq3}
     (i)&&\Lambda^{-1}|\xi|^2\leq \mbox{Re }\bigl (\sum_{i,j=1}^{n+1} A_{i,j}(X,t)\xi_i\bar\xi_j\bigr ),\notag\\
     (ii)&& |A\xi\cdot\zeta|\leq \Lambda|\xi||\zeta|,
    \end{eqnarray}
    for some $\Lambda$, $1\leq\Lambda<\infty$, and for all $\xi,\zeta\in \mathbb C^{n+1}$, $(X,t)\in\mathbb R^{n+2}$. Here $u\cdot v=u_1v_1+...+u_{n+1}v_{n+1}$,
    $\bar u$ denotes the complex conjugate of $u$ and $u\cdot \bar v$ is the (standard) inner product on $\mathbb C^{n+1}$. In addition, we consistently assume that
     \begin{eqnarray}\label{eq4}
A(x_1,..,x_{n+1},t)=A(x_1,..,x_{n}),\ \mbox{i.e., $A$ is independent of $x_{n+1}$ and $t$}.
    \end{eqnarray}
The solvability of $(D2)$, $(N2)$ and $(R2)$ for the operator $\mathcal{H}$ in $\mathbb R^{n+2}_+=\{(x,x_{n+1},t)\in \mathbb R^{n}\times\mathbb R\times\mathbb R:\ x_{n+1}>0\}$,  with data prescribed
on  $\mathbb R^{n+1}=\partial \mathbb R^{n+2}_+=\{(x,x_{n+1},t)\in \mathbb R^{n}\times\mathbb R\times\mathbb R:\ x_{n+1}=0\}$ and by way of layer potentials, can roughly be decomposed into two steps: boundedness of layer potentials and invertibility of layer potentials. In this paper we first prove, in the case of equations of the form \eqref{eq1}, satisfying \eqref{eq3}-\eqref{eq4} and the De Giorgi-Moser-Nash estimates stated in \eqref{eq14+}-\eqref{eq14++} below, that a set of key  boundedness estimates for associated single layer potentials can be reduced to two crucial estimates (Theorem \ref{th0}), one being a square function estimate involving the single layer potential. By establishing a local parabolic Tb-theorem for square functions, and by establishing a version of the main result in \cite{FS} for equations of the form \eqref{eq1}, assuming in addition that $A$ is real and symmetric,  we are then subsequently able to verify the two  crucial estimates in the case of real, symmetric operators \eqref{eq1} satisfying \eqref{eq3}-\eqref{eq4} (Theorem \ref{th2}). As part of this argument we establish, and this is of independent interest, a scale-invariant reverse H{\"o}lder inequality for the parabolic Poisson kernel (Theorem \ref{parabolicm}). The invertibility of layer potentials, and hence the solvability of the Dirichlet, Neumann and Regularity
problems $L^2$-data, is addressed in \cite{N1}.

Jointly, this paper and \cite{N1} yield solvability for $(D2)$, $(N2)$ and $(R2)$, by way of layer potentials,  when the coefficient matrix is  either
     \begin{eqnarray*}
       (i)&&\mbox{a small complex perturbation of a constant (complex) matrix, or}\notag\\
       (ii)&&\mbox{a real and symmetric matrix, or}\notag\\
       (iii)&&\mbox{a small complex perturbation of a real and symmetric matrix}.
       \end{eqnarray*}
In all cases the unique solutions can be represented
 in terms of layer potentials.  We claim that the results established in this paper and in \cite{N1}, and the tools developed, pave the way for important developments in the area of parabolic PDEs. In particular, it is interesting to generalize the present paper and \cite{N1} to the context of $L^p$ and relevant endpoint spaces, and to  challenge the assumption in \eqref{eq4}.

 The main results of this paper and \cite{N1} can jointly be seen as a parabolic analogue of the elliptic results established in \cite{AAAHK} and we recall that in \cite{AAAHK} the authors establish results concerning the solvability of the Dirichlet, Neumann and Regularity
problems with data in $L^2$, i.e.,  $(D2)$, $(N2)$ and $(R2)$, by way of  layer potentials and for elliptic operators of the form $-\mbox{div}\, A(X)\nabla,$ in
$\mathbb R_+^{n+1}:=\{X=(x,x_{n+1})\in \mathbb R^{n}\times \mathbb R:\ x_{n+1}>0\}$, $n\geq 2$, assuming  that $A$ is a $(n+1)\times (n+1)$-dimensional matrix which is bounded, measurable, uniformly elliptic and complex, and assuming, in addition, that the entries of $A$ are independent of the spatial coordinate $x_{n+1}$. Moreover, if $A$ is real and symmetric,  $(D2)$, $(N2)$ and $(R2)$ were solved in \cite{JK}, \cite{KP}, \cite{KP1}, but the major achievement in \cite{AAAHK} is that the authors prove that the solutions can be represented by way of  layer potentials. In \cite{HMM} a version of
 \cite{AAAHK}, but in the context of $L^p$ and relevant endpoint spaces, was developed and in \cite{HMaMi} the structural assumption that $A$ is independent of the spatial coordinate $x_{n+1}$ is challenged. The core of the impressive arguments and estimates in \cite{AAAHK} is based on the fine and elaborated techniques developed in the context of the proof of the Kato conjecture, see \cite{AHLMcT} and \cite{AHLeMcT}, \cite{HLMc}.

 \subsection{Notation}\label{nota}  Based on \eqref{eq4} we let $\lambda=x_{n+1}$, and when using the symbol
    $\lambda$  we will write the point $(X,t)=(x_1,..,x_{n},x_{n+1},t)$ as $ (x,t, \lambda)=(x_1,..,x_{n},t,\lambda)$. Using this notation,  $$\mathbb R^{n+2}_+=\{(x,t,\lambda)\in \mathbb R^{n}\times\mathbb R\times\mathbb R:\ \lambda>0\},$$ and  $$\mathbb R^{n+1}=\partial\mathbb R^{n+2}_+=\{(x,t,\lambda)\in \mathbb R^{n}\times\mathbb R\times\mathbb R:\ \lambda=0\}.$$
    We write $\nabla :=(\nabla_{||},\partial_\lambda)$ where $\nabla_{||}:=(\partial_{x_1},...,\partial_{x_n})$. We let $L^2(\mathbb R^{n+1},\mathbb C)$ denote the  Hilbert space of functions $f:\mathbb R^{n+1}\to \mathbb C$ which are square integrable and we let $||f||_2$ denote the norm of $f$. We also introduce   \begin{eqnarray}\label{keyestint-ex+-}|||\cdot|||:=\biggl (\int_0^\infty\int_{\mathbb R^{n+1}}|\cdot|^2\, \frac{dxdtd\lambda}\lambda\biggr )^{1/2}.
      \end{eqnarray}
      Given $(x,t)\in\mathbb R^{n}\times\mathbb R$ we let $\|(x,t)\|$ be the unique positive
solution $\rho$ to the equation
\begin{eqnarray*}
	\frac{t^2}{\rho^4}+\sum\limits^{n}_{i=1}\frac{x^2_i}{\rho^2}=1.
\end{eqnarray*}
Then $\|(\gamma x,\gamma^2t)\|=\gamma\|(x,t)\|$, $\gamma>0$, and we call $\|(x,t)\|$
the parabolic norm of $(x,t)$. We define the parabolic first order differential operator $\mathbb D$
through the relation
\begin{eqnarray*}
	\widehat{(\mathbb D f)}(\xi,\tau):=\|(\xi,\tau)\|\hat{f}(\xi,\tau),
\end{eqnarray*}
where $\widehat{(\mathbb D f)}$ and $\hat f$ denote the Fourier transform of $\mathbb D f$ and $f$, respectively. We define the
fractional (in time) differentiation operators $D_{1/2}^t$ through the relation
\begin{eqnarray*}
	\widehat {(D_{1/2}^t f)}(\xi,\tau):=|\tau|^{1/2}\hat{f}(\xi,\tau).
\end{eqnarray*}
We let $H_t$ denote a Hilbert transform in the $t$-variable defined through the multiplier $i\mbox{sgn}({\tau})$. We make the construction so that $$\partial_t=D_{1/2}^tH_tD_{1/2}^t.$$
 By applying
Plancherel's theorem  we have
\begin{eqnarray*}
\|\mathbb D f\|_{2}\approx\|\nabla_{||}  f\|_2+\|H_tD^t_{1/2}f\|_2,\end{eqnarray*}
with constants depending only on $n$.

\subsection{Non-tangential maximal functions} Given $(x_0,t_0)\in \mathbb R^{n+1}$, and $\beta>0$, we define the cone
    $$\Gamma^\beta(x_0,t_0):=\{(x,t,\lambda)\in\mathbb R^{n+2}_+:\ ||(x-x_0,t-t_0)||<\beta\lambda\}.$$
    Consider a function $U$ defined on $\mathbb R^{n+2}_+$. The non-tangential maximal operator $N_\ast^\beta$ is defined
     \begin{eqnarray*}
           N_{\ast}^\beta(U)(x_0,t_0):=\sup_{(x,t,\lambda)\in \Gamma^\beta(x_0,t_0)}\ |U(x,t,\lambda)|.
    \end{eqnarray*}
    Given $(x,t)\in\mathbb R^{n+1}$, $\lambda>0$, we let
    \begin{eqnarray*}
    		Q_\lambda(x,t):=\{(y,s):\ |x_i-y_i|<\lambda,\ |t-s|<\lambda^2\}
    \end{eqnarray*}
    denote the parabolic cube on $\mathbb R^{n+1}$, with center $(x,t)$ and side length $\lambda$. We let
    $$W_\lambda(x,t):=\{(y,s,\sigma):\ (y,s)\in Q_\lambda(x,t),\lambda/2<\sigma<3\lambda/2\}$$
    be an associated Whitney type set. Using this notation we also introduce
           \begin{eqnarray*}
           \tilde N_\ast^\beta(U)(x_0,t_0):=\sup_{(x,t,\lambda)\in \Gamma^\beta(x_0,t_0)}\biggl (
           \mean{W_\lambda(x,t)}|U(y,s,\sigma)|^2\, dydsd\sigma\biggr )^{1/2}.
    \end{eqnarray*}
  We let
\begin{eqnarray*}
           \Gamma(x_0,t_0):=\Gamma^1(x_0,t_0),\ N_{\ast}(U):=N_{\ast}^1(U),\ \tilde N_{\ast}(U):=\tilde N_{\ast}^1(U).
    \end{eqnarray*}
    Furthermore, in many estimates it is necessary to increase the $\beta$ in $\Gamma^\beta$  as the estimate progresses. We will use the convention, when the exact $\beta$ is not important, that $N_{\ast\ast}(U)$, $\tilde N_{\ast\ast}(U)$, equal $N_{\ast}^\beta(U)$, $\tilde N_{\ast}^\beta(U)$, for some $\beta>1$. In fact, the $L^p$-norms of $N_{\ast}$ and $N_{\ast}^\beta$ are equivalent,
for any $\beta>0$ (see for example \cite[Lemma 1, p. 166]{FeSt}).

\subsection{Single layer potentials} Consider $\mathcal{H}=\partial_t+\mathcal{L}=\partial_t-\div A\nabla$. Assume that $\mathcal{H}$,  $\mathcal{H}^\ast$,   satisfy \eqref{eq3}-\eqref{eq4}. Then $\mathcal{L}=-\div A\nabla$ defines, recall that $A$ is independent of $t$, a maximal accretive operator on $L^2(\mathbb R^{n+1},\mathbb C)$ and $-\mathcal{L}$ generates a contraction semigroup on $L^2(\mathbb R^{n+1},\mathbb C)$, $e^{-t\mathcal{L}}$, for $t>0$, see p.28 in \cite{AT}.  Let $K_t(X,Y)$ denote the distributional or Schwartz kernel of $e^{-t\mathcal{L}}$. In the statement of our main results, and hence throughout the paper, we will assume, in addition to \eqref{eq3}-\eqref{eq4}, that $\mathcal{H}$,  $\mathcal{H}^\ast$,  both satisfy  De Giorgi-Moser-Nash estimates stated in \eqref{eq14+}-\eqref{eq14++} below. This assumption implies, in particular,  that $K_t(X,Y)$ is, for each $t>0$, H{\"o}lder continuous in $X$ and $Y$ and that $K_t(X,Y)$ satisfies the Gaussian (pointwise) estimates stated in Definition 2 on p.29 in \cite{AT}. Under these assumptions we introduce
\begin{eqnarray*}
\Gamma(x,t,\lambda,y,s,\sigma):=\Gamma^{\mathcal{H}}(X,t,Y,s):=K_{t-s}(X,Y)=K_{t-s}(x,\lambda,y,\sigma)
\end{eqnarray*}
whenever $t-s>0$ and we put $\Gamma(x,t,\lambda,y,s,\sigma)=0$ whenever $t-s<0$. Then $\Gamma(x,t,\lambda,y,s,\sigma)$, for $(x,t,\lambda)$, $(y,s,\sigma)\in \mathbb R^{n+2}$ is a  fundamental solution, heat kernel, associated to the operator $\mathcal{H}$.   We let
$$\Gamma^\ast(y,s,\sigma,x,t,\lambda):=\overline{\Gamma(x,t,\lambda,y,s,\sigma)}$$ denote the fundamental solution associated to  $\mathcal{H}^\ast:=-\partial_t+\mathcal{L}^\ast$, where $\mathcal{L}^\ast$ is the hermitian adjoint of $\mathcal{L}$, i.e., $\mathcal{L}^\ast=-\div A^\ast\nabla$.  Based on \eqref{eq4} we in the following  let
                    \begin{eqnarray*}
                    \Gamma_\lambda(x,t,y,s)&:=&\Gamma(x,t,\lambda,y,s,0),\notag\\
                    \Gamma_\lambda^\ast(y,s,x,t)&:=&\Gamma^\ast(y,s,0,x,t,\lambda),
    \end{eqnarray*}
    and we introduce associated single  layer potentials
    \begin{eqnarray*}
    \mathcal{S}_\lambda^{\mathcal{H}} f(x,t)&:=&\int_{\mathbb R^{n+1}}\Gamma_\lambda(x,t,y,s)f(y,s)\, dyds,\notag\\
     \mathcal{S}_\lambda^{\mathcal{H}^\ast} f(x,t)&:=&\int_{\mathbb R^{n+1}}\Gamma_\lambda^\ast(y,s,x,t)f(y,s)\, dyds.
    \end{eqnarray*}

    \subsection{Statement of main results}  The following are our main results.

\begin{theorem}\label{th0} Consider $\mathcal{H}=\partial_t-\div A\nabla$. Assume that $\mathcal{H}$,  $\mathcal{H}^\ast$,   satisfy \eqref{eq3}-\eqref{eq4} as well as the De Giorgi-Moser-Nash estimates stated in \eqref{eq14+}-\eqref{eq14++} below.  Assume that there exists a constant $C $ such that
     \begin{eqnarray}\label{keyestint-}
     (i)&&\sup_{\lambda> 0}||\partial_\lambda \mathcal{S}_{\lambda}^{\mathcal{H}}f||_2+\sup_{\lambda> 0}||\partial_\lambda \mathcal{S}_{\lambda}^{\mathcal{H}^\ast}f||_2\leq C ||f||_2,\notag\\
     (ii)&&|||\lambda\partial_\lambda^2 \mathcal{S}_{\lambda}^{\mathcal{H}}f|||+|||\lambda\partial_\lambda^2 \mathcal{S}_{\lambda}^{\mathcal{H}^\ast}f|||\leq  C ||f||_2,
     \end{eqnarray}
     whenever $f\in L^2(\mathbb R^{n+1},\mathbb C)$. Then there exists a constant $c$, depending at most
     on $n$, $\Lambda$,  the De Giorgi-Moser-Nash constants and $C $, such that
     \begin{eqnarray}\label{keyestint+a}
(i)&& ||N_\ast(\partial_\lambda \mathcal{S}_\lambda^{\mathcal{H}} f)||_2+||N_\ast(\partial_\lambda \mathcal{S}_\lambda^{\mathcal{H}^\ast} f)||_2\leq c||f||_2,\notag\\
  (ii)&&\sup_{\lambda>0}||\mathbb D\mathcal{S}_{\lambda}^{\mathcal{H}}f||_{2}+\sup_{\lambda>0}||\mathbb D\mathcal{S}_{\lambda}^{\mathcal{H}^\ast} f||_{2}\leq c||f||_2,\notag\\
(iii)&&||\tilde N_\ast(\nabla_{||}\mathcal{S}_\lambda^{\mathcal{H}} f)||_2+||\tilde N_\ast(\nabla_{||}\mathcal{S}_\lambda^{\mathcal{H}^\ast} f)||_2\leq c||f||_2,\notag\\
(iv)&&||\tilde N_\ast(H_tD_{1/2}^t\mathcal{S}_\lambda^{\mathcal{H}} f)||_2+||\tilde N_\ast(H_tD_{1/2}^t\mathcal{S}_\lambda^{\mathcal{H}^\ast} f)||_2\leq  c||f||_2,
\end{eqnarray}
whenever  $f\in L^2(\mathbb R^{n+1},\mathbb C)$. \end{theorem}

 \begin{theorem}\label{th2} Consider $\mathcal{H}=\partial_t-\div A\nabla$. Assume that $\mathcal{H}$ satisfies \eqref{eq3}-\eqref{eq4}. Assume in addition that
 $A$ is real and symmetric. Then there exists a constant $C $, depending at most
     on $n$, $\Lambda$, such that  \eqref{keyestint-} holds with this $C $. In particular, the estimates in \eqref{keyestint+a} all hold, with constants depending only
     on $n$, $\Lambda$, $C $, in the case when
     $A$ is real, symmetric and satisfies \eqref{eq3}-\eqref{eq4}.
     \end{theorem}

     \begin{theorem}\label{parabolicm} Assume that $\mathcal{H}=\partial_t-\div A\nabla$ satisfies \eqref{eq3}-\eqref{eq4}.
Suppose in addition that
 $A$ is real and symmetric. Then the parabolic measure associated to $\mathcal{H}$, in $\mathbb R^{n+2}_+$, is absolutely continuous with respect to the
 measure $dxdt$ on $\mathbb R^{n+1}=\partial \mathbb R^{n+2}_+$. Moreover, let $Q\subset\mathbb R^{n+1}$ be a parabolic cube and let $K(A_Q,y,s)$ be the to
  $\mathcal{H}$ associated Poisson kernel at $A_Q:=(x_Q,l(Q),t_Q)$ where $(x_Q,t_Q)$ is the center of the cube $Q$ and $l(Q)$ defines its size. Then there exists $c\geq 1$, depending only on $n$ and $\Lambda$, such that
  $$\int_{Q}|K(A_Q,y,s)|^{2}\, dyds\leq c|Q|^{-1}.$$
\end{theorem}

  \begin{remark}\label{fut2}  Note that \eqref{keyestint-} $(i)$ is a uniform (in $\lambda$) $L^2$-estimate involving the first order partial derivative, in the $\lambda$-coordinate, of single layer potentials, while  \eqref{keyestint-} $(ii)$  is a square function estimate involving the second order partial derivatives, in the $\lambda$-coordinate, of single layer potentials. A relevant question is naturally in what generality the estimates in \eqref{keyestint-} can be expected to hold. In \cite{N1} it is proved, under additional assumptions, that these estimates are stable under small complex perturbations of the coefficient matrix. However, in the elliptic case and after \cite{AAAHK} appeared, it was proved in \cite{R}, see \cite{GH} for an alternative proof, that if $-\mbox{div}\, A(X)\nabla$ satisfies the basic assumptions imposed in
     \cite{AAAHK}, then the elliptic version of \eqref{keyestint-} $(ii)$ always holds. In fact, the approach in \cite{R}, which is based on functional calculus, even dispenses of the De Giorgi-Moser-Nash estimates underlying \cite{AAAHK}. Furthermore, in the elliptic case \eqref{keyestint-} $(ii)$ can be seen to imply \eqref{keyestint-} $(i)$ by the results of \cite{AA}. Hence, in the elliptic case, and under the assumptions of \cite{AAAHK},  the elliptic version of \eqref{keyestint-} always holds. Based on this it is fair to pose the question whether or not a similar line of development can be anticipated in the parabolic case. Based on \cite{N}, this paper and \cite{N1}, we anticipated that a parabolic version of \cite{GH} can be developed, To develop a parabolic version of \cite{AA} is a very interesting and potentially challenging project.
 \end{remark}

 Theorem \ref {parabolicm} is used in the proof of Theorem \ref{th2} and to our knowledge Theorem \ref{th0}, Theorem \ref{th2} and Theorem \ref {parabolicm} are all new. To put these results in the context of the current literature devoted to parabolic layer potentials and parabolic singular integrals, in $C^1$- regular or Lipschitz regular cylinders, it is fair to first mention \cite{FR}, \cite{FR1}, \cite{FR2} where a theory of singular integral operators with mixed homogeneity was developed and Theorem \ref{th0} $(i)-(iv)$ were proved in the context of the heat operator and in the context of time-independent $C^1$-cylinders. These results were then extended in \cite{B}, \cite{B1}, still in the context of the heat operator, to the setting of time-independent Lipschitz domains. The more challenging setting of  time-dependent Lipschitz type domains was considered in \cite{LM}, \cite{HL}, \cite{H}, see also \cite{HL1}. In particular, in these papers the correct notion of time-dependent Lipschitz type domains, from the perspective of parabolic singular integral operators and parabolic  layer potentials,  was found. One  major contribution of these papers, see \cite{HL}, \cite{H} and \cite{HL1} in particular, is the proof of Theorem \ref{th0} in the  context of the heat operator in time-dependent Lipschitz type domains. Beyond these results the literature only contains modest contributions to the study of
     parabolic layer potentials associated to second order parabolic operators (in divergence form) with  variable, bounded, measurable, uniformly elliptic (and complex) coefficients. Based on this we believe that our results will pave the way for important developments in the area of parabolic PDEs.

      While Theorem \ref{th0} and Theorem \ref{th2} coincide, in the stationary case, with the set up  and the corresponding results established in   \cite{AAAHK} for  elliptic equations, we claim that our results, Theorem \ref{th0} in particular, are not, for at least two reasons, straightforward generalizations of the corresponding results in \cite{AAAHK}. First, our result rely on \cite{N} where certain square function estimates are established for second order parabolic operators of the form $\mathcal{H}$, and where, in particular,  a parabolic version of the technology in \cite{AHLMcT} is developed. Second, in general the presence of the (first order) time-derivative forces one to consider fractional time-derivatives leading, as in \cite{LM}, \cite{HL}, \cite{H}, see also \cite{HL1}, to rather elaborate additional estimates. Theorem \ref{parabolicm} gives a parabolic version of an elliptic result due to Jerison and Kenig \cite{JK} and a version of the main result in \cite{FS} for equations of the form \eqref{eq1}, assuming in addition that $A$ is real and symmetric.

     \subsection{Proofs and organization of the paper} In general we will only supply the proof of our statements for $\mathcal{S}_\lambda:=\mathcal{S}_\lambda^{\mathcal{H}}$. The corresponding results  for $\mathcal{S}_\lambda^\ast:=\mathcal{S}_\lambda^{\mathcal{H}^\ast}$ then follow readily by analogy. In Section \ref{sec2}, which is of preliminary nature, we introduce notation, weak solutions, state the De Giorgi-Moser-Nash estimates referred to in Theorem \ref{th0}, we prove energy estimates, and we state/prove a few fact from Littlewood-Paley theory. In Section \ref{sec3} we prove a set of important preliminary  estimates related to the boundedness of single layer potentials: off-diagonal estimates and  uniform  (in $\lambda$) $L^2$-estimates. Section \ref{sec4}  is devoted to the proof of two important lemmas: Lemma \ref{lemsl1++} and Lemma \ref{lemsl1}. To briefly describe these results  we introduce $\Phi(f)$ where
     \begin{eqnarray}\label{keyestint-ex+}
\Phi(f):=\sup_{\lambda> 0}||\partial_\lambda \mathcal{S}_\lambda f||_2+|||\lambda \partial_\lambda^2\mathcal{S}_{\lambda}f|||.
     \end{eqnarray}
     Lemma \ref{lemsl1++} concerns estimates of non-tangential maximal functions and in this lemma we establish bounds of $||N_\ast(\partial_\lambda \mathcal{S}_\lambda f)||_2$, $||\tilde N_\ast(\nabla_{||}\mathcal{S}_\lambda f)||_2$ and $||\tilde N_\ast(H_tD_{1/2}^t\mathcal{S}_\lambda f)||_2$ in terms of a constant times
     $$\Phi(f)+||f||_2+\sup_{\lambda>0}||\mathbb D \mathcal{S}_\lambda f||_2.$$
     In Lemma \ref{lemsl1} we establish square function estimates of the form,
      \begin{eqnarray*}
              (i)&&|||\lambda^{m+2l+4}\nabla\partial_\lambda\partial_t^{l+1}\partial_\lambda^{m+1}\mathcal{S}_{\lambda}f|||\leq c(\Phi(f)+||f||_2),\notag\\
     (ii)&&|||\lambda^{m+2l+4}\partial_t\partial_t^{l+1}\partial_\lambda^{m+1}\mathcal{S}_{\lambda}f|||\leq c(\Phi(f)+||f||_2),
     \end{eqnarray*}
whenever $f\in L^2(\mathbb R^{n+1},\mathbb C)$, and for $m\geq -1$, $l\geq -1$. Using Lemma \ref{lemsl1++}, the proof of Theorem \ref{th0} boils down to proving the estimate
     \begin{eqnarray}\label{keyestint-ex+eed}
\sup_{\lambda>0}||\mathbb D \mathcal{S}_\lambda f||_2\leq c(\Phi(f)+||f||_2).
     \end{eqnarray}
     The estimate in \eqref{keyestint-ex+eed}, which is rather demanding, uses Lemma \ref{lemsl1} and make extensive use of recent results concerning resolvents, square functions and Carleson measures, established in \cite{N}. In Section \ref{sec5} we collect the material from \cite{N} needed in the proof of \eqref{keyestint-ex+eed}. In \cite{N} a parabolic version of the main and hard estimate in \cite{AHLMcT} is established. In the final subsection of Section \ref{sec5}, Section \ref{kato}, we also seize the opportunity to clarify some statements made in \cite{N} concerning the Kato square root problem for parabolic operators. The conclusion is that in \cite{N} the Kato square root problem for parabolic operators is solved for  for the first time in the literature.   In
     Section \ref{sec6} we prove \eqref{keyestint-ex+eed} as a consequence of Lemma \ref{lemsl1c}, Lemma \ref{lemsl1+}, and Lemma \ref{lemsl1+k} stated below. For clarity, the final proof of Theorem \ref{th0}, based on the estimates established in the previous sections, is summarized in Section \ref{sec7}. In Section \ref{sec8} we prove Theorem \ref{th2} by first establishing a local parabolic Tb-theorem for square functions, see Theorem \ref{ltb}, and then by establishing Theorem \ref{parabolicm}. We believe that our proof of Theorem \ref{parabolicm} adds to the clarity of the corresponding argument in \cite{FS}.

    \section{Preliminaries}\label{sec2}

    Let
$x=(x_1,..,x_{n})$, $X=(x,x_{n+1})$, $(x,t)=(x_1,..,x_{n},t)$, $(X,t)=(x_1,..,x_{n}, x_{n+1},t)$. Given $(X,t)=(x,x_{n+1}, t)$, $r>0$, we let $Q_r(x,t)$ and
 $\tilde Q_r(X,t)$ denote, respectively, the  parabolic cubes in $\mathbb R^{n+1}$ and $\mathbb R^{n+2}$,  centered at $(x,t)$ and $(X,t)$, and of size $r$. By $Q$, $\tilde Q$ we denote any such parabolic cubes and we let $l(Q)$, $l(\tilde Q)$, $(x_Q,t_Q)$, $(X_{\tilde Q},t_{\tilde Q})$ denote their sizes and centers, respectively. Given $\gamma>0$, we let $\gamma Q$, $\gamma \tilde Q$ be the cubes which have the same centers as $Q$ and $\tilde Q$, respectively, but with sizes defined by $\gamma l(Q)$ and $\gamma l(\tilde Q)$.  Given a set $E\subset \mathbb R^{n+1}$ we let $|E|$ denote its Lebesgue measure and by
$1_E$ we denote the indicator function for $E$. Finally, by $||\cdot||_{L^2(E)}$ we mean $||\cdot 1_E||_2$. Furthermore, as mentioned and based on \eqref{eq4}, we will frequently also use a different convention concerning the labeling of the coordinates: we let $\lambda=x_{n+1}$ and when using the symbol
    $\lambda$, the point $(X,t)=(x,x_{n+1},t)$ will be written as $ (x,t, \lambda)=(x_1,..,x_{n},t,\lambda)$.  We write $\nabla =(\nabla_{||},\partial_\lambda)$ where $\nabla_{||}=(\partial_{x_1},...,\partial_{x_n})$. The notation $L^2(\mathbb R^{n+1},\mathbb C)$, $||\cdot||_2$, $\|(\cdot,\cdot)\|$, $\mathbb D$, $D_{1/2}^t$, $H_t$, was introduced in subsection \ref{nota} above. In the following we will, in addition
   to $\mathbb D$ and $D_{1/2}^t$, at instances also use  the parabolic half-order time derivative
\begin{eqnarray*}
	\widehat{\mathbb D_{n+1}f}(\xi,\tau):=\frac{\tau}{\|(\xi,\tau)\|}\hat f(\xi,\tau).
\end{eqnarray*}
We let
$\mathbb H:=\mathbb H(\mathbb R^{n+1},\mathbb C)$ be the closure of $C_0^\infty(\mathbb R^{n+1},\mathbb C)$ with respect to
\begin{eqnarray}\label{hsapace}
	\|f\|_{\mathbb H}:=\|\mathbb Df\|_2.
\end{eqnarray}
By applying
Plancherel's theorem  we have
\begin{eqnarray}\label{uau}
(i)&&\|f\|_{\mathbb H}\approx\|\nabla_{||}  f\|_2+\|H_tD_{1/2}^tf\|_2,\notag\\
(ii)&&\|\mathbb D_{n+1}f\|_2\leq c\|D^t_{1/2}f\|_2,\end{eqnarray}
with constants depending only on $n$. Furthermore, we let $\tilde{\mathbb H}:=\tilde{\mathbb H}(\mathbb R^{n+2},\mathbb C)$ be the closure of $C_0^\infty(\mathbb R^{n+2},\mathbb C)$
with respect to
\begin{eqnarray*}
\|F\|_{\tilde{\mathbb H}}:=\biggl (\int_{-\infty}^\infty\int_{\mathbb R^{n+1}}\biggl(|\partial_\lambda F|^2+|\mathbb DF|^2\biggr )\, dxdtd\lambda\biggr )^{1/2}.
\end{eqnarray*}
Similarly, we let $\tilde{\mathbb H}_+:=\tilde{\mathbb H}_+(\mathbb R^{n+2}_+,\mathbb C)$ be the closure of $C_0^\infty(\mathbb R^{n+2}_+,\mathbb C)$ with respect to the expression in the last display but with integration over the interval $(-\infty,\infty)$ replaced by integration over the interval $(0,\infty)$.

\subsection{Weak solutions} Let $\Omega\subset\{X=(x,x_{n+1})\in\mathbb R^n\times\mathbb R_+\}$ be a domain and let, given $-\infty<t_1< t_2<\infty$,
$\Omega_{t_1,t_2}=\Omega\times (t_1,t_2)$. We let $W^{1,2}(\Omega,\mathbb C)$ be the  Sobolev space of complex valued functions $v$, defined on $\Omega$, such that $v$ and $\nabla v$ are in $L^{2}(\Omega,\mathbb C)$. $L^2(t_1,t_2,W^{1,2}(\Omega,\mathbb C))$ is the space of  functions $u:\Omega_{t_1,t_2}\to \mathbb C$ such that
$$||u||_{L^2(t_1,t_2,W^{1,2}(\Omega,\mathbb C))}:=\biggl (\int_{t_1}^{t_2}||u(\cdot,t)||_{W^{1,2}(\Omega,\mathbb C)}^2\, dt\biggr )^{1/2}<\infty.$$
We say that $u\in L^2(t_1,t_2,W^{1,2}(\Omega,\mathbb C))$  is a weak solution to the equation
\begin{eqnarray}\label{ggag}
\mathcal{H}u=(\partial_t+\mathcal{L})u=0,\end{eqnarray}
in $\Omega_{t_1,t_2}$, if \begin{equation}\label{weak}
\int_{\mathbb R^{n+2}_+} \biggl( A\nabla u\cdot\nabla\bar \phi-u \partial_t\bar\phi\biggr )\, dXdt=0,
\end{equation}
whenever $\phi \in C_0^{\infty} (\Omega_{t_1,t_2},\mathbb C)$. Similarly, we say that $u$ is a weak solution to \eqref{ggag} in $\mathbb R^{n+2}_+$ if
$u\phi\in L^2(-\infty,\infty,W^{1,2}(\mathbb R^n\times\mathbb R_+,\mathbb C))$ whenever $\phi \in C_0^{\infty} (\mathbb R^{n+2}_+,\mathbb C)$ and if
\eqref{weak} holds whenever $\phi \in C_0^{\infty} (\mathbb R^{n+2}_+,\mathbb C)$. Assuming  that $\mathcal{H}$  satisfies \eqref{eq3}-\eqref{eq4} as well as the De Giorgi-Moser-Nash estimates stated in \eqref{eq14+}-\eqref{eq14++} below, it follows that any weak solution is smooth as a function of $t$ and in this case
     \begin{equation*}
\int_{\mathbb R^{n+2}_+} \biggl( A\nabla u\cdot\nabla\bar \phi+\partial_tu \bar\phi\biggr )\, dXdt=0,
\end{equation*}
holds whenever $\phi \in C_0^{\infty} (\Omega_{t_1,t_2},\mathbb C)$. Furthermore, if $u$ is globally defined in $\mathbb R^{n+2}_+$, and if
$D_{1/2}^tu\overline{H_tD_{1/2}^t\phi}$ is integrable in $\mathbb R^{n+2}_+$, whenever $\phi\in C_0^\infty(\mathbb R^{n+2}_+,\mathbb C)$, then
\begin{eqnarray}\label{eq4-ed}
B_+(u,\phi)=0  \mbox{ whenever } \phi\in C_0^\infty(\mathbb R^{n+2}_+,\mathbb C),
    \end{eqnarray}
where the sesquilinear form $B_+(\cdot,\cdot)$ is defined on $ \tilde{\mathbb H}_+\times  \tilde{\mathbb H}_+$ as
       \begin{eqnarray*}
  B_+(u,\phi):= \int_0^\infty\int_{\mathbb R^{n+1}}
      \biggl(A\nabla u\cdot\nabla\bar \phi-D_{1/2}^tu\overline{H_tD_{1/2}^t\phi}\biggr)\, dxdtd\lambda.
      \end{eqnarray*}
      In particular, whenever $u$ is a weak solution to \eqref{ggag} in $\mathbb R^{n+2}_+$ such that $u\in \tilde{\mathbb H}_+$, then
      \eqref{eq4-ed} holds. From now on, whenever we write that $\mathcal{H}u=0$ in a bounded domain $\Omega_{t_1,t_2}$, then we mean that
      \eqref{weak} holds whenever $\phi \in C_0^{\infty} (\Omega_{t_1,t_2},\mathbb C)$, and when we write that $\mathcal{H}u=0$ in $\mathbb R^{n+2}_+$, then we mean that \eqref{weak} holds whenever $\phi \in C_0^{\infty} (\mathbb R^{n+2}_+,\mathbb C)$.

      \subsection{De Giorgi-Moser-Nash estimates}  We say  that solutions to
     $\mathcal{H}u=0$ satisfy De Giorgi--Moser-Nash estimates if there exist, for each $1\leq p<\infty$ fixed, constants $c$ and $\alpha\in (0,1)$ such that the following is true. Let $\tilde Q\subset\mathbb R^{n+2}$ be a parabolic cube and assume that
$\mathcal{H}u=0$ in $2\tilde Q$. Then
                       \begin{eqnarray}\label{eq14+}
                \sup_{ \tilde Q}|u|\leq c\biggl (\mean{2\tilde Q}|u|^p\biggr )^{1/p},
    \end{eqnarray}
         and
        \begin{eqnarray}\label{eq14++}
                &&|u(X,t)-u(\tilde X,\tilde t)|\leq c\biggl (\frac {||(X-\tilde X,t-\tilde t)||}r\biggr )^{\alpha}
                \biggl (\mean{2\tilde Q}|u|^p\biggr )^{1/p},
    \end{eqnarray}
    whenever $(X,t)$, $(\tilde X,\tilde t)\in \tilde Q$, $r:=l(\tilde Q)$. The constant $c$ and $\alpha$ will be referred to as the
    De Giorgi-Moser-Nash constants. It is well known that if \eqref{eq14+}-\eqref{eq14++} hold for one $p$, $1\leq p<\infty$, then these estimates hold for all $p$ in this range.

\subsection{Energy estimates}

     \begin{lemma}\label{le1--} Assume that $\mathcal{H}$ satisfies \eqref{eq3}-\eqref{eq4}. Let $\tilde Q\subset\mathbb R^{n+2}$ be a parabolic cube and let $\beta>1$ be a fixed constant. Assume that
$\mathcal{H}u=0$ in $\beta\tilde Q$.  Let $\phi\in C_0^\infty(\beta\tilde Q)$ be   a cut-off function for $\tilde Q$ such that $0\leq \phi\leq 1$, $\phi=1$ on $\tilde Q$. Then there exists a constant
     $c=c(n,\Lambda,\beta)$, $1\leq c<\infty$, such that
     \begin{eqnarray*}
     \int|\nabla u(X,t)|^2(\phi(X,t))^2\, dXdt\leq  c\int|u(X,t)|^2(|\nabla \phi(X,t)|^2+\phi(X,t)|\partial_t\phi(X,t)|)\, dXdt.
\end{eqnarray*}
\end{lemma}
\begin{proof} The lemma is a standard energy estimate. Indeed,
 \begin{equation*}
\int \biggl( A\nabla u\cdot\nabla(\bar u\phi^2)-u \partial_t(\bar u\phi^2)\biggr )\, dXdt=0,
\end{equation*}
by the definition of weak solutions. Hence,
 \begin{equation*}
\int |\nabla u|^2\phi^2\, dxdt\leq c\int|u|^2(|\nabla \phi|^2+\phi|\partial_t\phi|)\, dXdt.
\end{equation*}
\end{proof}

\begin{lemma}\label{le1} Assume that $\mathcal{H}$ satisfies \eqref{eq3}-\eqref{eq4}.  Let $Q\subset\mathbb R^{n+1}$ be a parabolic cube, $\lambda_0\in \mathbb R$, and let $\beta_1>1$, $\beta_2\in (0,1]$ be fixed constants. Let $I=(\lambda_0-\beta_2l(Q),\lambda_0+\beta_2l(Q))$, $\gamma I=
(\lambda_0-\gamma \beta_2l(Q),\lambda_0+\gamma \beta_2l(Q))$ for $\gamma\in (0,1)$. Assume that $\mathcal{H}u=0$ in $\beta_1^2Q\times I$. Then there exists a constant
     $c=c(n,\Lambda,\beta_1,\beta_2)$, $1\leq c<\infty$, such that
\begin{eqnarray*}
(i)&&\mean{Q}|\nabla u(x,t,\lambda_0)|^2\, dxdt\leq c\mean{\beta_1Q\times\frac 1 4I}|\nabla
u(X,t)|^2\, dXdt,\notag\\
(ii)&&\mean{Q}|\nabla u(x,t,\lambda_0)|^2\, dxdt\leq \frac c{l(Q)^2}\mean{\beta_1^2Q\times\frac 1 2I}|u(X,t)|^2\,
dXdt.
\end{eqnarray*}
\end{lemma}
\begin{proof} It suffices to prove the lemma with $\beta_1=2$, $\beta_2=1$. Furthermore, we only prove  $(i)$ as  $(ii)$ follows from  $(i)$ and  Lemma \ref{le1--}. For $\lambda_0\in \mathbb R$ fixed, and with $\gamma I$ as above, we let
\begin{eqnarray*}
J_1&:=& \biggl (\mean{Q}\biggl |\nabla u(x,t,\lambda_0)-\mean{\frac 1 {16}I}\nabla u(x,t,\lambda)\, d\lambda\biggr |^2dxdt\biggr )^{1/2},\notag\\
J_2&:=&\biggl (\mean{Q}\biggl |\mean{\frac 1 {16}I}\nabla u(x,t,\lambda)\, d\lambda\biggr |^2dxdt\biggr )^{1/2}.
\end{eqnarray*}
Then
\begin{eqnarray*}
\biggl (\mean{Q}|\nabla u(x,t,\lambda_0)|^2\, dxdt\biggr )^{1/2}\leq J_1+J_2.
\end{eqnarray*}
Using the H{\"o}lder inequality
\begin{eqnarray*}
J_2\leq c\biggl (\mean{2Q\times \frac 1 8I}|\nabla
u(X,t)|^2\, dXdt\biggr )^{1/2}.
\end{eqnarray*}
Using the fundamental theorem of calculus and the H{\"o}lder inequality,
\begin{eqnarray*}
J_1\leq cl(Q)\biggl (\mean{Q\times \frac 1 {16}I}|\nabla \partial_\lambda u(X,t)|^2\, dXdt\biggr )^{1/2}.
\end{eqnarray*}
Using that $\partial_\lambda u$ is a solution to the same equation as $u$ it follows from Lemma \ref{le1--}  that
\begin{eqnarray*}
J_1\leq c\biggl (\mean{\frac 3 2Q\times \frac 1 8I}|\partial_\lambda
u(X,t)|^2\, dXdt\biggr )^{1/2}.
\end{eqnarray*}
Hence the estimate in $(i)$ follows. \end{proof}

\begin{lemma}\label{le1a} Assume that $\mathcal{H}$ satisfies \eqref{eq3}-\eqref{eq4}. Let $\tilde Q\subset\mathbb R^{n+2}$ be a parabolic cube and let $\beta>1$ be a fixed constant. Assume that
$\mathcal{H}u=0$ in $\beta\tilde Q$.  Then there exists a constant
     $c=c(n,\Lambda, \beta)$, $1\leq c<\infty$, such that
\begin{eqnarray*}
\mean{\tilde Q}|\partial_t u(X,t)|^2\, dXdt\leq \frac {c}{l(\tilde Q)^4}\mean{\beta\tilde Q}|u(X,t)|^2\, dXdt.
\end{eqnarray*}
\end{lemma}
\begin{proof} Let $\phi\in C_0^\infty(\beta\tilde Q)$ be   a cut-off function for $\tilde Q$ such that $0\leq \phi\leq 1$, $\phi=1$ on $\tilde Q$, $|\nabla\phi|\leq c/l(\tilde Q)$, $|\partial_t\phi|\leq c/l(\tilde Q)^2$. Let
\begin{eqnarray*}
J_1:=\int|\partial_t u|^2\phi^4\, dXdt,
\end{eqnarray*}
and
\begin{eqnarray*}
J_2:=\int|\nabla u|^2\phi^2\, dXdt,\ J_3:=\int|\nabla \partial_t u|^2\phi^6\, dXdt.
\end{eqnarray*}
As $\partial_t u$ is a solution to the same equation as $u$,
 \begin{equation*}
\int \biggl( A\nabla \partial_tu\cdot\nabla(\bar u\phi^4)-\partial_t u \partial_t(\bar u\phi^4)\biggr )\, dXdt=0.
\end{equation*}
Hence,
\begin{eqnarray*}
J_1=\int \biggl( (A\nabla\partial_t u\cdot\nabla \bar u) \phi^4+4(A\nabla \partial_tu\cdot\nabla\phi)\bar u\phi^3-4(\partial_t u \partial_t\phi)\bar u\phi^3\biggr )\, dXdt,
\end{eqnarray*}
and
\begin{eqnarray*}
J_1\leq l(\tilde Q)^2\epsilon J_3+\frac {c(\epsilon)}{l(\tilde Q)^2}J_2+ \frac {c(\epsilon)}{l(\tilde Q)^4}\mean{\beta\tilde Q}|u(X,t)|^2\, dXdt
\end{eqnarray*}
where $\epsilon$ is a degree of freedom. Again using that $\partial_t u$ is a solution to the same equation as $u$, and essentially  Lemma \ref{le1--}, we see that
\begin{eqnarray*}
J_3\leq  c\int|\partial_t u|^2\phi^4(|\nabla \phi|^2+|\partial_t\phi|)\, dXdt\leq \frac {c}{l(\tilde Q)^2}J_1.
\end{eqnarray*}
Combining the above estimates, and again using Lemma \ref{le1--},  the lemma follows.
\end{proof}

\subsection{Littlewood-Paley theory}\label{littleth}
We define a parabolic approximation of the identity, which will be fixed throughout the paper, as follows. Let
$\P\in C_0^\infty(Q_1(0))$, $\P\geq 0$ be real-valued,  $\int \P\, dxdt=1$, where $Q_1(0)$ is the unit parabolic cube in $\mathbb R^{n+1}$ centered at $0$. At instances we will also assume that $\int x_i\P(x,t)\, dxdt=0$ for all $i\in \{1,..,n\}$. We set $\P_\lambda(x,t)=\lambda^{-n-2}\P(\lambda^{-1}x,\lambda^{-2}t)$ whenever $\lambda>0$. We let
$\P_\lambda$ denote the convolution operator
$$\P_\lambda f(x,t)=\int_{\mathbb R^{n+1}}\P_\lambda(x-y,t-s)f(y,s)\, dyds.$$
Similarly, by $\mathcal{Q}_\lambda$ we denote a generic approximation to the zero operator, not necessarily the same at each instance, but chosen from a finite set of such
operators depending only on our original choice of $\P_\lambda$. In particular, $\mathcal{Q}_\lambda(x,t)=\lambda^{-n-2}\mathcal{Q}(\lambda^{-1}x,\lambda^{-2}t)$ where
$\mathcal{Q}\in C_0^\infty(Q_1(0))$, $\int \mathcal{Q}\, dxdt=0$. In addition we will, following \cite{HL}, assume that $\mathcal{Q}_\lambda$ satisfies the conditions
\begin{eqnarray*} 
|\mathcal{Q}_\lambda(x,t)|&\leq&\frac {c\lambda}{(\lambda+||(x,t)||)^{n+3}},\notag\\
|\mathcal{Q}_\lambda(x,t)-\mathcal{Q}_\lambda(y,s)|&\leq&\frac {c||(x-y,t-s)||^\alpha}{(\lambda+||(x,t)||)^{n+2+\alpha}},
\end{eqnarray*}
where the latter estimate holds for some $\alpha\in (0,1)$ whenever $2||(x-y,t-s)||\leq ||(x,t)||$. Under these assumptions it is well known that
\begin{eqnarray}\label{li}
	\int_0^\infty\int_{\mathbb R^{n+1}}|\mathcal{Q}_\lambda f|^2\, \frac{dxdtd\lambda}\lambda
		\leq c\int_{\mathbb R^{n+1}}|f|^2\, {dxdt},
\end{eqnarray}
for all $f\in L^2(\mathbb R^{n+1},\mathbb C)$. In the following we collect a number of elementary observations used in the forthcoming sections.
\begin{lemma} \label{little2} Let $\P_\lambda$ be as above. Then
\begin{eqnarray*}
|||\lambda\nabla \P_\lambda f|||+|||\lambda^2\partial_t\P_\lambda f|||+|||\lambda\mathbb D \P_\lambda f|||\leq c||f||_2,
\end{eqnarray*}
for all $f\in L^2(\mathbb R^{n+1},\mathbb C)$.
\end{lemma}
\begin{proof} This lemma essentially follows immediately from \eqref{li}. For slightly more details we refer to the proof of Lemma 2.30 in \cite{N}.
\end{proof}

Consider a cube $Q\subset\mathbb R^{n+1}$. In the following we let $\mathcal{A}_\lambda^Q$ denote the dyadic averaging operator induced by $Q$, i.e., if $\hat Q_\lambda(x,t)$ is the minimal dyadic cube
      (with respect to the grid induced by $Q$) containing $(x,t)$, with side length at least $\lambda$, then
      \begin{eqnarray}\label{dy}\mathcal{A}_\lambda^Q f(x,t):=\mean{\hat Q_\lambda(x,t)}f\, dyds,
      \end{eqnarray}
      is the average of $f$ over $\hat Q_\lambda(x,t)$.

      \begin{lemma} \label{little3} Let $\mathcal{A}_\lambda^Q$ and $\P_\lambda$ be as above. Then
\begin{eqnarray*}
\int_0^\infty\int_{\mathbb R^{n+1}}|(\mathcal{A}_\lambda^Q-\P_\lambda)f|^2\, \frac{dxdtd\lambda}\lambda\leq c\int_{\mathbb R^{n+1}}|f|^2\, {dxdt},
\end{eqnarray*}
for all $f\in L^2(\mathbb R^{n+1},\mathbb C)$.
\end{lemma}
\begin{proof} The lemma follows by  orthogonality estimates and we here include a sketch of the proof for completion. Let $F\in C_0^\infty(\mathbb R^{n+2}_+,\mathbb C)$ be such that $|||F|||=1$. It suffices to prove that
\begin{eqnarray*}
\int_0^\infty\int_{\mathbb R^{n+1}}F(x,t,\lambda)\overline{(\mathcal{A}_\lambda^Q-\P_\lambda)f(x,t)}\, \frac{dxdtd\lambda}\lambda\leq c||f||_2,
\end{eqnarray*}
for all $f\in L^2(\mathbb R^{n+1},\mathbb C)$. To prove this we first note that $|(\mathcal{A}_\lambda^Q-\P_\lambda)f(x_0,t_0)|\leq cM(f)(x_0,t_0)$ whenever $(x_0,t_0)\in\mathbb R^{n+1}$ and where
$M$ is the parabolic Hardy-Littlewood maximal function. Hence,
      \begin{eqnarray*}
      \sup_{\lambda>0}||(\mathcal{A}_\lambda^Q-\P_\lambda)||_{2\to 2}\leq c.
      \end{eqnarray*}
       Let  $\mathcal{Q}_\lambda$ be an approximation of the zero operator defined based on a function $\mathcal{Q}$ so normalized that $\mathcal{Q}_\lambda$  is a resolution of the identity, i.e.,
       $$\int_0^\infty\mathcal{Q}_\lambda^2g\, \frac {d\lambda}{\lambda}=g,$$
       whenever $g\in C_0^\infty(\mathbb R^{n+1},\mathbb C)$. Then
       \begin{eqnarray}\label{a1-sea}||(\mathcal{A}_\lambda^Q-\P_\lambda)\mathcal{Q}_\sigma||_{2\to 2}\leq c\min\{(\lambda/\sigma)^\delta,(\sigma/\lambda)^\delta\},\end{eqnarray}
       for some $\delta>0$. Indeed, let $\mathcal{R}_\lambda(x,t,y,s)$ be the kernel associated to $\mathcal{A}_\lambda^Q-\P_\lambda$, i.e.,
       $$\mathcal{R}_\lambda(x,t,y,s)=\frac 1{|\hat Q_\lambda(x,t)|}1_{\hat Q_\lambda(x,t)}(y,s)-\P_\lambda(x-y,t-s).$$
       Then $\mathcal{R}_\lambda 1=0$ and it is easily seen that
       \begin{eqnarray*}
       (i)&&|\mathcal{R}_\lambda(x,t,y,s)|\leq \lambda^\delta (\lambda+||(x,t)||)^{-n-2-\delta},\notag\\
       (ii)&& \int_{\mathbb R^{n+1}}\sup_{\{(z,w):\ ||(z-y,w-s)||\leq\sigma\}} |\mathcal{R}_\lambda(x,t,z,w)-\mathcal{R}_\lambda(x,t,y,s)|\, dyds\leq c(\sigma/\lambda)^\delta,
       \end{eqnarray*}
       whenever $(x,t)\in\mathbb R^{n+1}$, $0<\sigma\leq\lambda<\infty$ and with $\delta=1$. Note that there is an unfortunate statement in
       the corresponding proof in \cite{N}: there $(ii)$ was stated in a pointwise sense which can, obviously, not hold as the indicator function $1_{\hat Q_\lambda(x,t)}$ is not H{\"o}lder continuous. Using $(i)$, $(ii)$, one can, arguing as in the proof of display (3.7) and Remark 3.11 in \cite{HMc}, conclude the validity of  \eqref{a1-sea}. Let
       $h_\delta(\lambda,\sigma):=\min\{(\lambda/\sigma)^\delta,(\sigma/\lambda)^\delta\}.$ We write
       \begin{eqnarray*}
	&&\biggl |\int_0^\infty\int_{\mathbb R^{n+1}}F(x,t,\lambda)\overline{(\mathcal{A}_\lambda^Q-\P_\lambda)f(x,t)}\, \frac{dxdtd\lambda}\lambda\biggr |\notag\\
	&& \qquad = \
\biggl |\int_0^\infty\int_0^\infty\int_{\mathbb R^{n+1}}F(x,t,\lambda)\overline{(\mathcal{A}_\lambda^Q-\P_\lambda)\mathcal{Q}_\sigma^2f}(x,t)\, dxdt\frac{d\lambda}\lambda \frac {d\sigma}{\sigma}\biggr |,
\end{eqnarray*}
Hence, using Cauchy-Schwarz we see that
       \begin{eqnarray*}
	\biggl |\int_0^\infty\int_{\mathbb R^{n+1}}F(x,t,\lambda)\overline{(\mathcal{A}_\lambda^Q-\P_\lambda)f(x,t)}\, \frac{dxdtd\lambda}\lambda\biggr |\leq I_1^{1/2}
I_2^{1/2},
\end{eqnarray*}
where
       \begin{eqnarray*}
I_1&:=&\int_0^\infty\int_0^\infty\int_{\mathbb R^{n+1}}|F(x,t,\lambda)|^2h_\delta(\lambda,\sigma)\, dxdt\frac{d\lambda}\lambda \frac {d\sigma}{\sigma},\notag\\
I_2&:=&\int_0^\infty\int_0^\infty\int_{\mathbb R^{n+1}}|{(\mathcal{A}_\lambda^Q-\P_\lambda)\mathcal{Q}_\sigma^2f}(x,t)|^2(h_\delta(\lambda,\sigma))^{-1}\, dxdt\frac{d\lambda}\lambda \frac {d\sigma}{\sigma}.
\end{eqnarray*}
Integrating with respect to $\sigma$ in $I_1$ we see that $I_1\leq c$. Furthermore, using
 \eqref{a1-sea} we see that
        \begin{eqnarray*}
I_2&\leq& \int_0^\infty\int_0^\infty\int_{\mathbb R^{n+1}}|\mathcal{Q}_\sigma f(x,t)|^2h_\delta(\lambda,\sigma)\, dxdt\frac{d\lambda}\lambda \frac{d\sigma}{\sigma}\notag\\
&\leq& c\int_0^\infty\int_{\mathbb R^{n+1}}|\mathcal{Q}_\sigma f(x,t)|^2\, dxdt \frac {d\sigma}{\sigma}\leq c||f||_2^2.
\end{eqnarray*}
This completes the proof of the lemma. See also the proof of Lemma 4.3 in \cite{HMc}.
\end{proof}

\section{Off-diagonal and uniform $L^2$-estimates for single layer potentials}\label{sec3}

We here establish  a number of elementary and preliminary estimates for single layer potentials. We will consistently only formulate and prove results for
 $\mathcal{S}_\lambda:=\mathcal{S}_\lambda^{\mathcal{H}}$ and for $\lambda>0$, where $\mathcal{H}=\partial_t-\div A\nabla$ is assumed to satisfy \eqref{eq3}-\eqref{eq4} as well as \eqref{eq14+}-\eqref{eq14++}. The corresponding results  for $\mathcal{S}_\lambda^\ast:=\mathcal{S}_\lambda^{\mathcal{H}^\ast}$  follow by analogy. Here we will also use the notation $\div_{||}=\nabla_{||}\cdot$, $D_i=\partial_{x_i}$ for
$i\in\{1,...,n+1\}$. We let
                \begin{eqnarray*}
     (\mathcal{S}_\lambda D_j)f(x,t)&:=&\int_{\mathbb R^{n+1}}\partial_{y_j}\Gamma_\lambda(x,t,y,s)f(y,s)\, dyds,\ 1\leq j\leq n,\notag\\
     (\mathcal{S}_\lambda D_{n+1})f(x,t)&:=&\int_{\mathbb R^{n+1}}\partial_{\sigma}\Gamma(x,t,\lambda,y,s,\sigma)|_{\sigma=0}f(y,s)\, dyds.
     \end{eqnarray*}
We set
              \begin{eqnarray*}
     (\mathcal{S}_\lambda\nabla)&:=&((\mathcal{S}_\lambda D_1),...,(\mathcal{S}_\lambda D_n),(\mathcal{S}_\lambda D_{n+1})),\notag\\
     (\mathcal{S}_\lambda\nabla\cdot){\bf f}&:=&\sum_{j=1}^{n+1}(\mathcal{S}_\lambda D_j)f_j,
     \end{eqnarray*}
     whenever ${\bf f}=(f_1,...,f_{n+1})$ and we note that
     \begin{eqnarray*}
     (\mathcal{S}_\lambda\nabla_{||})\cdot{\bf f}_{||}(x,t)=-\mathcal{S}_\lambda(\div_{||}{\bf f}_{||}),\quad
      (\mathcal{S}_\lambda D_{n+1})=-\partial_\lambda\mathcal{S}_\lambda,
     \end{eqnarray*}
     whenever ${\bf f}=({\bf f}_{||},f_{n+1}) \in C_0^\infty(\mathbb R^{n+1},\mathbb C^{n+1})$ as the fundamental solution is translation invariant in the $\lambda$-variable.  Given a function $f\in L^2(\mathbb R^{n+1},\mathbb C)$, and $h=(h_1,...,h_{n+1})\in \mathbb R^{n+1}$, we let $(\mathbb D^hf)(x,t)=f(x_1+h_1,...,x_n+h_n,t+h_{n+1})-f(x,t)$. Given $m\geq -1$, $l\geq -1$ we let
     \begin{eqnarray}\label{kernel}
      K_{m,\lambda}(x,t,y,s)&:=&\partial_\lambda^{m+1}\Gamma_\lambda(x,t,y,s),\notag\\
     K_{m,l,\lambda}(x,t,y,s)&:=&\partial_t^{l+1}\partial_\lambda^{m+1}\Gamma_\lambda(x,t,y,s),
     \end{eqnarray}
     and we introduce
     \begin{eqnarray*}
     d_\lambda(x,t,y,s):=|x-y|+|t-s|^{1/2} +\lambda.
     \end{eqnarray*}
\begin{lemma}\label{le2+} Consider  $m\geq -1$, $l\geq -1$. Then there exists constants $c_{m,l}$ and $\alpha\in (0,1)$, depending at most on $n$, $\Lambda$, the De Giorgi-Moser-Nash constants, $m$, $l$, such that
\begin{eqnarray*}
(i)&&|K_{m,l,\lambda}(x,t,y,s)|\leq c_{m,l}({d_\lambda(x,t,y,s)})^{-n-m-2l-4},\notag\\
(ii)&&|(\mathbb D^hK_{m,l,\lambda}(\cdot,\cdot,y,s))(x,t)|\leq c_{m,l}||h||^\alpha({d_\lambda(x,t,y,s)})^{-n-m-2l-4-\alpha},\notag\\
(iii)&&|(\mathbb D^hK_{m,l,\lambda}(x,t,\cdot,\cdot))(y,s)|\notag\leq c_{m,l}||h||^\alpha({d_\lambda(x,t,y,s)})^{-n-m-2l-4-\alpha},
\end{eqnarray*}
whenever $2||h||\leq ||(x-y,t-s)||$ or $2||h||\leq\lambda$.
\end{lemma}
\begin{proof} Assume first that $l=-1$. Then $K_{m,l,\lambda}=K_{m,\lambda}$. In the case $m=-1$ the estimates in $(i)-(iii)$ follow from  \eqref{eq14+} and
     \eqref{eq14++}, see also \cite{A} and Section 1.4 in \cite{AT}. In the cases $m\geq 0$, the corresponding estimates  follow by induction using \eqref{eq14+}, \eqref{eq14++}, Lemma \ref{le1--} and Lemma \ref{le1}. This establishes the estimates in $(i)-(iii)$
     for $K_{m,-1,\lambda}$ whenever $m\geq -1$.  We next consider the case of $K_{m,l,\lambda}$, $l\geq 0$. Fix $(y,s)\in\mathbb R^{n+1}$ and let $u=u(x,t,\lambda)=K_{m,l,\lambda}(x,t,y,s)$ for some $l\geq 0$. Given $(x,t,\lambda)\in \mathbb R_+^{n+2}$, let $\tilde Q\subset\mathbb R^{n+2}$ be the largest parabolic cube centered at $(x,t,\lambda)$  such that
     $16\tilde Q\subset\mathbb R^{n+2}_+$ and such $\mathcal{H}u=0$ in $16\tilde Q$. Then $l(\tilde Q)\approx \min\{\lambda,||(x-y,t-s)||\}$, and
\begin{eqnarray*}
|\partial_t u(x,t,\lambda)|\leq c\biggl (\mean{2\tilde Q}|\partial_t u|^2\, dXdt\biggr )^{1/2},
\end{eqnarray*}
by \eqref{eq14+} as $\partial_t u$ is a solution to the same equation as $u$. Using Lemma \ref{le1a} we can therefore conclude that
\begin{eqnarray*}
|\partial_t u(x,t,\lambda)|^2\leq \frac{c}{l(\tilde Q)^4}\biggl (\mean{8\tilde Q}|u|^2\, dXdt\biggr ).
\end{eqnarray*}
Using this and induction,  the estimate in $(i)$ follows for $K_{m,l,\lambda}(x,t,y,s)$ whenever $l\geq -1$. Using \eqref{eq14++}, the estimates in $(ii)$ and $(iii)$ are proved similarly.
\end{proof}

\begin{lemma}\label{le3} Consider  $m\geq -1$, $l\geq -1$. Then there exists a constant
     $c_{m,l}$, depending at most on $n$, $\Lambda$, the De Giorgi-Moser-Nash constants, $m$, $l$, such that the following holds whenever $Q\subset\mathbb R^{n+1}$ is a parabolic cube, $k\geq 1$ is an integer and $(x,t)\in Q$.
\begin{eqnarray*}
(i)&&\int_{2^{k+1}Q\setminus 2^kQ}|(2^kl(Q))^{m+2l+3}\nabla_yK_{m,l,\lambda}(x,t,y,s)|^2dy ds\leq c_{m,l}(2^kl(Q))^{-n-2},\notag\\
(ii)&&\int_{2Q}|(l(Q))^{m+2l+3}\nabla_yK_{m,l,\lambda}(x,t,y,s)|^2dy ds\leq c_{m,l,\rho}(l(Q))^{-n-2},\notag\\
&&\mbox{whenever $\rho>1$,  } l(Q)/\rho\leq\lambda\leq \rho l(Q).
\end{eqnarray*}
\end{lemma}
\begin{proof}  Fix $(x,t)\in Q$ and let
$$v(y,s,\lambda):=\overline{K_{m,l,\lambda}(x,t,y,s)}.$$
Then $v$ is a solution to the adjoint equation. The lemma now follows from Lemma \ref{le1} $(ii)$, applied to the adjoint equation, and Lemma \ref{le2+} $(i)$. Indeed, it is easy to see that Lemma \ref{le1} also is valid in when $Q$ is replaced by the annular region $2^{k+1}Q\setminus 2^kQ$.
\end{proof}

\begin{lemma}\label{le4} Consider  $m\geq -1$, $l\geq -1$. Then there exists a constant
     $c_{m,l}$, depending at most on $n$, $\Lambda$, the De Giorgi-Moser-Nash constants, $m$, $l$, such that the following holds whenever $Q\subset\mathbb R^{n+1}$ is a parabolic cube, $k\geq 1$ is an integer. Let ${\bf f}
     \in L^2(\mathbb R^{n+1},\mathbb C^{n})$, ${f}
     \in L^2(\mathbb R^{n+1},\mathbb C)$. Then
     \begin{eqnarray*}
(i)&&||\partial_t^{l+1}\partial_\lambda^{m+1}(\mathcal{S}_\lambda\nabla_{||}\cdot)({\bf f}1_{2^{k+1}Q\setminus 2^kQ})||_{L^2(Q)}^2 \leq \ c_{m,l}2^{-(n+2)k}(2^kl(Q))^{-2m-4l-6}||{\bf f}||^2_{L^2(2^{k+1}Q\setminus 2^kQ)},\notag\\
(ii)&&||\partial_t^{l+1}\partial_\lambda^{m+1}(\mathcal{S}_\lambda\nabla_{||}\cdot)({\bf f}1_{2Q})||_{L^2(Q)}^2
	\ \leq \ c_{m,l,\rho}(l(Q))^{-2m-4l-6}||{\bf f}||^2_{L^2(2Q)},\notag\\
&&\mbox{whenever $\rho>1$, $l(Q)/\rho\leq\lambda\leq \rho l(Q)$},\notag\\
(iii)&&||\partial_t^{l+1}\partial_\lambda^{m+1}(\mathcal{S}_\lambda)({f}1_{2^{k+1}Q\setminus 2^kQ})||_{L^2(Q)}^2\leq \ c_{m,l}2^{-(n+2)k}(2^kl(Q))^{-2m-4l-4}||{f}||^2_{L^2(2^{k+1}Q\setminus 2^kQ)},\notag\\
(iv)&&||\partial_t^{l+1}\partial_\lambda^{m+1}(\mathcal{S}_\lambda)({f}1_{2Q})||_{L^2(Q)}^2
\ \leq \
c_{m,l,\rho}(l(Q))^{-2m-4l-4}||{f}||^2_{L^2(2Q)},\notag\\
&&\mbox{whenever $\rho>1$, $l(Q)/\rho\leq\lambda\leq \rho l(Q)$}.
\end{eqnarray*}
\end{lemma}
\begin{proof} Let $(x,t)\in Q$. To prove $(i)$ we note that
 \begin{eqnarray*}
 &&|\partial_t^{l+1}\partial_\lambda^{m+1}(\mathcal{S}_\lambda\nabla_{||}\cdot)({\bf f}1_{2^{k+1}Q\setminus 2^kQ})(x,t)|^2\notag\\
 && \qquad = \ \biggl |\int_{2^{k+1}Q\setminus 2^kQ}\nabla_yK_{m,l,\lambda}(x,t,y,s)\cdot{\bf f}(y,s)\, dyds\biggr|^2\notag\\
 && \qquad \leq \ ||\nabla_yK_{m,l,\lambda}(x,t,y,s)||_{L^2(2^{k+1}Q
 \setminus 2^kQ)}^2||{\bf f}||^2_{L^2(2^{k+1}Q\setminus 2^kQ)}\notag\\
 && \qquad \leq \ c_{m,l}(2^kl(Q))^{-n-2m-4l-8}||{\bf f}||^2_{L^2(2^{k+1}Q\setminus 2^kQ)},
\end{eqnarray*}
by Lemma \ref{le3} $(i)$. Hence, integrating with respect to $(x,t)$ we see that
\begin{eqnarray*}
&&||\partial_t^{l+1}\partial_\lambda^{m+1}(\mathcal{S}_\lambda\nabla_{||}\cdot)({\bf f}1_{2^{k+1}Q\setminus 2^kQ})||_{L^2(Q)}^2\notag\\
&& \qquad \leq \ c_{m,l}(l(Q))^{n+2}(2^kl(Q))^{-n-2m-4l-8}||{\bf f}||^2_{L^2(2^{k+1}Q\setminus 2^kQ)}\notag\\
&&\qquad \leq \ c_{m,l}2^{-(n+2)k}(2^kl(Q))^{-2m-4l-6}||{\bf f}||^2_{L^2(2^{k+1}Q\setminus 2^kQ)}.
\end{eqnarray*}
This completes the proof of $(i)$. The proof of $(ii)$ is similar. To prove $(iii)$ we again consider $(x,t)\in Q$. Then
 \begin{eqnarray*}
 &&|\partial_t^{l+1}\partial_\lambda^{m+1}(\mathcal{S}_\lambda)({f}1_{2^{k+1}Q\setminus 2^kQ})(x,t)|^2\notag\\
 && \qquad = \ \biggl |\int_{2^{k+1}Q\setminus 2^kQ}K_{m,l,\lambda}(x,t,y,s){f}(y,s)\, dyds\biggr|^2\notag\\
 && \qquad \leq \ ||K_{m,l,\lambda}(x,t,y,s)||_{L^2(2^{k+1}Q
 \setminus 2^kQ)}^2||{f}||^2_{L^2(2^{k+1}Q\setminus 2^kQ)}\notag\\
 && \qquad \leq \ c_{m,l}(2^kl(Q))^{-n-2m-4l-6}||{f}||^2_{L^2(2^{k+1}Q\setminus 2^kQ)}.
\end{eqnarray*}
We can now proceed as above to complete the proof of $(iii)$. The proof of $(iv)$ is similar.
\end{proof}

\begin{lemma}\label{le5}  Assume $m\geq -1$, $l\geq -1$, $m+2l\geq -2$,  Then there exists a constant
     $c_{m,l}$, depending at most on $n$, $\Lambda$, the De Giorgi-Moser-Nash constants, $m$, $l$, such that the following holds. Let ${\bf f}
     \in L^2(\mathbb R^{n+1},\mathbb C^{n})$ and ${f}
     \in L^2(\mathbb R^{n+1},\mathbb C)$. Then
 \begin{eqnarray*}
(i)&&\sup_{\lambda>0}||\lambda^{m+2l+3}\partial_t^{l+1}\partial_\lambda^{m+1}(\mathcal{S}_\lambda\nabla_{||}\cdot){\bf f}||_2\leq c_{m,l}||{\bf f}||_2,\notag\\
(ii)&&\sup_{\lambda>0} ||\lambda^{m+2l+3}\partial_t^{l+1}\partial_\lambda^{m+1}(\nabla_{||} \mathcal{S}_\lambda f)||_2\leq c_{m,l}||{f}||_2.
\end{eqnarray*}
Furthermore, if $m+2l\geq -1$ then
 \begin{eqnarray*}
(iii)&&\sup_{\lambda>0} ||\lambda^{m+2l+2}\partial_t^{l+1}\partial_\lambda^{m+1} (\mathcal{S}_\lambda f)||_2\leq c_{m,l}||{f}||_2.
\end{eqnarray*}
\end{lemma}
\begin{proof} We first note that to prove $(ii)$ it suffices to only prove $(i)$, as, by duality, $(ii)$  follows from $(i)$ applied to $\mathcal{S}_\lambda^\ast$. To prove
$(i)$, fix $\lambda>0$ and consider $m\geq -1$, $l\geq -1$. Then
 \begin{eqnarray*}
&&||\lambda^{m+2l+3}\partial_t^{l+1}\partial_\lambda^{m+1}(\mathcal{S}_\lambda\nabla_{||}\cdot){\bf f}||_{2}^2
\ \leq \ \sum_Q\int_Q|\lambda^{m+2l+3}\partial_t^{l+1}\partial_\lambda^{m+1}(\mathcal{S}_\lambda\nabla_{||}\cdot){\bf f}(x,t)|^2\, dxdt,
\end{eqnarray*}
where the sum runs over the dyadic grid of parabolic cubes with $l(Q)\approx\lambda$. With $Q$ fixed we see that
 \begin{eqnarray*}
&&\int_Q|\lambda^{m+2l+3}\partial_t^{l+1}\partial_\lambda^{m+1}(\mathcal{S}_\lambda\nabla_{||}\cdot){\bf f}(x,t)|^2\, dxdt\notag\\
&& \qquad \leq \ \int_Q|\lambda^{m+2l+3}\partial_t^{l+1}\partial_\lambda^{m+1}(\mathcal{S}_\lambda\nabla_{||}\cdot)({\bf f}1_{2Q})(x,t)|^2\, dxdt\notag\\
&& \qquad \qquad + \ \sum_{k\geq 1}\int_Q|\lambda^{m+2l+3}\partial_t^{l+1}\partial_\lambda^{m+1}(\mathcal{S}_\lambda\nabla_{||}\cdot)({\bf f}1_{2^{k+1}Q\setminus 2^kQ})(x,t)|^2\, dxdt\notag\\
&&\qquad \leq \ c\lambda^{2m+4l+6}(l(Q))^{-2m-4l-6}||{\bf f}||^2_{L^2(2Q)}\notag\\
&&\qquad \qquad + \ \sum_{k\geq 1}c2^{-(n+2)k}\lambda^{2m+4l+6}(2^kl(Q))^{-2m-4l-6}
||{\bf f}||^2_{L^2(2^{k+1}Q\setminus 2^kQ)}\notag\\
&& \qquad \leq \ c\biggl (||{\bf f}||^2_{L^2(2Q)}+\sum_{k\geq 1}2^{-(n+2)k}2^{-(2m+4l+6)k}
||{\bf f}||^2_{L^2(2^{k+1}Q\setminus 2^kQ)}\biggr ),
\end{eqnarray*}
by Lemma \ref{le4} $(i)$ and $(ii)$, as $l(Q)\approx\lambda$. Hence,
 \begin{eqnarray}\label{kk1}
&&||\lambda^{m+2l+3}\partial_t^{l+1}\partial_\lambda^{m+1}(\mathcal{S}_\lambda\nabla_{||}\cdot){\bf f}||_{L^2(\mathbb R^{n+1})}^2\notag\\
&& \qquad \leq \ c||{\bf f}||_2^2+c\sum_Q\sum_{k\geq 1}2^{-(n+2)k} 2^{-(2m+4l+6)k}
||{\bf f}||^2_{L^2(2^{k+1}Q\setminus 2^kQ)}.
\end{eqnarray}
To complete the proof of $(i)$ we now note that there exists, given a point $(x,t)$, at most $c_n2^{(n+2)k}$ cubes $Q$ such that $(x,t)\in 2^{k+1}Q\setminus 2^kQ$. Hence, using this, and the estimate in \eqref{kk1}, we see that
 \begin{eqnarray*}
||\lambda^{m+2l+3}\partial_t^{l+1}\partial_\lambda^{m+1}(\mathcal{S}_\lambda\nabla_{||}\cdot){\bf f}||_{L^2(\mathbb R^{n+1})}^2&\leq& c||{\bf f}||_2^2+c\sum_{k\geq 1} 2^{-(2m+4l+6)k}
||{\bf f}||^2_2\notag\\
&\leq& c||{\bf f}||_2^2,
\end{eqnarray*}
as long as $m+2l>-3$. This completes the proof of $(i)$.  Using Lemma \ref{le4} $(iii)$ and $(iv)$, the proof of $(iii)$ is similar. We omit further details.
\end{proof}

\begin{lemma}\label{appf} Let $f\in C_0^\infty(\mathbb R^{n+1},\mathbb C)$ and $\lambda_0>0$. Then $\mathcal{S}_{\lambda_0} f\in  {\mathbb H}(\mathbb R^{n+1},\mathbb C)\cap L^2(\mathbb R^{n+1},\mathbb C)$.
\end{lemma}
\begin{proof}
Given $f\in C_0^\infty(\mathbb R^{n+1},\mathbb C)$ we let $Q\subset \mathbb R^{n+1}$ be a  parabolic cube, centered at $(0,0)$, such that the support of $f$ is contained in $Q$. Let $\lambda_0>0$ be fixed. We have to prove that
$||\nabla_{||}\mathcal{S}_{\lambda_0} f||_2<\infty$,
$||H_tD_{1/2}^t\mathcal{S}_{\lambda_0} f||_2<\infty$, and that $||\mathcal{S}_{\lambda_0} f||_2<\infty$. To estimate
$||\nabla_{||}\mathcal{S}_{\lambda_0} f||_2$ we see, by duality, that it suffices to bound
 \begin{eqnarray*}
\int_{Q}|(\mathcal{S}_{\lambda_0}^\ast\nabla_{||}\cdot){\bf f}(x,t)|^2\, dxdt
&\leq&\int_{Q}|(\mathcal{S}_{\lambda_0}^\ast\nabla_{||}\cdot)({\bf f}1_{2Q})(x,t)|^2\, dxdt\notag\\
&&+\sum_{k\geq 1}\int_{Q}|(\mathcal{S}_{\lambda_0}^\ast \nabla_{||}\cdot)({\bf f}1_{2^{k+1}Q\setminus 2^kQ})(x,t)|^2\, dxdt,
\end{eqnarray*}
where ${\bf f}\in C_0^\infty(\mathbb R^{n+1},\mathbb C^n)$, $||{\bf f}||_2=1$. However, now using  the adjoint version of Lemma \ref{le4} $(i)$, $(ii)$ with $l=-1=m$, we immediately see that
 \begin{eqnarray*}
\int_{Q}|(\mathcal{S}_{\lambda_0}^\ast\nabla_{||}\cdot){\bf f}(x,t)|^2\, dxdt\leq c(n,\Lambda,\lambda_0)<\infty,
\end{eqnarray*}
 whenever ${\bf f}\in C_0^\infty(\mathbb R^{n+1},\mathbb C^n)$, $||{\bf f}||_2=1$. To estimate
 $||H_tD_{1/2}^t\mathcal{S}_{\lambda_0} f||_2$ we first note that
 $$||H_tD_{1/2}^t\mathcal{S}_{\lambda_0} f||_2^2\leq ||\partial_t\mathcal{S}_{\lambda_0} f||_2||\mathcal{S}_{\lambda_0} f||_2.$$
 Using Lemma \ref{le5} $(iii)$ we see that $||\partial_t\mathcal{S}_{\lambda_0} f||_2\leq c(n,\Lambda,\lambda_0)||f||_2<\infty$. To estimate $||\mathcal{S}_{\lambda_0} f||_2$ we write
 \begin{eqnarray*}
\int_{\mathbb R^{n+1}}|\mathcal{S}_{\lambda_0}f(x,t)|^2\, dxdt
&\leq&\int_{2Q}|\mathcal{S}_{\lambda_0} f(x,t)|^2\, dxdt\notag\\
&&+\sum_{k\geq 1}\int_{2^{k+1}Q\setminus 2^kQ}|\mathcal{S}_{\lambda_0} f(x,t)|^2\, dxdt.
\end{eqnarray*}
Using this and Lemma \ref{le2+} $(i)$ we deduce that
 \begin{eqnarray*}
\int_{\mathbb R^{n+1}}|\mathcal{S}_{\lambda_0}f(x,t)|^2\, dxdt\leq c(n,\Lambda,\lambda_0)<\infty.
\end{eqnarray*}
This completes the proof of the lemma.
\end{proof}

\section{Estimates of non-tangential maximal functions and square functions}\label{sec4}

 Consider $\mathcal{S}_\lambda=\mathcal{S}_\lambda^{\mathcal{H}}$, for $\lambda>0$, where $\mathcal{H}=\partial_t-\div A\nabla$ is assumed to satisfy \eqref{eq3}-\eqref{eq4} as well as \eqref{eq14+}-\eqref{eq14++}.  Recall the notation $|||\cdot|||$, $\Phi(f)$, introduced in \eqref{keyestint-ex+-}, \eqref{keyestint-ex+}. This section is devoted to the proof of the following two lemmas.

\begin{lemma}\label{lemsl1++}  Then there exists a constant $c$, depending at most
     on $n$, $\Lambda$, and the De Giorgi-Moser-Nash constants, such that
\begin{eqnarray*}
(i)&& ||N_\ast(\partial_\lambda \mathcal{S}_\lambda f)||_2\leq c(\sup_{\lambda>0}||\partial_\lambda \mathcal{S}_\lambda||_{2\to 2}+1)||f||_2,\notag\\
(ii)&&||\tilde N_\ast(\nabla_{||}\mathcal{S}_\lambda f)||_2\leq c \biggl (||f||_2+\sup_{\lambda>0}||\nabla_{||}\mathcal{S}_\lambda f||_2+||N_{\ast\ast}(\partial_\lambda \mathcal{S}_\lambda f)||_2\biggl ),\notag\\
(iii)&&||\tilde N_\ast(H_tD_{1/2}^t\mathcal{S}_\lambda f)||_2\leq c\biggl (||f||_2+\sup_{\lambda>0}||H_tD_{1/2}^t\mathcal{S}_\lambda f||_2\biggr)\notag\\
&&\quad\quad\quad\quad\quad\quad\quad\quad + \ c\biggl (||\tilde N_{\ast\ast}(\nabla_{||}
\mathcal{S}_\lambda f)||_2+||N_{\ast\ast}(\partial_\lambda \mathcal{S}_\lambda f)||_2\biggr ),
\end{eqnarray*}
whenever  $f\in L^2(\mathbb R^{n+1},\mathbb C)$.
\end{lemma}


\begin{lemma}\label{lemsl1}  Assume $m\geq -1$, $l\geq -1$. Let $\Phi(f)$ be defined as in \eqref{keyestint-ex+}. Assume that $\Phi(f)<\infty$ whenever $f\in L^2(\mathbb R^{n+1},\mathbb C)$.   Then there exists a constant $c$, depending at most
     on $n$, $\Lambda$, the De Giorgi-Moser-Nash constants, and $m$, $l$, such that
         \begin{eqnarray*}
              (i)&&|||\lambda^{m+2l+4}\nabla\partial_\lambda\partial_t^{l+1}\partial_\lambda^{m+1}\mathcal{S}_{\lambda}f|||
              \leq c(\Phi(f)+||f||_2),\notag\\
     (ii)&&|||\lambda^{m+2l+4}\partial_t\partial_t^{l+1}\partial_\lambda^{m+1}\mathcal{S}_{\lambda}f|||
     		\leq c(\Phi(f)+||f||_2),
     \end{eqnarray*}
whenever $f\in L^2(\mathbb R^{n+1},\mathbb C)$.
\end{lemma}

\subsection{Proof of Lemma \ref{lemsl1++}} Throughout the proof we  can, without loss of generality, assume that $f\in C_0^\infty(\mathbb R^{n+1},\mathbb C)$. We  let $Q\subset\mathbb R^{n+1}$ be the (smallest) cube centered at $(0,0)$ such that the support of $f$ is contained in $\frac 1 2Q$. Let $\delta>0$ be small and let $1_{\lambda>2\delta}$ denote the indicator function for the set $\{\lambda:\ {\lambda>2\delta}\}\subset\mathbb R$.

\begin{proof}[Proof of Lemma \ref{lemsl1++} $(i)$]
We let $(x_0,t_0)\in\mathbb R^{n+1}$. Recall that the kernel of $\partial_\lambda{\mathcal{S}}_\lambda$ is $K_{0,\lambda}(x,t,y,s)$ introduced in \eqref{kernel}. $K_{0,\lambda}(x,t,y,s)$ is a (parabolic) Calderon-Zygmund kernel  satisfying the  Calderon-Zygmund type estimates stated in Lemma \ref{le2+}. Given $(x_0,t_0)\in\mathbb R^{n+1}$ we consider $(x,t,\lambda)\in \Gamma(x_0,t_0)$. Then
\begin{eqnarray*}
&&|\partial_\lambda {\mathcal{S}}_{\lambda} (f)(x,t)-\partial_\lambda {\mathcal{S}}_{\lambda} (f)(x_0,t_0)|\notag\\
&& \qquad \leq \ \int_{\mathbb R^{n+1}}|K_{0,{\lambda}}(x,t,y,s)-K_{0,{\lambda}}(x_0,t_0,y,s)||f(y,s)|\, dyds\notag\\
&& \qquad \leq \ cM(f)(x_0,t_0),
\end{eqnarray*}
by Lemma \ref{le2+} and  where $M$ is the  parabolic Hardy-Littlewood maximal function. Hence
\begin{eqnarray*}
N_\ast(1_{\lambda>2\delta}\partial_\lambda {\mathcal{S}}_{\lambda}f)(x_0,t_0)\leq \sup_{\lambda>2\delta}|\partial_\lambda {\mathcal{S}}_{\lambda} (f)(x_0,t_0)|+cM(f)(x_0,t_0),
\end{eqnarray*}
and we intend to  estimate $|\partial_\lambda {\mathcal{S}}_{\lambda} (f)(x_0,t_0)|$ for $\lambda>2\delta$. To do this we fix $\lambda>2\delta$ and we decompose
$\partial_\lambda {\mathcal{S}}_{\lambda} (f)(x_0,t_0)$ as
\begin{eqnarray*}
&&\int_{||(x_0-y,t_0-s)||>5\lambda}(K_{0,{\lambda}}(x_0,t_0,y,s)-K_{0,\delta}(x_0,t_0,y,s))f(y,s)\, dyds\notag\\
&& \qquad\qquad + \ \int_{||(x_0-y,t_0-s)||\leq5\lambda}K_{0,{\lambda}}(x_0,t_0,y,s)f(y,s)\, dyds\notag\\
&& \qquad\qquad - \ \int_{\lambda<||(x_0-y,t_0-s)||<5\lambda}K_{0,\delta}(x_0,t_0,y,s)f(y,s)\, dyds\notag\\
&& \qquad\qquad + \ \int_{||(x_0-y,t_0-s)||>\lambda}K_{0,\delta}(x_0,t_0,y,s)f(y,s)\, dyds\notag\\
&&  \qquad =: \ I_1^\delta(x_0,t_0,{\lambda})+I_2^\delta(x_0,t_0,\lambda)+I_3^\delta(x_0,t_0,\lambda)+I_4^\delta(x_0,t_0,\lambda).
\end{eqnarray*}
Using Lemma \ref{le2+} we see that
\begin{eqnarray*}
|I_1^\delta(x_0,t_0,\lambda)+I_2^\delta(x_0,t_0,\lambda)+I_3^\delta(x_0,t_0,\lambda)|\leq cM(f)(x_0,t_0).
\end{eqnarray*}
Furthermore,
\begin{eqnarray*}
|I_4^\delta(x_0,t_0,\lambda)|\leq \mathcal{T}_\ast^\delta f(x_0,t_0),
\end{eqnarray*}
where
\begin{eqnarray*}
 \mathcal{T}_\ast^\delta f(x_0,t_0)=\sup_{\epsilon>2\delta} |\mathcal{T}_\epsilon^\delta f(x_0,t_0)|
\end{eqnarray*}
and
\begin{eqnarray*}
 \mathcal{T}_\epsilon^\delta f(x_0,t_0)=\int_{||(x_0-y,t_0-s)||>\epsilon}K_{0,\delta}(x_0,t_0,y,s)f(y,s)\, dyds.
\end{eqnarray*}
We have to prove that $ \mathcal{T}_\ast^\delta:L^2(\mathbb R^{n+1},\mathbb C)\to L^2(\mathbb R^{n+1},\mathbb C)$ and we have to estimate
$||\mathcal{T}_\ast^\delta||_{2\to 2}$. To do this we carry out an argument similar to the proof of Cotlar's inequality for Calderon-Zygmund operators.  With $\epsilon>0$ fixed, we let $Q_\epsilon$ be the the largest parabolic cube, centered at $(x_0,t_0)$, which satisfies
that $2Q_\epsilon\cap\{(y,s)\in\mathbb R^{n+1}:\ ||(x_0-y,t_0-s)||>\epsilon\}=\emptyset$. Then $l(Q_\epsilon)\approx\epsilon$.  Write
$f=f{1}_{2Q_\epsilon}+f{1}_{\mathbb R^{n+1}\setminus 2Q_\epsilon}$. Then
\begin{eqnarray*}
 |\mathcal{T}_\epsilon^\delta f(x_0,t_0)|&=&|\partial_\lambda\mathcal{S}_\delta (f{1}_{\mathbb R^{n+1}\setminus 2Q_\epsilon})(x_0,t_0)|\notag\\
 &\leq &cM(f)(x_0,t_0)+ |\partial_\lambda\mathcal{S}_\delta f(x,t)|+ |\partial_\lambda\mathcal{S}_\delta (f{1}_{2Q_\epsilon})(x,t)|,
\end{eqnarray*}
whenever $(x,t)\in Q_\epsilon$ and where have used Lemma \ref{le2+} once again. Let $r\in (0,1)$. Taking a $L^r$ average in the last display with respect  to $(x,t)$, we see that
\begin{eqnarray*}
 |\mathcal{T}_\epsilon^\delta f(x_0,t_0)|&\leq& cM(f)(x_0,t_0)+ (M(|\partial_\lambda\mathcal{S}_\delta f|^r)(x_0,t_0))^{1/r}\notag\\
 &&+\bigl(\mean{Q_\epsilon}|\partial_\lambda\mathcal{S}_\delta (f{1}_{2Q_\epsilon})|^r\, dxdt\bigr)^{1/r}.
\end{eqnarray*}
Hence,
\begin{eqnarray*}
 |\mathcal{T}_\epsilon^\delta f(x_0,t_0)|&\leq& cM(f)(x_0,t_0)+ (M(|\partial_\lambda\mathcal{S}_\delta f|^r)(x_0,t_0))^{1/r}+M(|\partial_\lambda\mathcal{S}_\delta f|)(x_0,t_0).
\end{eqnarray*}
Furthermore, using an equality attributed to Kolmogorov, see Lemma 10 on p. 35 in \cite{CM} for example, and that the support of $f$ is contained in $Q$, we see that
\begin{eqnarray*}
(M(|\partial_\lambda\mathcal{S}_\delta f|^r)(x_0,t_0))^{1/r}\leq c||\partial_\lambda\mathcal{S}_\delta||_{L^1(Q)\to L^{1,\infty}(5Q)}\bigr)M(f)(x_0,t_0),
\end{eqnarray*}
where $L^{1,\infty}(5Q)$ is weak-$L^1$. Using that $\partial_\lambda\mathcal{S}_\delta$ is a Calderon-Zygmund operator one can deduce, by retracing, and localizing, the proof of weak estimates in  Calderon-Zygmund theory based on $L^2$ estimates, that
$$||\partial_\lambda\mathcal{S}_\delta||_{L^1(Q)\to L^{1,\infty}(5Q)}\leq c\bigl(1+||\partial_\lambda\mathcal{S}_\delta||_{L^2(Q)\to L^{2}(\mathbb R^{n+1})}\bigr),$$
where $c$ depends on the kernel $K_{0,\lambda}$ through the constants appearing in Lemma \ref{le2+}. For a detailed account of the dependence of the constant c, see \cite{NTV}. Hence
\begin{eqnarray*}
\mathcal{T}_\ast^\delta f(x_0,t_0)\leq c\bigl (1+||\partial_\lambda\mathcal{S}_\delta||_{L^2(Q)\to L^{2}(\mathbb R^{n+1})}\bigr)M(f)(x_0,t_0)+M(|\partial_\lambda\mathcal{S}_\delta f|)(x_0,t_0)
\end{eqnarray*}
and retracing the estimates we can conclude that we have proved that
\begin{eqnarray*}
N_\ast(1_{\lambda>2\delta}\partial_\lambda {\mathcal{S}}_{\lambda}f)(x_0,t_0)\leq c\bigl (1+||\partial_\lambda\mathcal{S}_\delta||_{2\to 2}\bigr)M(f)(x_0,t_0)+M(|\partial_\lambda\mathcal{S}_\delta f|)(x_0,t_0)
\end{eqnarray*}
whenever $(x_0,t_0)\in\mathbb R^{n+1}$ and $\delta>0$. Hence,
\begin{eqnarray*}
||N_\ast(1_{\lambda>2\delta}\partial_\lambda {\mathcal{S}}_{\lambda}f)||_2\leq c\bigl (1+\sup_{\lambda>0}||\partial_\lambda\mathcal{S}_\lambda||_{2\to 2}\bigr)||f||_2,
\end{eqnarray*}
whenever $f\in  C_0^\infty(\mathbb R^{n+1},\mathbb C)$ and for a constant $c$, depending at most
     on $n$, $\Lambda$, and the De Giorgi-Moser-Nash constants, in particular $c$ is independent of $\delta$. Letting $\delta\to 0$  completes the proof of Lemma \ref{lemsl1++}  $(i)$.
\end{proof}

\begin{proof}[Proof of Lemma \ref{lemsl1++} $(ii)$]

We let $(x_0,t_0)\in\mathbb R^{n+1}$.  To estimate $\tilde N_\ast(1_{\lambda>2\delta}\nabla_{||}{\mathcal{S}}_\lambda f)(x_0,t_0)$ it suffices to bound
\begin{eqnarray*}
\biggl (\mean{W_\lambda(x_0,t_0)}|\nabla_{||}{\mathcal{S}}_\sigma f(y,s)|^2\, dydsd\sigma\biggr)^{1/2},\
\end{eqnarray*}
 where $$W_\lambda(x,t):=\{(y,s,\sigma):\ (y,s)\in Q_\lambda(x,t),\lambda/2<\sigma<3\lambda/2\}$$ and for $\lambda>4\delta/3$ which we from now on assume.  In the following
we let, for $m\in\{0,1,...,4\}$
$$2^m W_{\lambda}(x,t):=\{(y,s,\sigma):\ (y,s)\in Q_{2^m\lambda}(x,t),\lambda/2-m\lambda 2^{-10}<\sigma<3\lambda/2+m\lambda 2^{-10}\}.$$
Then $2^0W_\lambda(x,t)=W_\lambda(x,t)$. Using this notation and energy estimates, Lemma \ref{le1--}, we see that
\begin{eqnarray*}
\mean{W_\lambda(x_0,t_0)}|\nabla_{||}{\mathcal{S}}_\sigma f(y,s)|^2\, dydsd\sigma\leq \frac c{\lambda^2}\mean{2W_{\lambda}(x_0,t_0)}|{\mathcal{S}}_\sigma f(y,s)-A|^2\, dydsd\sigma,
\end{eqnarray*}
where $A$ is a constant which in the following is a degree of freedom. Furthermore, using \eqref{eq14+} with $p=1$  we see that
\begin{eqnarray*}
&&\biggl (\mean{W_{\lambda}(x_0,t_0)}|\nabla_{||}{\mathcal{S}}_\sigma f(y,s)|^2\, dydsd\sigma\biggr)^{1/2}
\ \leq \
\frac c{\lambda}\mean{2^2W_{\lambda}(x_0,t_0)}|{\mathcal{S}}_\sigma f(y,s)-A|\, dydsd\sigma.
\end{eqnarray*}
We write
\begin{eqnarray*}
 &&\frac 1{\lambda}\mean{2^2W_{\lambda}(x_0,t_0)}|{\mathcal{S}}_\sigma f(y,s)-A|\, dydsd\sigma\notag\\
 && \qquad \leq \ \frac 1{\lambda}\mean{2^2W_{\lambda}(x_0,t_0)}|{\mathcal{S}}_\sigma f(y,s)-
 {\mathcal{S}}_\sigma f(y,t_0)|\, dydsd\sigma\notag\\
 && \qquad \qquad + \ \frac 1{\lambda}\mean{2^2W_{\lambda}(x_0,t_0)}|{\mathcal{S}}_\sigma f(y,t_0)-A|\, dydsd\sigma\notag\\
 && \qquad =: \ I_1+I_2.
\end{eqnarray*}
By the fundamental theorem of calculus we have
\begin{eqnarray*}
 I_1\leq \mean{2^3W_{\lambda}(x_0,t_0)}|\lambda \partial_t {\mathcal{S}}_\sigma f(y,s)|\, dydsd\sigma.
\end{eqnarray*}
Let $Q\subset\mathbb R^{n+1}$ be a parabolic cube centered at $(x_0,t_0)$ and with side length $8\lambda$. Then $I_1$ is bounded by
\begin{eqnarray*}
&&c\int_{\lambda/8}^{2\lambda}\int_{Q}|\lambda^{-n-2}\partial_t {\mathcal{S}}_\sigma (f{\bf 1}_{2Q})(y,s)|\, dydsd\sigma\notag\\
 && \qquad + \ c\int_{\lambda/8}^{2\lambda}\int_{Q}|\lambda^{-n-2}\bigl (\partial_t {\mathcal{S}}_\sigma (f{\bf 1}_{\mathbb R^{n+1}\setminus 2Q})(y,s)-\partial_t {\mathcal{S}}_\sigma (f{\bf 1}_{\mathbb R^{n+1}\setminus 2Q})(x_0,t_0)\bigr)|\, dydsd\sigma\notag\\
 &&\qquad + \ c\int_{\lambda/8}^{2\lambda}|\partial_t {\mathcal{S}}_\sigma (f{\bf 1}_{\mathbb R^{n+1}\setminus 2Q})(x_0,t_0)|\, d\sigma\notag\\
 && \quad =: \ I_{11}+I_{12}+I_{13}.
\end{eqnarray*}
Using Lemma \ref{le2+} we see that
\begin{eqnarray*}
 I_{11}+I_{12}\leq c M(f)(x_0,t_0),
\end{eqnarray*}
where $M$ is the  parabolic Hardy-Littlewood maximal function. Furthermore,
\begin{eqnarray*}
I_{13}&\leq &c\sum_{k=1}^\infty \int_{\lambda/8}^{2\lambda}|\partial_t {\mathcal{S}}_\sigma (f{\bf 1}_{2^{k+1}Q\setminus 2^kQ})(x_0,t_0)|\, d\sigma\notag\\
&\leq &c\lambda\sum_{k=1}^\infty (2^k\lambda)^{-n-3}\int_{2^{k+1}Q}|f(y,s)|\,  dyds\leq c M(f)(x_0,t_0).
\end{eqnarray*}
Hence, we can conclude that
\begin{eqnarray}\label{ss1}
I_1\leq c M(f)(x_0,t_0).
\end{eqnarray}
Focusing on $I_2$ we see that
\begin{eqnarray*}
I_2&\leq &\frac 1{\lambda}\mean{2^2W_{\lambda}(x_0,t_0)}|{\mathcal{S}}_\sigma f(y,t_0)-{\mathcal{S}}_{\delta/4} f(y,t_0)|\, dydsd\sigma\notag\\
&&+\frac 1{\lambda}\mean{2^2W_{\lambda}(x_0,t_0)}|{\mathcal{S}}_{\delta/4}f(y,t_0)-A|\, dydsd\sigma\notag\\
&=:&I_{21}+I_{22}.
\end{eqnarray*}
By the fundamental theorem of calculus
\begin{eqnarray*}
I_{21}&\leq &\frac 1{\lambda}\mean{2^3W_{\lambda}(x_0,t_0)}\lambda|N_{\ast\ast}^x(\partial_\lambda {\mathcal{S}}_\lambda f(\cdot,t_0))(y)|\, dydsd\sigma\notag\\
&\leq & M^x(N_{\ast\ast}^x(\partial_\lambda {\mathcal{S}}_\lambda f(\cdot,t_0))(\cdot))(x_0),
\end{eqnarray*}
where $M^x$ is the Hardy-Littlewood maximal function in $x$ only and $N_{\ast\ast}^x$ is an elliptic non tangential maximal function on a fixed time slice. Finally, let
$A$ be the average of ${\mathcal{S}}_{\delta/4}f(y,t_0)$, with respect to $y$, on an spatial surface cube around $x_0$ with sidelength $\lambda$. Then, using the  $L^1$-Poincare inequality we deduce that
\begin{eqnarray*}
I_{22}&\leq &c M^x(\nabla_{||}{\mathcal{S}}_{\delta/4}f(\cdot,t_0))(x_0).
\end{eqnarray*}
Retracing the argument we can conclude that
\begin{eqnarray*}
\tilde N_\ast(1_{\lambda>2\delta}\nabla_{||}{\mathcal{S}}_\lambda f)(x_0,t_0)&\leq &c \bigl (M(f)(x_0,t_0)+M^x(N_{\ast\ast}^x(\partial_\lambda {\mathcal{S}}_\lambda f(\cdot,t_0))(\cdot))(x_0)\notag\\
&&+
M^x(\nabla_{||}{\mathcal{S}}_{\delta/4}f(\cdot,t_0))(x_0)\bigr ).
\end{eqnarray*}
Hence
\begin{eqnarray*}
||\tilde N_\ast(1_{\lambda>2\delta}\nabla_{||}{\mathcal{S}}_\lambda f)||_2^2&\leq &c \bigl (||f||_2^2+||\nabla_{||}{\mathcal{S}}_{\delta/4}f||_2^2\bigr )\notag\\
&&+\int_{-\infty}^\infty
\int_{\mathbb R^{n}}|N_{\ast\ast}^x(\partial_\lambda {\mathcal{S}}_\lambda f(\cdot,t))(x)|^2\, dxdt.
\end{eqnarray*}
However,
$$N_{\ast\ast}^x(\partial_\lambda {\mathcal{S}}_\lambda f(\cdot,t_0))(x_0)\leq N_{\ast\ast}(\partial_\lambda {\mathcal{S}}_\lambda f)(x_0,t_0)$$
and we can conclude that
\begin{eqnarray*}
||\tilde N_\ast(1_{\lambda>2\delta}\nabla_{||}{\mathcal{S}}_\lambda f)||_2\leq c \bigl (||f||_2+\sup_{\lambda>0}||\nabla_{||}{\mathcal{S}}_\lambda f||_2+||N_{\ast\ast}(\partial_\lambda {\mathcal{S}}_\lambda f)||_2\bigr ).
\end{eqnarray*}
This completes the proof of Lemma \ref{lemsl1++} $(ii)$.
\end{proof}

\begin{proof}[Proof of Lemma \ref{lemsl1++} $(iii)$]
We again fix $(x_0,t_0)\in\mathbb R^{n+1}$ and we note that to estimate $$\tilde N_\ast(1_{\lambda>2\delta}H_tD_{1/2}^t{\mathcal{S}}_\lambda f)(x_0,t_0)$$ it suffices to bound
\begin{eqnarray*}
\biggl (\mean{W_\lambda(x_0,t_0)}|H_tD_{1/2}^t{\mathcal{S}}_\sigma f(y,s)|^2\, dydsd\sigma\biggr)^{1/2},  \ \lambda>4\delta/3.
\end{eqnarray*}
Consider $(y,s,\sigma)\in W_\lambda(x_0,t_0)$, $\lambda>4\delta/3$, and let $K\gg 1$ be a degree of freedom to be chosen. Then
\begin{eqnarray*}
H_tD_{1/2}^t({\mathcal{S}}_\sigma f)(y,s)&=&\lim_{\epsilon\to 0}\int_{\epsilon\leq |s-t|<1/\epsilon}\frac {\mbox{sgn}(s-t)}{|s-t|^{3/2}}({\mathcal{S}}_\sigma f)(y,t)\, dt\notag\\
&=&\lim_{\epsilon\to 0}\int_{\epsilon\leq |s-t|<(K\sigma)^2}\frac {\mbox{sgn}(s-t)}{|s-t|^{3/2}}({\mathcal{S}}_\sigma f)(y,t)\, dt\notag\\
&&+\lim_{\epsilon\to 0}\int_{(K\sigma)^2\leq |s-t|<1/\epsilon}\frac {\mbox{sgn}(s-t)}{|s-t|^{3/2}}({\mathcal{S}}_\sigma f)(y,t)\, dt\notag\\
&=:&g_1(y,s,\sigma)+g_2(y,s,\sigma).
\end{eqnarray*}
Let
$$g_3(x_0,t_0,\sigma):=\sup_{\{y:\ |y-x_0|\leq 4\sigma\}}\sup_{\{\tau:\ |\tau-t_0|\leq (4K\sigma)^2\}}|\partial_\tau({\mathcal{S}}_\sigma f)(y,\tau)|.$$
Then, using the oddness about $s$ of the kernel in the definition of $g_1$,
\begin{eqnarray*}
|g_1(y,s,\sigma)|\leq cK\lambda g_3(x_0,t_0,\sigma),
\end{eqnarray*}
whenever $(y,s,\sigma)\in W_\lambda(x_0,t_0)$.
Hence,
\begin{eqnarray*}
\biggl (\mean{W_\lambda(x_0,t_0)}|g_1(y,s,\sigma)|^2\, dydsd\sigma\biggr)
\leq c\lambda^2\int_{\lambda/8}^{2\lambda}| g_3(x_0,t_0,\sigma)|^2\, d\sigma.
\end{eqnarray*}
To estimate the right hand side in the last display, let $(y,\tau)$ be such that $ |y-x_0|\leq 4\sigma$, $ |\tau-t_0|\leq (4K\sigma)^2$. Let $Q\subset\mathbb R^{n+1}$ be a parabolic cube centered at $(x_0,t_0)$ and with side length $16K\sigma$. Then, for $K$ large enough we see that
\begin{eqnarray*}
|\lambda\partial_\tau({\mathcal{S}}_\sigma f)(y,\tau)|&\leq&\lambda|\partial_\tau{\mathcal{S}}_\sigma (f1_{2Q})(y,\tau)|\notag\\
&&+\lambda|\partial_\tau{\mathcal{S}}_\sigma( f1_{\mathbb R^{n+1}\setminus 2Q})(y,\tau)-\partial_\tau{\mathcal{S}}_\sigma (f1_{\mathbb R^{n+1}\setminus 2Q})(x_0,t_0)|\notag\\
&&+\lambda|\partial_\tau{\mathcal{S}}_\sigma (f1_{\mathbb R^{n+1}\setminus 2Q})(x_0,t_0)|.
\end{eqnarray*}
Basically repeating the proof of \eqref{ss1} we see that
\begin{eqnarray*}
\biggl (\mean{W_\lambda(x_0,t_0)}|g_1(y,s,\sigma)|^2\, dydsd\sigma\biggr)^{1/2}\leq c M(f)(x_0,t_0).
\end{eqnarray*}
To estimate $g_2(y,s,\sigma)$, whenever $(y,s,\sigma)\in W_\lambda(x_0,t_0)$, we introduce the function
\begin{eqnarray*}
g_4(\bar y,\bar s,\sigma):=\lim_{\epsilon\to 0}\int_{(K\sigma)^2\leq |t-\bar s|<1/\epsilon}\frac {\mbox{sgn}
(\bar s-t)}{|\bar s-t|^{3/2}}({\mathcal{S}}_{\delta/4} f)(\bar y,t)\, dt.
\end{eqnarray*}
Now
\begin{eqnarray*}
|g_2(y,s,\sigma)-g_4(x_0,t_0,\sigma)|&\leq & |g_2(y,s,\sigma)-g_2(x_0,s,\sigma)|\notag\\
&&+|g_2(x_0,s,\sigma)-g_2(x_0,t_0,\sigma)|\notag\\
&&+|g_2(x_0,t_0,\sigma)-g_4(x_0,t_0,\sigma)|.
\end{eqnarray*}
In particular,
\begin{eqnarray*}
|g_2(y,s,\sigma)-g_4(x_0,t_0,\sigma)|&\leq & \int_{(K\sigma)^2\leq |s-t|}\frac {|{\mathcal{S}}_\sigma f(y,t)-{\mathcal{S}}_\sigma f(x_0,t)|}{|t-s|^{3/2}}\, dt\notag\\
&&+\int_{(K\sigma)^2\leq |\xi|}\frac {|{\mathcal{S}}_\sigma f(x_0,\xi+s)-{\mathcal{S}}_\sigma f(x_0,\xi+t_0)|}{|\xi|^{3/2}}\, d\xi\notag\\
&&+\int_{(K\sigma)^2\leq |t-t_0|}\frac {|{\mathcal{S}}_\sigma f(x_0,t)-{\mathcal{S}}_{\delta/4} f(x_0,t)|}{|t_0-t|^{3/2}}\, dt\notag\\
&=:&h_1(y,s,\sigma)+h_2(y,s,\sigma)+h_3(x_0,t_0,\sigma).
\end{eqnarray*}
We note that
\begin{eqnarray*}
h_2(y,s,\sigma)&\leq&c\sigma^2\int_{(K\sigma)^2\leq |\xi|}\frac {N_\ast(\partial_t{\mathcal{S}}_\sigma f)(x_0,\xi+t_0)}{|\xi|^{3/2}}\, d\xi\notag\\
&\leq&c\sigma\int_{(K\sigma)^2\leq |\xi|}\frac {M(f)(x_0,\xi+t_0)}{|\xi|^{3/2}}\, d\xi\leq c
M^t(M(f)(x_0,\cdot))(t_0),
\end{eqnarray*}
where $M^t$ is the Hardy-Littlewood maximal operator in the $t$-variable, as we see by arguing as above. Similarly,
\begin{eqnarray*}
h_3(y,s,\sigma)\leq c M^t(N_\ast(\partial_\lambda {\mathcal{S}}_\sigma f)(x_0,\cdot))(t_0).
\end{eqnarray*}
We therefore focus on $h_1(y,s,\sigma)$. Let
$$\tilde h_1(y,\sigma):=\int_{\lambda^2\leq |t-t_0|}\frac {|{\mathcal{S}}_\sigma f(y,t)-{\mathcal{S}}_\sigma f(x_0,t)|}{|t-t_0|^{3/2}}\, dt.$$
If  $K$ is large enough, then $h_1(y,s,\sigma)\leq c\tilde h_1(y,\sigma)$, whenever $(y,s,\sigma)\in W_\lambda(x_0,t_0)$.
Hence we only have to estimate
$$\biggl (\mean{\hat Q_\lambda(x_0)\times I_{\lambda/2}(\lambda)}\tilde h_1^2\, dyd\sigma\biggr )^{1/2}=\sup\biggl | \mean{\hat Q_\lambda(x_0)\times I_{\lambda/2}(\lambda)}
\tilde h_1g\, dyd\sigma\biggl |,$$
where $\hat Q_\lambda(x_0)\subset \mathbb R^n$ now is a (non-parabolic) cube with side length $\lambda$ and center $x_0$, $I_{\lambda/2}(\lambda)$ is the interval $(\lambda/2,3\lambda/2)$,   and where the sup is taken with respect to all $g\in C_0^\infty(\mathbb R^{n+1},\mathbb R)$ such that
\begin{eqnarray}\label{g}\biggl (\mean{\hat Q_\lambda(x_0)\times I_{\lambda/2}(\lambda)}g^2\, dyd\sigma\biggr )^{1/2}=1.
\end{eqnarray}
Given $g$ as in \eqref{g} we let
$$E:=\mean{\hat Q_\lambda(x_0)\times I_{\lambda/2}(\lambda)}
\tilde h_1g\, dyd\sigma.$$
Then
\begin{eqnarray*}
E&=&\mean{\hat Q_\lambda(x_0)\times I_{\lambda/2}(\lambda)}\biggl (\int_{\lambda^2\leq |t-t_0|}\frac {|{\mathcal{S}}_\sigma f(y,t)-
{\mathcal{S}}_\sigma f(x_0,t)|}{|t-t_0|^{3/2}}\, dt\biggr )g(y,\sigma)\, dyd\sigma\notag\\
&\leq &c\sum_{j=0}^\infty(\lambda^22^j)^{-3/2}\mean{\hat Q_\lambda(x_0)\times I_{\lambda/2}(\lambda)}\biggl (\int_{I_j}{|{\mathcal{S}}_\sigma f(y,t)-
{\mathcal{S}}_\sigma f(x_0,t)|}\, dt\biggr )g(y,\sigma)\, dyd\sigma,
\end{eqnarray*}
where $I_j=\{t:\ \lambda^22^j\leq |t-t_0|<\lambda^22^{j+1}\}$. Let $\eta\in (-\lambda^2/100,\lambda^2/100)$ be a degree of freedom. Given any integer $i\in\{2^{j-1},....,2^{j+3}\}$ we let $t_{j,i}^\pm=t_0\pm i\lambda^2$, $N_j=(2^{j+3}-2^{j-1}+1)$. Given $\eta$ we let
$I_{j,i}(t_{j,i}^\pm+\eta,\lambda^2)$ be the interval centered at $t_{j,i}^\pm+\eta$ and of length $2\lambda^2$. Then $\{I_{j,i}(t_{j,i}^\pm+\eta,\lambda^2)\}_i$ is, for each $\eta\in (-\lambda^2/100,\lambda^2/100)$, a covering of
$I_j$ and $\{I_{j,i}(t_{j,i}+\eta,\lambda^2/10^4)\}$ is a disjoint collection. Using this we see that $|E|$ can be bounded from above by
\begin{eqnarray*}
&&c\lambda^2\sum_{j=0}^\infty(\lambda^22^j)^{-3/2}\sum_{i=1}^{N_j}\mean{W_\lambda(x_0,t_{j,i}^\pm+\eta)}|{\mathcal{S}}_\sigma f(y,t)-
{\mathcal{S}}_\sigma f(x_0,t)||g(y,\sigma)|\, dydtd\sigma\notag\\
&& \qquad \leq \  c\lambda^3\sum_{j=0}^\infty(\lambda^22^j)^{-3/2}\sum_{i=1}^{N_j}\tilde N_{\ast\ast}(\nabla_{||}{\mathcal{S}}_\lambda f)(x_0,t_{j,i}^\pm+\eta).
\end{eqnarray*}
This estimate holds uniformly with respect to $\eta\in (-\lambda^2/100,\lambda^2/100)$. In particular, taking the average with respect to $\eta$ we see that
\begin{eqnarray*}
|E|&\leq& c\lambda\sum_{j=0}^\infty(\lambda^22^j)^{-3/2}\int_{\{t:\ \lambda^22^{j-2}\leq |t-t_0|<\lambda^22^{j+4}\}}\tilde N_{\ast\ast}(\nabla_{||}{\mathcal{S}}_\lambda f)(x_0,t)\, dt\notag\\
&\leq& cM^t(\tilde N_{\ast\ast}(\nabla_{||}{\mathcal{S}}_\lambda f)(x_0,\cdot))(t_0).
\end{eqnarray*}
Putting the estimates together we can conclude, for $\lambda>4\delta/3$, that
\begin{eqnarray*}
\biggl (\mean{W_\lambda(x_0,t_0)}|H_tD_{1/2}^t{\mathcal{S}}_\sigma f(y,s)|^2\, dydsd\sigma\biggr)^{1/2}
\end{eqnarray*}
is bounded by
\begin{eqnarray*}
 && cM^t(\tilde N_{\ast\ast}(\nabla_{||}{\mathcal{S}}_\lambda f)(x_0,\cdot))(t_0)+c M^t(M(f)(x_0,\cdot))(t_0)+cM^t(N_\ast(\partial_\lambda {\mathcal{S}}_\sigma f)(x_0,\cdot))(t_0)\notag\\
&&+\biggl (\mean{W_\lambda(x_0,t_0)}|g_4(x_0,t_0,\sigma)|^2\, dydsd\sigma\biggr)^{1/2},
\end{eqnarray*}
 where $M^t$ is the Hardy-Littlewood maximal operator in the $t$-variable and $M$ is the  parabolic Hardy-Littlewood maximal function. Hence, letting $$\psi(x_0,t_0):=\sup_{\sigma>0}|g_4(x_0,t_0,\sigma)|$$
we see that
\begin{eqnarray*}
||\tilde N_\ast(1_{\lambda>2\delta}H_tD_{1/2}^t{\mathcal{S}}_\lambda f)||_2&\leq& c||f||_2\notag\\
&&+ c\bigl (||\tilde N_{\ast\ast}(\nabla_{||}
\mathcal{S}_\lambda f)||_2+||N_{\ast\ast}(\partial_\lambda \mathcal{S}_\lambda f)||_2\bigr )\notag\\
 &&+c||\psi||_2
\end{eqnarray*}
where the constant $c$ is independent of $\delta$. Hence, to complete the proof of $(iii)$ it remains to estimate $||\psi||_2$. To do this we first recall that $f\in C_0^\infty(\mathbb R^{n+1},\mathbb C)$. Hence, using  Lemma \ref{appf} we know that
$\mathcal{S}_{\delta/4} f\in  {\mathbb H}(\mathbb R^{n+1},\mathbb C)\cap L^2(\mathbb R^{n+1},\mathbb C)$. Using this it follows that
$${\mathcal{S}}_{\delta/4}f(x,t)=cI_{1/2}^t(D_{1/2}^t{\mathcal{S}}_{\delta/4}f)(x,t)=c I_{1/2}^th(x,t), $$
where $ I_{1/2}^t$ is the (fractional)  Riesz operator in $t$ defined on the Fourier transform side through the multiplier $|\tau|^{-1/2}$ and
$h(x,t):=(D_{1/2}^t{\mathcal{S}}_{\delta/4}f)(x,t)$. Using this we see that
$$\psi(x_0,t_0)=c\sup_{\epsilon>0}|\tilde V_\epsilon h(x_0,t_0)|=:c\tilde V_\ast h(x_0,t_0),$$
where $V_\epsilon$ is defined on functions $k\in L^2(\mathbb R,\mathbb R)$ by
$$V_\epsilon k(t)=\int_{\{|s-t|>\epsilon\}}\frac{\mbox{sgn}(t-s) I_{1/2}^tk(s)}{|s-t|^{3/2}}\, ds,$$
and $\tilde V_\epsilon h(x,t)=V_\epsilon h(x,\cdot)$ evaluated at $t$. However, using this notation we can apply Lemma 2.27 in \cite{HL} and conclude that
$$||\psi||_2\leq c||h||_2=c||D_{1/2}^t{\mathcal{S}}_{\delta/4}f||_2\leq c\sup_{\lambda>0}||H_tD_{1/2}^t\mathcal{S}_\lambda f||_2.$$
This completes the proof of Lemma \ref{lemsl1++} $(iii)$.
\end{proof}

\subsection{Proof of Lemma \ref{lemsl1}} We first note, using Lemma \ref{le1--}, Lemma \ref{le1a} and induction, that it suffices to prove
\begin{eqnarray*}
              (i')&&|||\lambda \nabla\partial_\lambda\mathcal{S}_{\lambda}f|||\leq c\Phi(f)+c||f||_2,\notag\\
     (ii')&&|||\lambda\partial_t\mathcal{S}_{\lambda}f|||\leq c\Phi(f)+c||f||_2,
     \end{eqnarray*}
whenever $f\in L^2(\mathbb R^{n+1},\mathbb C)$. To prove $(i')$ it  suffices to estimate $|||\lambda \nabla_{||}\partial_\lambda \mathcal{S}_{\lambda}f|||$. Given $\epsilon>0$ we let
\begin{eqnarray*}
A_1&:=&-\frac 12\int_{\epsilon}^{1/\epsilon}\int_{\mathbb R^{n+1}}\nabla_{||}\partial_\lambda^2 \mathcal{S}_{\lambda}f\cdot
\overline{\nabla_{||}\partial_\lambda \mathcal{S}_{\lambda}f}\,  \lambda^2{dxdtd\lambda},\notag\\
A_2&:=&-\frac 12\int_{\epsilon}^{1/\epsilon}\int_{\mathbb R^{n+1}}\nabla_{||}\partial_\lambda \mathcal{S}_{\lambda}f\cdot
\overline{\nabla_{||}\partial_\lambda^2 \mathcal{S}_{\lambda}f}\,  \lambda^2{dxdtd\lambda},\notag\\
A_3&:=&\int_{\mathbb R^{n+1}}\nabla_{||}\partial_\lambda \mathcal{S}_{\lambda}f\cdot
\overline{\nabla_{||}\partial_\lambda \mathcal{S}_{\lambda}f}\,  \lambda^2{dxdt}\biggl |_{\lambda=1/\epsilon},\notag\\
A_4&:=&\int_{\mathbb R^{n+1}}\nabla_{||}\partial_\lambda \mathcal{S}_{\lambda}f\cdot
\overline{\nabla_{||}\partial_\lambda \mathcal{S}_{\lambda}f}\,  \lambda^2{dxdt}\biggl |_{\lambda=\epsilon}.
\end{eqnarray*}
Using partial integration with respect to $\lambda$,
\begin{eqnarray*}
\int_{\epsilon}^{1/\epsilon}\int_{\mathbb R^{n+1}}\nabla_{||}\partial_\lambda \mathcal{S}_{\lambda}f\cdot
\overline{\nabla_{||}\partial_\lambda \mathcal{S}_{\lambda}f}\,  \lambda{dxdtd\lambda}=A_1+A_2+A_3+A_4.
\end{eqnarray*}
Furthermore, using Lemma \ref{le5} $(ii)$,
\begin{eqnarray*}
|A_1|+|A_2|+|A_3|+|A_4|\leq c|||\lambda^2 \nabla_{||}\partial_\lambda^2 \mathcal{S}_{\lambda}f|||^2+c||f||_2^2,
\end{eqnarray*}
with $c$ independent of $\epsilon$. Hence
\begin{eqnarray}\label{eq4.43}
|||\lambda \nabla_{||}\partial_\lambda \mathcal{S}_{\lambda}f|||^2&=&\lim_{\epsilon\to 0}\int_{\epsilon}^{1/\epsilon}\int_{\mathbb R^{n+1}}\nabla_{||}\partial_\lambda \mathcal{S}_{\lambda}f\cdot
\overline{\nabla_{||}\partial_\lambda \mathcal{S}_{\lambda}f}\,  \lambda{dxdtd\lambda}\notag\\
&\leq& c|||\lambda^2 \nabla_{||}\partial_\lambda^2 \mathcal{S}_{\lambda}f|||^2+c||f||_2^2.
\end{eqnarray}
$(i')$ now follows from an application of  Lemma \ref{le1--}. To prove $(ii')$ we first introduce, for $\epsilon>0$,
\begin{eqnarray*}\label{ua1-}
B_1&:=&-\frac 1 2\int_{\epsilon}^{1/\epsilon}\int_{\mathbb R^{n+1}}\partial_t\partial_\lambda \mathcal{S}_{\lambda}f
\overline{\partial_t \mathcal{S}_{\lambda}f}\,  \lambda^2{dxdtd\lambda},\notag\\
B_2&:=&-\frac 1 2\int_{\epsilon}^{1/\epsilon}\int_{\mathbb R^{n+1}}\partial_t\mathcal{S}_{\lambda}f
\overline{\partial_t \partial_\lambda \mathcal{S}_{\lambda}f}\,  \lambda^2{dxdtd\lambda},\notag\\
B_3&:=&\int_{\mathbb R^{n+1}}\partial_t \mathcal{S}_{\lambda}f
\overline{\partial_t \mathcal{S}_{\lambda}f}\,  \lambda^2{dxdt}\biggl |_{\lambda=1/\epsilon},\notag\\
B_4&:=&-\int_{\mathbb R^{n+1}}\partial_t \mathcal{S}_{\lambda}f
\overline{\partial_t \mathcal{S}_{\lambda}f}\,  \lambda^2{dxdt}\biggl |_{\lambda=\epsilon}.
\end{eqnarray*}
Then, using Lemma \ref{le5} $(iii)$
\begin{eqnarray*}\label{ua2-}
|B_1|+|B_2|+|B_3|+|B_4|\leq c|||\lambda^2\partial_t\partial_\lambda\mathcal{S}_{\lambda}f|||^2+c||f||_2^2,
\end{eqnarray*}
with $c$ independent of $\epsilon$. Hence, again by  integration by parts with respect to $\lambda$,
\begin{eqnarray}\label{ua2}
|||\lambda\partial_t\mathcal{S}_{\lambda}f|||^2 &=&\lim_{\epsilon\to 0}\int_{\epsilon}^{1/\epsilon}\int_{\mathbb R^{n+1}}\partial_t \mathcal{S}_{\lambda}f
\overline{\partial_t \mathcal{S}_{\lambda}f}\,  \lambda{dxdtd\lambda}\notag\\
&\leq& c|||\lambda^2\partial_t\partial_\lambda\mathcal{S}_{\lambda}f|||^2+c||f||_2^2.
\end{eqnarray}
Furthermore, repeating the above argument  it also follows that
\begin{eqnarray*}
|||\lambda^2\partial_t\partial_\lambda\mathcal{S}_{\lambda}f|||^2\leq c|||\lambda^3\partial_t\partial_\lambda^2\mathcal{S}_{\lambda}f|||^2+c||f||_2^2.
\end{eqnarray*}
Finally, using Lemma \ref{le1a} we can combine the above estimates and conclude that
\begin{eqnarray*}
|||\lambda\partial_t\mathcal{S}_{\lambda}f|||\leq c\Phi(f)+c||f||_2.
\end{eqnarray*}
This completes the proof of  $(ii')$ and hence the proof of Lemma \ref{lemsl1}.

\section{Resolvents, square functions and Carleson measures}\label{sec5}  In the following we collect some of the main results from \cite{N} to be used in the proof of our main results. Throughout the section we assume that  $\mathcal{H}$,  $\mathcal{H}^\ast$ satisfy \eqref{eq3}-\eqref{eq4}. We
let
 \begin{eqnarray*}
 \mathcal{L}_{||}:=-\div_{||} A_{||}\nabla_{||},
 \end{eqnarray*}
where $\div_{||}$ is the divergence operator in the variables $(\partial_{x_1},...,\partial_{x_n})$. $A_{||}$ is the $n\times n$-dimensional sub matrix of $A$ defined by $\{A_{i,j}\}_{i,j=1}^n$.  We also let
     \begin{eqnarray*}
     \mathcal{H}_{||}:=\partial_t+\mathcal{L}_{||},\quad \mathcal{H}_{||}^\ast:=-\partial_t+\mathcal{L}_{||}^\ast.
 \end{eqnarray*}
    Using this notation the equation $\mathcal{H} u=0$ can be written, formally, as
     \begin{eqnarray}\label{block++}
     	\mathcal{H}_{||}u-\sum_{j=1}^{n+1}A_{n+1,j}D_{n+1}D_ju-\sum_{i=1}^{n}D_i(A_{i,n+1}D_{n+1}u)=0.
     \end{eqnarray}
     In the proof of Lemma \ref{lemsl1c} below we will use that \eqref{block++} holds in an appropriate weak sense on cross sections $\lambda=$ constant. Indeed, let $\lambda\in (a,b)$ and let $\epsilon<\min (\lambda-a,b-\lambda)$. Set $\varphi_\epsilon(\sigma)=\epsilon^{-1}\varphi(\sigma/\epsilon)$ where $\varphi\in C_0^\infty(-1/2,1/2)$, $0\leq\varphi$, $\int\varphi\, d\sigma=1$. We let
     $\phi_{\lambda,\epsilon}(x,t, \sigma)=\psi(x,t)\varphi_\epsilon(\sigma)$ where $\psi\in C_0^\infty(\mathbb R^{n+1},\mathbb C)$.  Then, by the notion of weak solutions we have
     \begin{eqnarray}\label{weak}
&&\int_{\mathbb R^{n+2}} \biggl( A_{||}(x)\nabla_{||} u(x,t, \sigma)\cdot\nabla_{||}\overline{\phi_{\lambda,\epsilon}(x,t, \sigma)}-u(x,t, \sigma)\partial_t\overline{\phi_{\lambda,\epsilon}(x,t, \sigma)}\biggr )\, dxdtd\sigma\notag\\
&&=\sum_{j=1}^{n+1}\int_{\mathbb R^{n+2}}  A_{n+1,j}(x)\partial_{x_j}\partial_\lambda u(x,t, \sigma)\overline{\phi_{\lambda,\epsilon}(x,t, \sigma)}\, dxdtd\sigma\notag\\
&&-\sum_{i=1}^{n}\int_{\mathbb R^{n+2}}  A_{i,n+1}(x)\partial_\lambda u(x,t, \sigma)\partial_{x_i}\overline{\phi_{\lambda,\epsilon}(x,t, \sigma)}\, dxdtd\sigma.
\end{eqnarray}
Hence, if
    \begin{eqnarray}\label{block++apa}
     \nabla u,\ \nabla\partial_\lambda u\in L^2(\mathbb R^{n+1},\mathbb C^{n+1}),
     \end{eqnarray}
     uniformly in $\lambda\in(a,b)$, with norms depending continuously on $\lambda\in(a,b)$, then we can conclude, by letting $\eta\to 0$ in \eqref{weak}, that
          \begin{eqnarray}\label{weak+}
&&\int_{\mathbb R^{n+1}}  \biggl( A_{||}(x)\nabla_{||} u(x,t, \lambda)\cdot\nabla_{||}\overline{\psi(x,t)}-u(x,t, \lambda)\partial_t\overline{\psi(x,t)}\biggr )\, dxdt\notag\\
&&=\sum_{j=1}^{n+1}\int_{\mathbb R^{n+1}}  A_{n+1,j}(x)\partial_{x_j}\partial_\lambda u(x,t, \lambda)\overline{\psi(x,t)}\, dxdt\notag\\
&&-\sum_{i=1}^{n}\int_{\mathbb R^{n+1}}  A_{i,n+1}(x)\partial_\lambda u(x,t, \lambda)\partial_{x_i}\overline{\psi(x,t)}\, dxdt.
\end{eqnarray}
In this sense, and under these assumptions, \eqref{block++} holds on cross sections $\lambda=$ constant.

\subsection{Resolvents and a parabolic Hodge decomposition associated to $\mathcal{H}_{||}$} Recall the function space $\mathbb H={\mathbb H}(\mathbb R^{n+1},\mathbb C)$ introduced in \eqref{hsapace}. We let $\mathbb H^\ast={\mathbb H}^\ast(\mathbb R^{n+1},\mathbb C)$ be the space dual to $\mathbb H$, with norm $||\cdot||_{\mathbb H^\ast}$,  and we let $\langle\cdot,\cdot\rangle_{\mathbb H^\ast}:\mathbb H^\ast\times\mathbb H\to\mathbb C $ denote the duality pairing. We let
$\bar{\mathbb H}=\bar{\mathbb H}(\mathbb R^{n+1},\mathbb C)$ be the closure of $C_0^\infty(\mathbb R^{n+1},\mathbb C)$ with respect to the norm
\begin{eqnarray*}
	\|f\|_{\bar{\mathbb H}}:=\|f\|_{\mathbb H}+\|f\|_2.
\end{eqnarray*}
 We let $\bar{\mathbb H}^\ast=\bar{\mathbb H}^\ast(\mathbb R^{n+1},\mathbb C)$ be the space dual to $\bar {\mathbb H}$, with norm $||\cdot||_{\bar {\mathbb H}^\ast}$, and we let $\langle\cdot,\cdot\rangle_{\bar {\mathbb H}^\ast}:\bar{\mathbb H}^\ast\times\bar{\mathbb H}\to\mathbb C $ denote the duality pairing. Let  $B:\mathbb H\times  \mathbb H\to\mathbb R$ be defined as
       \begin{eqnarray}\label{form2-}
  B(u,\phi):= \int_{\mathbb R^{n+1}}
      (A_{||}\nabla_{||} u\cdot\nabla_{||}\bar \phi-D_{1/2}^tu\overline{H_tD_{1/2}^t\phi})\, dxdt,
      \end{eqnarray}
      and let, for $\delta\in (0,1)$, $ B_\delta:\mathbb H\times  \mathbb H\to\mathbb R$ be defined as
       \begin{eqnarray}\label{form2}
  B_\delta(u,\phi)&:=&\int_{\mathbb R^{n+1}}A_{||}\nabla_{||} u\cdot\overline{\nabla_{||}(I+\delta H_t)\phi}\, dxdt\notag\\
  &&-\int_{\mathbb R^{n+1}}D_{1/2}^tu\overline{H_tD_{1/2}^t(I+\delta H_t)\phi}\, dxdt.
      \end{eqnarray}

     \begin{definition}\label{de1} Let $F\in {\mathbb H}^\ast(\mathbb R^{n+1},\mathbb C)$. We say that a function
      $u\in {\mathbb H}(\mathbb R^{n+1},\mathbb C)$ is a (weak) solution to the equation $\mathcal{H}_{||}u=F$,
in $\mathbb R^{n+1}$, if
     \begin{eqnarray*}\label{eq4-}
B(u,\phi)=\langle F,\phi\rangle_{\mathbb H^\ast},
    \end{eqnarray*}
    whenever $\phi\in {\mathbb H}(\mathbb R^{n+1},\mathbb C)$.
      \end{definition}

          \begin{definition}\label{de2} Let $\lambda>0$ be given. Let $F\in \bar{\mathbb H}^\ast(\mathbb R^{n+1},\mathbb C)$. We say that a function
      $u\in\bar{\mathbb{H}}(\mathbb R^{n+1},\mathbb C)$ is a (weak) solution to the equation $u+\lambda^2\mathcal{H}_{||}u=F$,
in $\mathbb R^{n+1}$, if
     \begin{eqnarray*}\label{eq4-a}
\int_{\mathbb R^{n+1}}u\bar \phi\, dxdt+\lambda^2 B(u,\phi)=\langle F,\phi\rangle_{\bar{\mathbb H}^\ast},
    \end{eqnarray*}
whenever $\phi\in \bar{\mathbb H}(\mathbb R^{n+1},\mathbb C)$.
      \end{definition}

\begin{lemma}\label{parahodge} Consider the operator $\mathcal{H}_{||}=\partial_t-\div_{||} A_{||}\nabla_{||}$ and assume that $A$ satisfies \eqref{eq3}, \eqref{eq4}. Let  $F\in {\mathbb H}^\ast(\mathbb R^{n+1},\mathbb C)$. Then there exists a weak solution to the equation $\mathcal{H}_{||}u=F$, in $\mathbb R^{n+1}$, in the sense of Definition \ref{de1}. Furthermore,
$$||u||_{\mathbb H}\leq c||F||_{\mathbb H^\ast},$$
for some constant $c$ depending only on $n$ and $\Lambda$. The solution is unique up to a constant.
\end{lemma}
\begin{proof} This is essentially  Lemma 2.6 in \cite{N}. Let
$\phi_\delta:=(I+\delta H_t)\phi$, $\phi\in {\mathbb H}(\mathbb R^{n+1},\mathbb C)$, $\delta\in (0,1)$.  Then
$$|\langle F,\phi_\delta\rangle_{\mathbb H^\ast}|\leq c||F||_{\mathbb H^\ast}||\phi||_{\mathbb H}.$$
Consider the sesquilinear form $B_\delta(\cdot,\cdot)$ introduced in \eqref{form2}. If $\delta=\delta(n,\Lambda)$ is small enough, then $B_\delta(\cdot,\cdot)$ is a sesquilinear, bounded, coercive form on $\mathbb H\times \mathbb H$. Hence, using the Lax-Milgram theorem we see that there exists a unique
$u\in {\mathbb H}$ such that
$$B(u,\phi_\delta)= B_\delta(u,\phi)=\langle F,\phi_\delta\rangle_{\mathbb H^\ast},$$
for all $\phi\in \mathbb H$. Using that $(I+\delta H_t)$ is invertible on $\mathbb H$, if $0<\delta\ll 1$ is small enough, we can conclude that
$$B(u,\psi)= \langle F,\psi\rangle_{\mathbb H^\ast},$$
whenever $\psi\in {\mathbb H}$. The bound $||u||_{\mathbb H}\leq c||F||_{\mathbb H^\ast}$ follows readily. This completes the existence and quantitative part of the lemma. The statement concerning uniqueness follows immediately.\end{proof}

\begin{lemma}\label{parahodge+} Let $\lambda>0$ be given. Consider the operator $\mathcal{H}_{||}=\partial_t-\div_{||} A_{||}\nabla_{||}$ and assume that $A$ satisfies \eqref{eq3}, \eqref{eq4}. Let  $F\in \bar{\mathbb H}^\ast(\mathbb R^{n+1},\mathbb C)$. Then there exists a weak solution to the equation $u+\lambda^2\mathcal{H}_{||}u=F$, in $\mathbb R^{n+1}$, in the sense of Definition \ref{de2}. Furthermore,
$$||u||_{2}+||\lambda\nabla_{||} u||_{2}+||\lambda D_{1/2}^tu||_{2}\leq c||F||_{\bar{\mathbb H}^\ast},$$
for some constant $c$ depending only on $n$ and $\Lambda$. The solution is unique.
\end{lemma}
\begin{proof} See the proof of Lemma 2.7 in \cite{N}.\end{proof}

    \begin{remark}\label{rem1} Definition \ref{de1}, Definition \ref{de2}, Lemma \ref{parahodge}, and Lemma \ref{parahodge+}, all have analogous formulations for the
    operator  $\mathcal{H}_{||}^\ast$.
    \end{remark}
      \begin{remark}\label{rem2} Let $\lambda>0$ be given. Consider the operator $\mathcal{H}_{||}=\partial_t-\div_{||} A_{||}\nabla_{||}$. Let  $F\in \bar{\mathbb H}^\ast(\mathbb R^{n+1},\mathbb C)$. By Lemma \ref{parahodge+} the equation $u+\lambda^2\mathcal{H}_{||}u=F$ has a unique weak solution $u\in\bar{\mathbb H}$. From now on we will denote this solution by $ \mathcal{E}_\lambda F$. In the case of the operator $\mathcal{H}_{||}^\ast$ we denote the corresponding solution by $ \mathcal{E}_\lambda^\ast F$. In this sense $\mathcal{E}_\lambda=(I+\lambda^2\mathcal{H}_{||})^{-1}$ and $\mathcal{E}_\lambda^\ast=(I+\lambda^2\mathcal{H}_{||}^\ast)^{-1}$.
    \end{remark}

    Consider $\lambda>0$ fixed, let $|h|\ll \lambda$ and consider $F\in \bar{\mathbb H}^\ast(\mathbb R^{n+1},\mathbb C)$. By definition,
      \begin{eqnarray}\label{cla1}
\int_{\mathbb R^{n+1}}\mathcal{E}_{\lambda+h}F\bar \phi\, dxdt+(\lambda+h)^2 B(\mathcal{E}_{\lambda+h}F,\phi)&=&\langle F,\phi\rangle_{\bar{\mathbb H}^\ast},\notag\\
\int_{\mathbb R^{n+1}}\mathcal{E}_{\lambda}F\bar \phi\, dxdt+\lambda^2 B(\mathcal{E}_{\lambda}F,\phi)&=&\langle F,\phi\rangle_{\bar{\mathbb H}^\ast},
    \end{eqnarray}
 for all $\phi\in \bar{\mathbb H}(\mathbb R^{n+1},\mathbb C)$. We let $\mathcal{D}_{\lambda}^hF:=\mathcal{E}_{\lambda+h}F-\mathcal{E}_{\lambda}F$. \eqref{cla1} implies
      \begin{eqnarray}\label{cla1+}
\int_{\mathbb R^{n+1}}\mathcal{D}_{\lambda}^hF\bar \phi_\delta\, dxdt+\lambda^2 B(\mathcal{D}_{\lambda}^hF,\phi_\delta)=-h(2\lambda+h)B(\mathcal{E}_{\lambda+h}F,\phi_\delta)
    \end{eqnarray}
     for all $\phi\in \bar{\mathbb H}(\mathbb R^{n+1},\mathbb C)$, $\phi_\delta:=(I+\delta H_t)\phi$. Again, arguing as in the proof of Lemma \ref{parahodge+} we see, if $\delta=\delta(n,\Lambda)$, $0<\delta\ll 1$ is small enough and as $\mathcal{D}_{\lambda}^hF\in \bar{\mathbb H}(\mathbb R^{n+1},\mathbb C)$, that
         \begin{eqnarray}\label{cla2}
||\mathcal{D}_{\lambda}^hF||_{2}+||\lambda\nabla_{||} \mathcal{D}_{\lambda}^hF||_{2}+||\lambda D_{1/2}^t\mathcal{D}_{\lambda}^hF||_{2}\leq c|h|||\mathcal{E}_{\lambda+h}F||_2\leq c|h|||F||_{\bar{\mathbb H}^\ast},
    \end{eqnarray}
    where $c$ is independent of $h$. Hence
             \begin{eqnarray}\label{cla3}
\lim_{h\to 0}\mathcal{D}_{\lambda}^hF=\lim_{h\to 0}\bigl (\mathcal{E}_{\lambda+h}F-\mathcal{E}_{\lambda}F\bigr)=0
    \end{eqnarray}
    in the sense that
             \begin{eqnarray}\label{cla4}
||\mathcal{D}_{\lambda}^hF||_{2}+||\lambda\nabla_{||} \mathcal{D}_{\lambda}^hF||_{2}+||\lambda D_{1/2}^t\mathcal{D}_{\lambda}^hF||_{2}\to 0\mbox{ as }h\to 0.
    \end{eqnarray}
    Similarly,
          \begin{eqnarray}\label{cla5}
\int_{\mathbb R^{n+1}}h^{-1}\mathcal{D}_{\lambda}^hF\bar \phi_\delta\, dxdt+\lambda^2 B(h^{-1}\mathcal{D}_{\lambda}^hF,\phi_\delta)=-(2\lambda+h)B(\mathcal{E}_{\lambda+h}F,\phi_\delta)
    \end{eqnarray}
    and hence
    \begin{eqnarray}\label{cla6}
||h^{-1}\mathcal{D}_{\lambda}^hF||_{2}+||\lambda\nabla_{||} (h^{-1}\mathcal{D}_{\lambda}^hF)||_{2}+||\lambda D_{1/2}^t(h^{-1}\mathcal{D}_{\lambda}^hF)||_{2}\leq c||F||_{\bar{\mathbb H}^\ast},
    \end{eqnarray}
    where $c$ is independent of $h$. Using \eqref{cla6}, \eqref{cla5} and \eqref{cla4} we see, as $\lambda$ is fixed, that
    \begin{eqnarray}\label{cla7}
\lim_{h\to 0}h^{-1}\mathcal{D}_{\lambda}^hF=:\mathcal{G}_\lambda F\mbox{ weakly in }\bar{\mathbb H}(\mathbb R^{n+1},\mathbb C),
    \end{eqnarray}
that \eqref{cla6} holds with $h^{-1}\mathcal{D}_{\lambda}^hF$ replaced by $\mathcal{G}_\lambda F$ and that
          \begin{eqnarray}\label{cla8}
\int_{\mathbb R^{n+1}}\mathcal{G}_\lambda F\bar \phi\, dxdt+\lambda^2 B(\mathcal{G}_\lambda F,\phi)=-2\lambda B(\mathcal{E}_\lambda F,\phi)=-2\lambda\langle\mathcal{H}_{||}\mathcal{E}_\lambda F,\phi\rangle_{\bar{\mathbb H}^\ast}
    \end{eqnarray}
     whenever $\phi\in \bar{\mathbb H}(\mathbb R^{n+1},\mathbb C)$. We define
     \begin{eqnarray}\label{cla8+}\partial_\lambda\mathcal{E}_\lambda F:=\mathcal{G}_\lambda F
     \end{eqnarray} and hence
               \begin{eqnarray}\label{cla9}
               \partial_\lambda\mathcal{E}_\lambda F=-2\lambda\mathcal{E}_\lambda \mathcal{H}_{||}\mathcal{E}_\lambda F
    \end{eqnarray}
    in the sense of \eqref{cla8}. Furthermore, if $F=f\in\mathbb H(\mathbb R^{n+1},\mathbb C)$ then
                \begin{eqnarray}\label{cla10}
         \langle \mathcal{H}_{||}\mathcal{E}_\lambda f,\phi\rangle_{\bar{\mathbb H}^\ast}-\langle\mathcal{E}_\lambda \mathcal{H}_{||} f,\phi\rangle_{\bar{\mathbb H}^\ast}&=&\langle \mathcal{H}_{||}\mathcal{E}_\lambda f,\phi\rangle_{\bar{\mathbb H}^\ast}-\langle \mathcal{H}_{||} f,\mathcal{E}_\lambda^\ast\phi\rangle_{\bar{\mathbb H}^\ast}\notag\\
         &=&B(\mathcal{E}_\lambda f,\phi)-B(f,\mathcal{E}_\lambda^\ast\phi)=0,
    \end{eqnarray}
    and hence $\mathcal{H}_{||}$ and $\mathcal{E}_\lambda$ commute in this sense. Furthermore, as $A$ is independent of $t$ we can, by arguing similarly, conclude that if $f\in\mathbb H(\mathbb R^{n+1},\mathbb C)$, then
                    \begin{eqnarray}\label{cla11}
        \langle \partial_t\mathcal{E}_\lambda f,\phi\rangle_{\bar{\mathbb H}^\ast}-\langle\mathcal{E}_\lambda \partial_t f,\phi\rangle_{\bar{\mathbb H}^\ast}=0=\langle \mathcal{L}_{||}\mathcal{E}_\lambda f,\phi\rangle_{\bar{\mathbb H}^\ast}-\langle\mathcal{E}_\lambda \mathcal{L}_{||} f,\phi\rangle_{\bar{\mathbb H}^\ast}
    \end{eqnarray}
    and hence $\partial_t$ and $\mathcal{E}_\lambda$, and $\mathcal{L}_{||}$ and $\mathcal{E}_\lambda$, commute in this sense. In particular, if
    $F=f\in\mathbb H(\mathbb R^{n+1},\mathbb C)$ then
    \begin{eqnarray}\label{cla12}
               \partial_\lambda\mathcal{E}_\lambda f=-2\lambda\mathcal{E}_\lambda^2 \mathcal{H}_{||} f
    \end{eqnarray}
    in the sense of \eqref{cla8}.

\subsection{Estimates of resolvents}  We here collect a set of the estimates for $\mathcal{E}_\lambda f$ and $\mathcal{E}_\lambda^\ast f$ to be used in the next section.

 \begin{lemma}\label{le8-}  Let $\lambda>0$ be given. Consider the operator $\mathcal{H}_{||}=\partial_t-\div_{||} A_{||}\nabla_{||}$ and assume that $A$ satisfies \eqref{eq3}, \eqref{eq4}. Let $\Theta_\lambda$ denote any of  the operators
       \begin{eqnarray*}
       &&\mbox{$\mathcal{E}_\lambda$, $\lambda \nabla_{||}\mathcal{E}_\lambda$, $\lambda D_{1/2}^t\mathcal{E}_\lambda$},
       \end{eqnarray*}
       or
   \begin{eqnarray*}
       &&\mbox{$\lambda\mathcal{E}_\lambda D_{1/2}^t$, $\lambda^2 \nabla_{||}\mathcal{E}_\lambda D_{1/2}^t $, $\lambda^2 D_{1/2}^t\mathcal{E}_\lambda D_{1/2}^t$},
       \end{eqnarray*}
       and let $\tilde \Theta_\lambda$ denote any of  the operators
         \begin{eqnarray*}
       &&\mbox{$\lambda\mathcal{E}_\lambda\div_{||} $, $\lambda^2 \nabla_{||}\mathcal{E}_\lambda\div_{||} $, $\lambda^2 D_{1/2}^t\mathcal{E}_\lambda\div_{||}$}.
       \end{eqnarray*}
       Then there exist $c$, depending only on $n, \Lambda$, such that
           \begin{eqnarray*}
       (i)&&\int_{\mathbb R^{n+1}}\ |\Theta_\lambda f(x,t)|^2\, dxdt\leq c\int_{\mathbb R^{n+1}}\ |f(x,t)|^2\, dxdt,\notag\\
       (ii)&&\int_{\mathbb R^{n+1}}\ |\tilde \Theta_\lambda  {\bf f}(x,t)|^2\, dxdt\leq
       c\int_{\mathbb R^{n+1}}\ |{\bf f}(x,t)|^2\, dxdt,
       \end{eqnarray*}
      whenever $f\in L^2(\mathbb R^{n+1},\mathbb C)$, ${\bf f}\in L^2(\mathbb R^{n+1},\mathbb C^{n})$.\end{lemma}
\begin{proof}  This is Lemma 2.11 in \cite{N}.\end{proof}

\begin{lemma} \label{ilem2--a} Let $\lambda>0$ be given. Consider the operator $\mathcal{H}_{||}=\partial_t-\div_{||} A_{||}\nabla_{||}$ and assume that $A$ satisfies \eqref{eq3}, \eqref{eq4}. Let  $A_{n+1}^{||}:=(A_{1,n+1},...,A_{n,n+1})$,
     \begin{eqnarray*}
\mathcal{U}_\lambda :=\lambda\mathcal{E}_\lambda \div_{||},
\end{eqnarray*}
and let
     \begin{eqnarray*}
\mathcal{R}_\lambda :=\mathcal{U}_\lambda A_{n+1}^{||}-(\mathcal{U}_\lambda A_{n+1}^{||})\P_\lambda,
\end{eqnarray*}
where $\P_\lambda$ be a parabolic approximation of the identity. Then there exists a constant $c$, depending only on $n$, $\Lambda$, such that
        \begin{eqnarray*}
 ||\mathcal{R}_\lambda f||_2\leq c(||\lambda \nabla f||_2+||\lambda^2 \partial_tf||_2),
       \end{eqnarray*}
      whenever $f\in C_0^\infty(\mathbb R^{n+1},\mathbb C)$.
     \end{lemma}
     \begin{proof}  The lemma is a consequence of Lemma 2.27  in \cite{N}.\end{proof}

\begin{lemma} \label{ilem2--} Let $\lambda>0$ be given. Consider the operator $\mathcal{H}_{||}=\partial_t-\div_{||} A_{||}\nabla_{||}$ and assume that $A$ satisfies \eqref{eq3}, \eqref{eq4}. Let  $A_{n+1}^{||}:=(A_{1,n+1},...,A_{n,n+1})$,
     \begin{eqnarray*}
\mathcal{U}_\lambda :=\lambda\mathcal{E}_\lambda \div_{||},
\end{eqnarray*}
and consider $\mathcal{U}_\lambda A_{n+1}^{||}$. Then there exists a constant $c$, depending only on $n$, $\Lambda$, such that
\begin{eqnarray*}\label{crucacar+}\int_0^{l(Q)}\int_Q|\mathcal{U}_\lambda  A_{n+1}^{||}|^2\frac {dxdtd\lambda}\lambda\leq c|Q|,
\end{eqnarray*}
for all cubes $Q\subset\mathbb R^{n+1}$.
     \end{lemma}
     \begin{proof}  This is Lemma 3.1 in \cite{N}.\end{proof}

     \begin{remark}\label{lloc} For the details of the proof of Lemma \ref{ilem2--a} and Lemma \ref{ilem2--} we refer to \cite{N}. We here simply note that for   $\lambda$ fixed,
$(\mathcal{U}_\lambda A_{n+1}^{||})$ (and $\R_\lambda 1$) exists as an element in $L^2_{\mbox{loc}}(\mathbb R^{n+1},\mathbb C)$. Indeed, let $Q_R$ be the parabolic cube on $\mathbb R^{n+1}$ with center at $(0,0)$ and with size determined by $R$. Writing
$$\mathcal{U}_\lambda A_{n+1}^{||}=\mathcal{U}_\lambda A_{n+1}^{||}1_{2Q_R}+ \mathcal{U}_\lambda A_{n+1}^{||}1_{\mathbb R^{n+1}\setminus 2Q_R}, $$
and using Lemma \ref{le8-}  we see that
$$||\mathcal{U}_\lambda (A_{n+1}^{||}1_{2Q_R})1_{Q_R}||_2\leq c||A||_\infty R^{(n+2)/2}.$$
Furthermore, by the off-diagonal estimates for $\mathcal{U}_\lambda$ proved in Lemma 2.17 in \cite{N} it follows that also
$$||\mathcal{U}_\lambda (A_{n+1}^{||}1_{\mathbb R^{n+1}\setminus 2Q_R})1_{Q_R}||_2\leq c||A||_\infty R^{(n+2)/2}.$$

     \end{remark}

     \begin{theorem}\label{thm1}  Consider the operators $\mathcal{H}_{||}=\partial_t+\mathcal{L}_{||}=\partial_t-\div_{||} A_{||}\nabla_{||}$, $\mathcal{H}_{||}^\ast=-\partial_t+\mathcal{L}_{||}^\ast=-\partial_t-\div_{||} A_{||}^\ast\nabla_{||}$, and assume that $A$ satisfies \eqref{eq3}, \eqref{eq4}.
Then there exists a constant $c$, $1\leq c<\infty$, depending only on $n$, $\Lambda$, such that
        \begin{eqnarray}\label{ea1}
      |||\lambda \mathcal{E}_\lambda \mathcal{H}_{||}f|||+|||\lambda \mathcal{E}_\lambda^\ast\mathcal{H}_{||}^\ast f|||\leq c||\mathbb D f||_2,
      \end{eqnarray}
      and
      \begin{eqnarray}\label{ea2}
               (i)&&|||\partial_\lambda\mathcal{E}_\lambda f|||+|||\partial_\lambda\mathcal{E}_\lambda^\ast f|||\leq c||\mathbb Df||_2,\notag\\
               (ii)&&|||\lambda\partial_t\mathcal{E}_\lambda f|||+|||\lambda\partial_t\mathcal{E}_\lambda^\ast f|||\leq c||\mathbb Df||_2,\notag\\
       (iii)&&|||\lambda\mathcal{E}_\lambda \mathcal{L}_{||} f|||+|||\lambda\mathcal{E}_\lambda^\ast \mathcal{L}_{||}^\ast f|||\leq c||\mathbb Df||_2,\notag\\
       (iv)&&|||\lambda \mathcal{L}_{||}\mathcal{E}_\lambda  f|||+|||\lambda \mathcal{L}_{||}^\ast\mathcal{E}_\lambda^\ast f|||\leq c||\mathbb Df||_2,
      \end{eqnarray}
      whenever $f\in\mathbb H(\mathbb R^{n+1},\mathbb C)$. \end{theorem}
      \begin{proof}  \eqref{ea1} is Theorem 1.17 in \cite{N}, \eqref{ea2} $(i)-(iv)$ is Corollary 1.18 in \cite{N}. However, as the proof of Corollary 1.18 in \cite{N} is presented in a slightly formal manner we here include the proof of the inequalities in \eqref{ea2} clarifying details. We only supply the proof in the case of $ \mathcal{H}_{||}$. To prove $(i)$ we note that $\partial_\lambda\mathcal{E}_\lambda f$ is defined as in \eqref{cla8+} and that we have, using \eqref{cla12}, $\partial_\lambda\mathcal{E}_\lambda f=-2\lambda\mathcal{E}_\lambda^2 \mathcal{H}_{||} f$
    in the sense of \eqref{cla8}. Hence $(i)$ follows from \eqref{ea1}. To prove $(ii)$ we note that $\partial_t$ and $\mathcal{E}_\lambda$ commute in the sense discussed above, see \eqref{cla11}, and that
               \begin{eqnarray*}\label{ea1}
 \lambda\mathcal{E}_\lambda \partial_t f=\lambda\mathcal{E}_\lambda  \mathcal{H}_{||} f-\lambda\mathcal{E}_\lambda \mathcal{L}_{||}f. \end{eqnarray*}
              Hence, using \eqref{ea1} we see that
                          \begin{eqnarray*}\label{ea1}
         |||\lambda\partial_t\mathcal{E}_\lambda f|||\leq c||\mathbb Df||_2+|||\lambda\mathcal{E}_\lambda \mathcal{L}_{||}f|||.
      \end{eqnarray*}
      Therefore, to prove $(ii)$ it suffices to prove $(iii)$. To prove $(iii)$, we let $ f\in \mathbb H(\mathbb R^{n+1},\mathbb C)$ and put $g=A_{||}\nabla_{||} f$. Using Lemma \ref{parahodge}  we then see that there exists a weak solution $u$ to the equation
     \begin{eqnarray}\label{ea1uua}\mbox{$\div_{||}(g)= \mathcal{H}_{||}u$ such that $||u||_{\mathbb H}\leq c||g||_2$}.
      \end{eqnarray}
In particular,
           \begin{eqnarray}\label{ea1}
           \lambda\mathcal{E}_\lambda \mathcal{L}_{||}f= \lambda \mathcal{E}_\lambda  \mathcal{H}_{||}u.
      \end{eqnarray}
      Hence, again using Theorem \ref{thm1} we see that
                    \begin{eqnarray}\label{ea1uub}
        |||\lambda\mathcal{E}_\lambda \mathcal{L}_{||}f|||\leq c||\mathbb D u||_2.
      \end{eqnarray}
      $(iii)$ now follows by combining \eqref{ea1uua} and \eqref{ea1uub}. To prove $(iv)$ we simply note that $\mathcal{L}$ and $\mathcal{E}_\lambda$ commute in the sense of \eqref{cla11}, and hence $(iv)$ follows from the argument in $(iii)$. This completes the proof of \eqref{ea2} $(i)-(iv)$.
       \end{proof}

       \subsection{Remark on the Kato problem for parabolic equations}
       \label{kato}
       In Section 5 in \cite{N} implications of  two of the results proved in \cite{N}, Theorem 1.17 and Theorem 1.19 in \cite{N}, for  Kato square root problems related to  the operator $\partial_t+\mathcal{L}_{||}$ (in \cite{N} this operator is denoted $\partial_t+\mathcal{L}$), as well as generalization of the generalizations of these results to  operators $\partial_t-\div A(x,t)\nabla $, i.e., to operators with time-dependent coefficients, are discussed. The discussion in the section is essentially flawless but author neglects to properly state that the  Kato square root problem for the operator $\partial_t+\mathcal{L}_{||}$ is in fact solved in \cite{N}. Indeed, the core of the approach in \cite{N} is the observation that
       $\partial_t+\mathcal{L}_{||}$ can be realized as an operator $\mathbb H\to \mathbb H^*$ via the sesquilinear form $B(u,\psi)$ introduced in \eqref{form2-}:
\begin{eqnarray*}
 \langle (\partial_t+\mathcal{L}_{||}) u, \psi \rangle := B(u,\psi),\ u,\psi \in \mathbb H.
\end{eqnarray*}
By the arguments in \cite{N} it follows, see also Lemma \ref{parahodge+} above, that if $\theta \in \mathbb C$ with $ \mbox{Re }\theta > 0$, then $$\theta + \partial_t+\mathcal{L}_{||}: \mathcal{D}(\partial_t+\mathcal{L}_{||}) \to L^2(\mathbb R^{n+1},\mathbb C)$$ is bijective and the resolvent satisfies the estimate $$\|(\theta + (\partial_t+\mathcal{L}_{||}))^{-1} f\|_2 \leq \frac 1{\mbox{Re } \theta }\|f\|_2.$$ In particular,  $\partial_t+\mathcal{L}_{||}$, with maximal domain $\mathcal{D}(\partial_t+\mathcal{L}_{||}) = \{u \in \mathbb H : (\partial_t+\mathcal{L}_{||}) u \in L^2(\mathbb R^{n+1},\mathbb C) \}$ in $L^2(\mathbb R^{n+1},\mathbb C)$, is maximal accretive and, see also the discussion in Section 5 in \cite{N},
$\partial_t+\mathcal{L}_{||}$ is sectorial of angle $<{\pi}/{2}$ and there is a square root $\sqrt{\partial_t+\mathcal{L}_{||}}$ abstractly defined by functional calculus. Furthermore, $\partial_t+\mathcal{L}_{||}$ has a bounded $H^\infty$ calculus. This is an other way of formulating the discussion in Section 5 in \cite{N} up to display (5.4) in \cite{N}. Furthermore, the inequality \begin{eqnarray}\label{est1apafin+}
       \quad||\sqrt{\partial_t+\mathcal{L}_{||}}f||_2^2\leq c\int_0^\infty \int_{\mathbb R^{n+1}}|(I+\lambda^2(\partial_t+\mathcal{L}_{||}))^{-1}\lambda (\partial_t+\mathcal{L}_{||})f|^2\, \frac {dxdtd\lambda}\lambda,
       \end{eqnarray}
       does hold for all  $f\in  C_0^\infty(\mathbb R^{n+1},\mathbb{C})$. In particular, the inequality in display (5.5) in \cite{N} is valid and this was the only point left open in \cite{N}. Based on this  we can conclude, using the main result proved in \cite{N},  that there exists a constant $c$, $1\leq c<\infty$, depending only on $n$, $\Lambda$, such that
        \begin{eqnarray}\label{ea1kkato}
      c^{-1}||\mathbb D f||_2\leq ||\sqrt{\partial_t+\mathcal{L}_{||}}f||_2\leq c||\mathbb D f||_2,
      \end{eqnarray}
      whenever $f\in\mathbb H$.

\section{Estimates in parabolic Sobolev spaces}\label{sec6}
Throughout this section we assume that $\mathcal{H}$, $\mathcal{H}^\ast$ satisfy \eqref{eq3}-\eqref{eq4} as well as \eqref{eq14+}-\eqref{eq14++}.
Using the estimates established and stated in Section \ref{sec4} and Section \ref{sec5} we in this section prove the following three lemmas.

\begin{lemma}\label{lemsl1c} Let $\Phi(f)$ be defined as in \eqref{keyestint-ex+}. Assume that $\Phi(f)<\infty$ whenever $f\in L^2(\mathbb R^{n+1},\mathbb C)$.  Then there exists a constant $c$, depending at most
     on $n$, $\Lambda$, and the De Giorgi-Moser-Nash constants, such that
\begin{eqnarray*}
||\nabla_{||} \mathcal{S}_{\lambda _0}f||_{2}
	\ \leq \  c(\Phi(f)+||f||_2+||N_{\ast\ast}(\partial_\lambda \mathcal{S}_{\lambda}f)||_2),
\end{eqnarray*}
whenever $f\in L^2(\mathbb R^{n+1},\mathbb C)$,  $\lambda _0>0$.
\end{lemma}

\begin{lemma}\label{lemsl1+} Let $\Phi(f)$ be defined as in \eqref{keyestint-ex+}. Assume that $\Phi(f)<\infty$ whenever $f\in L^2(\mathbb R^{n+1},\mathbb C)$.   Then there exists a constant $c$, depending at most
     on $n$, $\Lambda$, and the De Giorgi-Moser-Nash constants, such that
\begin{eqnarray*}
||\mathbb D_{n+1}\mathcal{S}_{\lambda _0}f||_{2}^2
	\ \leq \  c(\Phi(f)+||f||_2),
\end{eqnarray*}
whenever $f\in L^2(\mathbb R^{n+1},\mathbb C)$,  $\lambda _0> 0$.
\end{lemma}

\begin{lemma}\label{lemsl1+k}  There exists a constant $c$, depending at most
     on $n$, such that
\begin{eqnarray*}
||H_tD_{1/2}^t\mathcal{S}_{\lambda _0}f||_{2}
	\ \leq \  c(||\mathbb D_{n+1}\mathcal{S}_{\lambda _0}f||_{2}+||\nabla_{||} \mathcal{S}_{\lambda _0}f||_{2}),
\end{eqnarray*}
 whenever $f\in L^2(\mathbb R^{n+1},\mathbb C)$,  $\lambda _0> 0$.
\end{lemma}

The proofs of Lemma \ref{lemsl1c}-Lemma \ref{lemsl1+k} are given below.

\subsection{Proof of Lemma \ref{lemsl1c}} Throughout the proof we can, without loss of generality, assume that
$f\in C_0^\infty(\mathbb R^{n+1},\mathbb C)$.  Let $\lambda _0>0$ be fixed. To prove the lemma it suffices to estimate
 \begin{eqnarray*}
I:=\int_{\mathbb R^{n+1}}\bar {\bf g}\cdot\nabla_{||} {\mathcal{S}_{\lambda_0}f}\, dxdt,
\end{eqnarray*}
where ${\bf g}\in C_0^\infty(\mathbb R^{n+1},\mathbb C^n)$ and $||{\bf  g}||_2=1$. Given $f\in C_0^\infty(\mathbb R^{n+1},\mathbb C)$, we note, see Lemma \ref{appf}, that $\mathcal{S}_{\lambda_0}f\in {\mathbb H}(\mathbb R^{n+1},\mathbb C)\cap L^2(\mathbb R^{n+1},\mathbb C)$. Hence, using Lemma \ref{parahodge},
 \begin{eqnarray*}
I&=&\int_{\mathbb R^{n+1}}A_{||}\nabla_{||} \mathcal{S}_{\lambda_0}f\cdot\overline{\nabla_{||}v}\, dxdt+\int_{\mathbb R^{n+1}}H_tD_{1/2}^t(\mathcal{S}_{\lambda_0}f)\overline{D_{1/2}^t(v)}\, dxdt,
\end{eqnarray*}
for a function $v\in\mathbb H=\mathbb H(\mathbb R^{n+1},\mathbb C)$ which satisfies
 \begin{eqnarray*}
 ||v||_{\mathbb H}\leq c||{\bf g}||_2,
 \end{eqnarray*}
for some constant $c$ depending only on $n$ and $\Lambda$. Let
 \begin{eqnarray*}
I_1&:=&\int_{\mathbb R^{n+1}}A_{||}\nabla_{||} \mathcal{S}_{\lambda_0}f\cdot\overline{\nabla_{||}v}\, dxdt,\notag\\
I_2&:=&\int_{\mathbb R^{n+1}}H_tD_{1/2}^t(\mathcal{S}_{\lambda_0}f)\overline{D_{1/2}^t(v)}\, dxdt.
\end{eqnarray*}
As $C_0^\infty(\mathbb R^{n+1},\mathbb C)$
is dense in $\mathbb{H}(\mathbb{R}^{n+1},\mathbb{C})$ we can in the following also assume, without loss of generality,
  that $v \in C_0^\infty(\mathbb{R}^{n+1},\mathbb{C})$. This reduction allows us to handle several boundary terms which appear when we integrate
by parts.

We first estimate $I_1$. Recall the resolvents, $\mathcal{E}_\lambda=(I+\lambda^2\mathcal{H}_{||})^{-1}$ and $\mathcal{E}_\lambda^\ast=(I+\lambda^2\mathcal{H}_{||}^\ast)^{-1}$,  introduced in Section \ref{sec5}. To start the estimate of $I_1$ we first note, applying Lemma \ref{le8-}, that
\begin{eqnarray}\label{pint1-}
&&\biggl |\int_{\mathbb R^{n+1}}A_{||}\nabla_{||} \mathcal{E}_\lambda \mathcal{S}_{\lambda+\lambda_0}f\cdot\overline{\nabla_{||}\mathcal{E}_\lambda^\ast v}\, dxdt\biggr |\leq \frac c{\lambda^2}||\mathcal{S}_{\lambda+\lambda_0}f||_2||v||_2.
\end{eqnarray}
Hence, using that
\begin{eqnarray}\label{pint1a}
	\mathcal{S}_{\lambda+\lambda_0}f-\mathcal{S}_{\lambda_0}f
		=\int_{\lambda_0}^{\lambda + \lambda_0}\partial_\sigma\mathcal{S}_{\sigma}f\, d\sigma,
\end{eqnarray}
the fact that $\Phi(f)<\infty$, Lemma \ref{appf} and that $f, v\in  C_0^\infty(\mathbb R^{n+1},\mathbb C)$, we can use
\eqref{pint1-} to conclude that
\begin{eqnarray}\label{pint1}
&&\biggl |\int_{\mathbb R^{n+1}}A_{||}\nabla_{||} \mathcal{E}_\lambda \mathcal{S}_{\lambda+\lambda_0}f\cdot\overline{\nabla_{||}\mathcal{E}_\lambda^\ast v}\, dxdt\biggr | \longrightarrow 0 \quad \mbox{ as $\lambda\to \infty$}.
\end{eqnarray}
Hence,
\begin{eqnarray}\label{pint2}
I_1&=&-\int_{0}^\infty\partial_\lambda \biggl (\int_{\mathbb R^{n+1}}A_{||}\nabla_{||} \mathcal{E}_\lambda \mathcal{S}_{\lambda+\lambda_0}f\cdot\overline{\nabla_{||}\mathcal{E}_\lambda^\ast v}\, dxdt\biggr )d\lambda.
\end{eqnarray}
Consider $\lambda>0$, $\lambda_0>0$ fixed, let $|h|\ll \min\{\lambda_0,\lambda\}$. Then
\begin{eqnarray}\label{pint2rig}
 &&\int_{\mathbb R^{n+1}}A_{||}\nabla_{||} \mathcal{E}_{\lambda+h} \mathcal{S}_{\lambda+\lambda_0+h}f\cdot\overline{\nabla_{||}\mathcal{E}_{\lambda+h}^\ast v}\, dxdt\notag\\
 &&-\int_{\mathbb R^{n+1}}A_{||}\nabla_{||} \mathcal{E}_{\lambda} \mathcal{S}_{\lambda+\lambda_0}f\cdot\overline{\nabla_{||}\mathcal{E}_{\lambda+h}^\ast v}\, dxdt=T_1^h+T_2^h+T_3^h,
\end{eqnarray}
where
\begin{eqnarray}\label{pint2rig}
T_1^h&:=&\int_{\mathbb R^{n+1}}A_{||}\nabla_{||}\bigl( \mathcal{E}_{\lambda+h}-\mathcal{E}_{\lambda}\bigr ) \mathcal{S}_{\lambda+\lambda_0}f\cdot\overline{\nabla_{||}\mathcal{E}_{\lambda+h}^\ast v}\, dxdt,\notag\\
T_2^h&:=&\int_{\mathbb R^{n+1}}A_{||}\nabla_{||} \mathcal{E}_{\lambda}\mathcal{S}_{\lambda+\lambda_0}f\cdot\overline{\nabla_{||}\bigl(\mathcal{E}_{\lambda+h}^\ast v-\mathcal{E}_{\lambda}^\ast v}\bigr )\, dxdt,\notag\\
T_3^h&:=&\int_{\mathbb R^{n+1}}A_{||}\nabla_{||} \mathcal{E}_{\lambda+h} \bigl(\mathcal{S}_{\lambda+\lambda_0+h}f-\mathcal{S}_{\lambda+\lambda_0}f\bigr )\cdot\overline{\nabla_{||}\mathcal{E}_{\lambda+h}^\ast v}\, dxdt.
\end{eqnarray}
Using \eqref{cla1}-\eqref{cla8+} we see that
\begin{eqnarray}\label{pint2rig++}
\lim_{h\to 0}h^{-1}T_1^h&=&\int_{\mathbb R^{n+1}}A_{||}\nabla_{||}\partial_\lambda\mathcal{E}_{\lambda} \mathcal{S}_{\lambda+\lambda_0}f\cdot\overline{\nabla_{||}\mathcal{E}_{\lambda}^\ast v}\, dxdt,\notag\\
\lim_{h\to 0}h^{-1}T_2^h&=&\int_{\mathbb R^{n+1}}A_{||}\nabla_{||} \mathcal{E}_{\lambda}\mathcal{S}_{\lambda+\lambda_0}f\cdot\overline{\nabla_{||}\partial_\lambda\mathcal{E}_{\lambda}^\ast v}\, dxdt,\notag\\
\lim_{h\to 0}h^{-1}T_3^h&=&\int_{\mathbb R^{n+1}}A_{||}\nabla_{||} \mathcal{E}_{\lambda} \partial_\lambda\mathcal{S}_{\lambda+\lambda_0}f\cdot\overline{\nabla_{||}\mathcal{E}_{\lambda}^\ast v}\, dxdt.
\end{eqnarray}
Using these deductions we can conclude that
 \begin{eqnarray*}
I_1&=&-\int_{0}^\infty\int_{\mathbb R^{n+1}}\bigl ((A_{||}\nabla_{||} \partial_\lambda\mathcal{E}_\lambda \mathcal{S}_{\lambda+\lambda_0}f)\cdot\overline{\nabla_{||}\mathcal{E}_\lambda^\ast v}\bigr )\, dxdtd\lambda\notag\\
&&-\int_{0}^\infty\int_{\mathbb R^{n+1}} \bigl ((A_{||}\nabla_{||} \mathcal{E}_\lambda \mathcal{S}_{\lambda+\lambda_0}f)\cdot\overline{\nabla_{||}\partial_\lambda\mathcal{E}_\lambda^\ast v}\bigr )\, dxdtd\lambda\notag\\
&&-\int_{0}^\infty\int_{\mathbb R^{n+1}}\bigl ((A_{||}\nabla_{||} \mathcal{E}_\lambda\partial_\lambda \mathcal{S}_{\lambda+\lambda_0}f)\cdot\overline{\nabla_{||}\mathcal{E}_\lambda^\ast v}\bigr )\, dxdtd\lambda\notag\\
&=:&I_{11}+I_{12}+I_{13},
\end{eqnarray*}
and we emphasize that by our assumptions, and  \eqref{cla1}-\eqref{cla8+}, $I_{11}-I_{13}$ are well defined. To proceed we first note that
 \begin{eqnarray*}
 I_{11}&=&-\int_{0}^\infty\langle \mathcal{L}_{||}^\ast\mathcal{E}_\lambda^\ast v, \partial_\lambda\mathcal{E}_\lambda \mathcal{S}_{\lambda+\lambda_0}f\rangle_{\bar{\mathbb H}^\ast}\, d\lambda=-\int_{0}^\infty\langle \mathcal{E}_\lambda^\ast\mathcal{L}_{||}^\ast v, \partial_\lambda\mathcal{E}_\lambda \mathcal{S}_{\lambda+\lambda_0}f\rangle_{\bar{\mathbb H}^\ast}\, d\lambda,\notag\\
 I_{12}&=&-\int_{0}^\infty\langle \mathcal{L}_{||}\mathcal{E}_\lambda \mathcal{S}_{\lambda+\lambda_0}f, \partial_\lambda\mathcal{E}_\lambda^\ast v \rangle_{\bar{\mathbb H}^\ast}\, d\lambda=-\int_{0}^\infty\langle \mathcal{E}_\lambda \mathcal{L}_{||} \mathcal{S}_{\lambda+\lambda_0}f, \partial_\lambda\mathcal{E}_\lambda^\ast v \rangle_{\bar{\mathbb H}^\ast}\, d\lambda,
\end{eqnarray*}
by \eqref{cla11}. Let
$$J:=\int_{0}^\infty\int_{\mathbb R^{n+1}} |\mathcal{E}_\lambda \mathcal{L}_{||}  \mathcal{S}_{\lambda+\lambda_0}f|^2\, \lambda dxdtd\lambda.$$
Then, using \eqref{cla9} , the $L^2$-boundedness of $\mathcal{E}_\lambda$ and $\mathcal{E}_\lambda^\ast$, Lemma \ref{le8-}, and the square function estimates , Theorem \ref{thm1}, we see that
\begin{eqnarray*}
|I_{11}|+|I_{12}|
	&\leq& c (|||\lambda\partial_t\mathcal{S}_{\lambda+\lambda_0}f|||+J^{1/2})||v||_{\mathbb H}\notag\\
&\leq&c(\Phi(f)+||f||_2+J^{1/2})||v||_{\mathbb H},
\end{eqnarray*}
where we on the last line have used Lemma \ref{lemsl1}. Next, referring to \eqref{weak+} we have
\begin{eqnarray*}
\mathcal{L}_{||}\mathcal{S}_{\lambda+\lambda_0}f&=&\sum_{j=1}^{n+1}A_{n+1,j}D_{n+1}D_j\mathcal{S}_{\lambda+\lambda_0}f\notag\\
&&+\sum_{i=1}^{n}D_i(A_{i,n+1}D_{n+1}\mathcal{S}_{\lambda+\lambda_0}f)+\partial_t\mathcal{S}_{\lambda+\lambda_0}f
\end{eqnarray*}
in a weak sense for almost every $\lambda$. Using this, and the $L^2$-boundedness of $\mathcal{E}_\lambda$, Lemma \ref{le8-}, we see that
\begin{eqnarray*}
J \ \leq \ c(|||\lambda\nabla\partial_\lambda \mathcal{S}_{\lambda+\lambda_0}f|||^2+|||\lambda\partial_t \mathcal{S}_{\lambda+\lambda_0}f|||^2+\tilde J),
\end{eqnarray*}
where
\begin{eqnarray*}
\tilde J:=\int_{0}^\infty\int_{\mathbb R^{n+1}} |\mathcal{E}_\lambda\sum_{i=1}^{n}D_i(A_{i,n+1}\partial_\lambda \mathcal{S}_{\lambda+\lambda_0}f)|^2\, \lambda dxdtd\lambda.
\end{eqnarray*}
In particular, again using Lemma \ref{lemsl1} we see that
\begin{eqnarray*}
J\leq c(\Phi(f)+||f||_2+\tilde J).
\end{eqnarray*}
To estimate $\tilde J$, let $A_{n+1}^{||}:=(A_{1,n+1},...,A_{n,n+1})$. Then
\begin{eqnarray*}
\tilde J&=&\int_{0}^\infty\int_{\mathbb R^{n+1}} |\mathcal{E}_\lambda \div_{||}(A_{n+1}^{||}\partial_\lambda \mathcal{S}_{\lambda+\lambda_0}f)|^2\, \lambda{dxdtd\lambda}\notag\\
&=&\int_{0}^\infty\int_{\mathbb R^{n+1}} |\mathcal{U}_\lambda (A_{n+1}^{||}\partial_\lambda \mathcal{S}_{\lambda+\lambda_0}f)|^2\, \frac{dxdtd\lambda}{\lambda},
\end{eqnarray*}
where $\mathcal{U}_\lambda:=\lambda\mathcal{E}_\lambda\div_{||}$. We write
\begin{eqnarray*}
\mathcal{U}_\lambda A_{n+1}^{||}&=&\mathcal{U}_\lambda A_{n+1}^{||}-(\mathcal{U}_\lambda A_{n+1}^{||})\P_\lambda+(\mathcal{U}_\lambda A_{n+1}^{||})\P_\lambda\notag\\
&=:&\mathcal{R}_\lambda+(\mathcal{U}_\lambda A_{n+1}^{||})\P_\lambda.
\end{eqnarray*}
Then
\begin{eqnarray*}
\tilde J\leq \tilde J_{1}+\tilde J_{2},
\end{eqnarray*}
where
\begin{eqnarray*}
\tilde J_{1}&:=&\int_{0}^\infty\int_{\mathbb R^{n+1}} |\R_\lambda\partial_\lambda \mathcal{S}_{\lambda+\lambda_0}f|^2\, \frac{dxdtd\lambda}{\lambda},\notag\\
\tilde J_{2}&:=&\int_{0}^\infty\int_{\mathbb R^{n+1}} |(\mathcal{U}_\lambda A_{n+1}^{||})\P_\lambda(\partial_\lambda \mathcal{S}_{\lambda+\lambda_0}f)|^2\, \frac{dxdtd\lambda}{\lambda}.
\end{eqnarray*}
 Using Lemma \ref{ilem2--a}, and Lemma \ref{lemsl1}, we see that
\begin{eqnarray*}
\tilde J_{1}&\leq& c\int_{0}^\infty\int_{\mathbb R^{n+1}} |\nabla\partial_\lambda \mathcal{S}_{\lambda+\lambda_0}f|^2\, \lambda{dxdtd\lambda}\notag\\
&&+c\int_{0}^\infty\int_{\mathbb R^{n+1}} |\partial_t\partial_\lambda \mathcal{S}_{\lambda+\lambda_0}f|^2\, \lambda^3{dxdtd\lambda}\notag\\
&\leq&c(\Phi(f)^2+||f||_2^2).
\end{eqnarray*}
Furthermore, by the Carleson measure estimate in Lemma \ref{ilem2--} we have
\begin{eqnarray*}
\tilde J_{2}\leq c||N_\ast(\P_\lambda(\partial_\lambda \mathcal{S}_{\lambda}f))||_2^2.
\end{eqnarray*}
Finally, we note that
\begin{eqnarray*}
||N_\ast(\P_\lambda(\partial_\lambda \mathcal{S}_{\lambda}f))||_2\leq c ||M(N_{\ast\ast}(\partial_\lambda \mathcal{S}_{\lambda}f))||_2\leq
c||N_{\ast\ast}(\partial_\lambda \mathcal{S}_{\lambda}f)||_2
\end{eqnarray*}
where $M$ is the  parabolic Hardy-Littlewood maximal function.
Putting all these estimates together we can conclude that
\begin{eqnarray*}
|I_{11}|+|I_{12}|\leq \bigl(\Phi(f)+||f||_2+||N_{\ast\ast}(\partial_\lambda \mathcal{S}_{\lambda}f)||_2\bigr)||v||_{\mathbb H},
\end{eqnarray*}
which completes the estimate of $|I_{11}|+|I_{12}|$. We next estimate $I_{13}$. Integrating by parts with respect to $\lambda$ we deduce, by repeating the argument above, that
\begin{eqnarray*}
I_{13}&=&-\int_{0}^\infty\int_{\mathbb R^{n+1}}\bigl (A_{||}\nabla_{||} \mathcal{E}_\lambda\partial_\lambda \mathcal{S}_{\lambda+\lambda_0}f\cdot\overline{\nabla_{||}\mathcal{E}_\lambda^\ast v}\bigr )\, dxdtd\lambda\notag\\
&=&\int_{0}^\infty\int_{\mathbb R^{n+1}}\partial_\lambda\bigl (A_{||}\nabla_{||} \mathcal{E}_\lambda\partial_\lambda \mathcal{S}_{\lambda+\lambda_0}f\cdot\overline{\nabla_{||}\mathcal{E}_\lambda^\ast v}\bigr )\, \lambda dxdtd\lambda\notag\\
&=&\int_{0}^\infty\int_{\mathbb R^{n+1}}\bigl ((A_{||}\nabla_{||} \partial_\lambda\mathcal{E}_\lambda \partial_\lambda\mathcal{S}_{\lambda+\lambda_0}f)\cdot\overline{\nabla_{||}\mathcal{E}_\lambda^\ast v}\bigr )\, \lambda dxdtd\lambda\notag\\
&&+\int_{0}^\infty\int_{\mathbb R^{n+1}} \bigl ((A_{||}\nabla_{||} \mathcal{E}_\lambda  \partial_\lambda\mathcal{S}_{\lambda+\lambda_0}f)\cdot\overline{\nabla_{||}\partial_\lambda\mathcal{E}_\lambda^\ast v}\bigr )\, \lambda dxdtd\lambda\notag\\
&&+\int_{0}^\infty\int_{\mathbb R^{n+1}}\bigl ((A_{||}\nabla_{||} \mathcal{E}_\lambda\partial_\lambda^2 \mathcal{S}_{\lambda+\lambda_0}f)\cdot\overline{\nabla_{||}\mathcal{E}_\lambda^\ast v}\bigr )\, \lambda dxdtd\lambda\notag\\
&=:&I_{131}+I_{132}+I_{133}.
\end{eqnarray*}
By repeating the estimates above used to control $|I_{11}|+|I_{12}|$ , we see that
\begin{eqnarray*}
(|I_{131}|+|I_{132}|)^2&\leq &c\int_{0}^\infty\int_{\mathbb R^{n+1}} |\nabla\partial_\lambda^2 \mathcal{S}_{\lambda+\lambda_0}f|^2\, \lambda^3 dxdtd\lambda,\notag\\
&&+c\int_{0}^\infty\int_{\mathbb R^{n+1}} |\partial_t \partial_\lambda\mathcal{S}_{\lambda+\lambda_0}f|^2\, \lambda^3 dxdtd\lambda\notag\\
&&+c\int_{0}^\infty\int_{\mathbb R^{n+1}} |\partial_t \partial_\lambda^2 \mathcal{S}_{\lambda+\lambda_0}f|^2\, \lambda^5 dxdtd\lambda+c||N_\ast(\P_\lambda(\lambda\partial_\lambda^2 \mathcal{S}_{\lambda}f))||_2^2.
\end{eqnarray*}
Furthermore,
 \begin{eqnarray*}
I_{133}&=&\int_{0}^\infty\int_{\mathbb R^{n+1}}\bigl (A_{||}\nabla_{||} \mathcal{E}_\lambda\partial_\lambda^2 \mathcal{S}_{\lambda+\lambda_0}f\cdot\overline{\nabla_{||}\mathcal{E}_\lambda^\ast v}\bigr )\, \lambda dxdtd\lambda\notag\\
&=&-\int_{0}^\infty\int_{\mathbb R^{n+1}}\mathcal{E}_\lambda\partial_\lambda^2 \mathcal{S}_{\lambda+\lambda_0}f\overline{\mathcal{E}_\lambda^\ast \mathcal{L}_{||}^\ast v}\, \lambda dxdtd\lambda,
\end{eqnarray*}
by previous arguments. Using the $L^2$-boundedness of $\mathcal{E}_\lambda$, Lemma \ref{le8-} and the square function estimate for $\mathcal{E}_\lambda^\ast \mathcal{L}_{||}^\ast$, Theorem \ref{thm1}, we can conclude that
 \begin{eqnarray*}
|I_{133}| &\leq &c\biggl (\int_{0}^\infty\int_{\mathbb R^{n+1}} |\partial_\lambda^2 \mathcal{S}_{\lambda+\lambda_0}f|^2\, \lambda dxdtd\lambda\biggr )^{1/2}||v||_{\mathbb H}.
\end{eqnarray*}
Hence, again using Lemma \ref{lemsl1} we see that
\begin{eqnarray*}
|I_{13}|\leq c\bigl(\Phi(f)+||f||_2+||N_\ast(\P_\lambda(\lambda\partial_\lambda^2 \mathcal{S}_{\lambda}f))||_2\bigr)||v||_{\mathbb H},
\end{eqnarray*}
Again
\begin{eqnarray*}
||N_\ast(\P_\lambda(\lambda\partial_\lambda^2 \mathcal{S}_{\lambda}f))||_2\leq c||M(N_{\ast\ast}(\lambda\partial_\lambda^2 \mathcal{S}_{\lambda}f))||_2\leq
c||N_{\ast\ast}(\lambda\partial_\lambda^2 \mathcal{S}_{\lambda}f)||_2,
\end{eqnarray*}
and using  \eqref{eq14+} and Lemma \ref{le1--} we see that
\begin{eqnarray*}
||N_{\ast\ast}(\lambda\partial_\lambda^2 \mathcal{S}_{\lambda}f)||_2\leq c||N_{\ast\ast}(\partial_\lambda \mathcal{S}_{\lambda}f)||_2,
\end{eqnarray*}
after a slight redefinition of the non-tangential maximal function on the right hand side. This completes the proof of $I_1$.

We next estimate $I_2$. To start the estimate of $I_2$ we first deduce, by arguing along the lines of \eqref{pint1}-\eqref{pint2rig++}, that
\begin{eqnarray*}
I_2&=&-\int_{0}^\infty\int_{\mathbb R^{n+1}}\partial_\lambda \bigl (H_tD_{1/2}^t\mathcal{E}_\lambda \mathcal{S}_{\lambda+\lambda_0}f\cdot\overline{D_{1/2}^t\mathcal{E}_\lambda^\ast v}\bigr )\, dxdtd\lambda\notag\\
&=&-\int_{0}^\infty\int_{\mathbb R^{n+1}}(H_tD_{1/2}^t\partial_\lambda\mathcal{E}_\lambda\mathcal{S}_{\lambda+\lambda_0}f)\cdot\overline{D_{1/2}^t\mathcal{E}_\lambda^\ast v}\, dxdtd\lambda\notag\\
&&-\int_{0}^\infty\int_{\mathbb R^{n+1}} (H_tD_{1/2}^t \mathcal{E}_\lambda \mathcal{S}_{\lambda+\lambda_0}f)\cdot\overline{D_{1/2}^t\partial_\lambda\mathcal{E}_\lambda^\ast v}\, dxdtd\lambda\notag\\
&&-\int_{0}^\infty\int_{\mathbb R^{n+1}}(H_tD_{1/2}^t\mathcal{E}_\lambda\partial_\lambda \mathcal{S}_{\lambda+\lambda_0}f)\cdot\overline{D_{1/2}^t\mathcal{E}_\lambda^\ast v}\, dxdtd\lambda\notag\\
&=:&I_{21}+I_{22}+I_{23}.
\end{eqnarray*}
Using the $L^2$-boundedness of $\mathcal{E}_\lambda$ and $\mathcal{E}_\lambda^\ast$, Lemma \ref{le8-}, and the square function estimates, Theorem \ref{thm1}, that $\mathcal{H}_{||}$ commutes with $\mathcal{E}_\lambda$, $D_{1/2}^t$, and $H_tD_{1/2}^t$,  and that
$\mathcal{H}_{||}^\ast$ commutes with $\mathcal{E}_\lambda^\ast$, $D_{1/2}^t$, and $H_tD_{1/2}^t$, in both cases in the sense described above, we can as in the estimate of $|I_{11}|+|I_{12}|$ deduce that
\begin{eqnarray}\label{kau}
|I_{22}|&\leq& c|||\lambda\partial_t\mathcal{S}_{\lambda+\lambda_0}f||| \, ||v||_{\mathbb H}\leq c(\Phi(f)+||f||_2)||v||_{\mathbb H}.
\end{eqnarray}
At the final step of this deduction we have also used Lemma \ref{lemsl1}.  Integrating by parts with respect to $\lambda$ in $I_{23}$, and repeating the arguments used in the estimates  of
$|I_{21}|$ and $|I_{22}|$, it is easily seen, using Lemma \ref{lemsl1}, that
\begin{eqnarray*}
|I_{23}|&\leq& c(\Phi(f)+||f||_2)||v||_{\mathbb H}+|\tilde I_{23}|,
\end{eqnarray*}
where
\begin{eqnarray*}
\tilde I_{23}=\int_{0}^\infty\int_{\mathbb R^{n+1}}\bigl ((H_tD_{1/2}^t\mathcal{E}_\lambda\partial_\lambda^2 \mathcal{S}_{\lambda+\lambda_0}f)\cdot\overline{D_{1/2}^t\mathcal{E}_\lambda^\ast v}\bigr )\, \lambda dxdtd\lambda.
\end{eqnarray*}
However, again using Lemma \ref{le8-} and Theorem \ref{thm1}
\begin{eqnarray*}
|\tilde I_{23}|\leq |||\lambda \partial_\lambda^2 \mathcal{S}_{\lambda+\lambda_0}f||| \,
 |||\lambda\partial_t\mathcal{E}_\lambda^\ast v |||\leq c\Phi(f)||v||_{\mathbb H}.
\end{eqnarray*}
This completes the proof of the lemma.

\subsection{Proof of Lemma \ref{lemsl1+}}
To prove Lemma \ref{lemsl1+} it suffices to estimate
\begin{eqnarray*}
 \int_{\mathbb R^{n+1}} (\mathbb D_{n+1}\mathcal{S}_{\lambda _0}f)\bar g\, dxdt
\end{eqnarray*}
when $f, g\in C_0^\infty(\mathbb  R^{n+1},\mathbb C)$, $||g||_2=1$. Let in the following $\P_\lambda$ be a parabolic approximation of the identity. Then, using \eqref{uau} $(ii)$  we see that
\begin{eqnarray*}
\biggl |\int_{\mathbb R^{n+1}}(\mathbb D_{n+1}\mathcal{S}_{\lambda+\lambda _0}f) \P_\lambda \bar g\, dxdt\biggr |
	&\leq& c|| D_{1/2}^t\mathcal{S}_{\lambda+\lambda _0}f||_2|| \P_\lambda \bar g||_2 \notag\\
	&\leq& \frac{c}{\lambda^{n/2+1}} \| \partial_t \mathcal{S}_{\lambda + \lambda_0}f \|_2
	     \| \mathcal{S}_{\lambda + \lambda_0}f \|_2.
\end{eqnarray*}
Again using \eqref{pint1a}, H{\"o}lder's inequality, the fact that $\Phi(f)<\infty$,
Lemma \ref{le5} and Lemma \ref{appf}  we deduce that
\begin{eqnarray*}
\biggl |\int_{\mathbb R^{n+1}}(\mathbb D_{n+1}\mathcal{S}_{\lambda+\lambda _0}f) \P_\lambda \bar g\, dxdt\biggr |
\longrightarrow 0 \quad \mbox{ as $\lambda\to\infty$}.
\end{eqnarray*}
Hence,
\begin{eqnarray*}
 -\int_{\mathbb R^{n+1}} (\mathbb D_{n+1}\mathcal{S}_{\lambda _0}f)\bar g\, dxdt
 &=&\int_0^\infty \int_{\mathbb R^{n+1}}
 \partial_\lambda((\mathbb D_{n+1}\mathcal{S}_{\lambda+\lambda _0}f) \P_\lambda \bar g)\, dxdtd\lambda\notag\\
 &=&\int_0^\infty \int_{\mathbb R^{n+1}}(\mathbb D_{n+1}\partial_\lambda \mathcal{S}_{\lambda+\lambda _0}f) \P_\lambda \bar g\, dxdtd\lambda\notag\\
 &&+\int_0^\infty \int_{\mathbb R^{n+1}} (\mathbb D_{n+1}\mathcal{S}_{\lambda+\lambda _0}f)\partial_\lambda(\P_\lambda \bar g)\, dxdtd\lambda\notag\\
 &=:&I+II.
 \end{eqnarray*}
Note that  $\mathbb D_{n+1}=i\mathbb D^{-1}\partial_t$  and that $\partial_\lambda \P_\lambda=\mathbb D \mathcal{Q}_\lambda$ where
$\mathcal{Q}_\lambda$ is an approximation of the zero operator. To
 prove this one can use that the kernel of $\partial_\lambda \P_\lambda$ has not only zero mean
 but also first order vanishing moments if $\P$ is an even function (see also \cite[p. 366]{HL}). Using this we see that
 \begin{eqnarray*}
|II|^2&\leq&\biggl |\int_0^\infty \int_{\mathbb R^{n+1}} (\partial_t \mathcal{S}_{\lambda+\lambda _0}f) \mathcal{Q}_\lambda \bar g\, dxdtd\lambda\biggr |^2\notag\\
&\leq &c\int_{0}^\infty\int_{\mathbb R^{n+1}}|\partial_t\mathcal{S}_{\lambda  +\lambda_0}f|^2\,  \lambda{ dxdtd\lambda}\leq c(\Phi(f)+||f||_2)^2,
 \end{eqnarray*}
 by \eqref{li} and  Lemma \ref{lemsl1}. To handle $I$ we again integrate by parts with respect to $\lambda$,
 \begin{eqnarray*}
 -I&=&\int_0^\infty \int_{\mathbb R^{n+1}}(\mathbb D_{n+1}\partial_\lambda^2 \mathcal{S}_{\lambda+\lambda _0}f) \P_\lambda \bar g\, \lambda dxdtd\lambda\notag\\
 &&+\int_0^\infty \int_{\mathbb R^{n+1}}(\mathbb D_{n+1}\partial_\lambda \mathcal{S}_{\lambda+\lambda _0}f)\partial_\lambda(\P_\lambda \bar g)\, \lambda dxdtd\lambda\notag\\
 &=:&I_1+I_2.
 \end{eqnarray*}
 Arguing as above we immediately see that
  \begin{eqnarray*}
|I_2|^2&\leq&c\int_{0}^\infty\int_{\mathbb R^{n+1}}|\partial_t\partial_\lambda\mathcal{S}_{\lambda  +\lambda_0}f|^2\,  \lambda^3{ dxdtd\lambda}\leq c(\Phi(f)+||f||_2)^2.
 \end{eqnarray*}
 Focusing on $I_1$,  Lemma \ref{little2} implies
 $$|I_1|
      \leq ||| \lambda \partial_\lambda^2 \mathcal{S}_{\lambda + \lambda_0}f ||| \ ||| \lambda \mathbb{D}_{n+1} \mathcal{P}_\lambda g  |||
      \leq c ||| \lambda \partial_\lambda^2 \mathcal{S}_{\lambda + \lambda_0}f ||| \ ||| \lambda \mathbb{D} \mathcal{P}_\lambda g  |||
      \leq c \Phi(f),$$
 and the proof of the lemma is complete.

\subsection{Proof of Lemma \ref{lemsl1+k}}
 Let $K\gg 2$ be a degree of freedom and let $\phi\in C_0^\infty(\mathbb R)$ be an even function with $\phi=1$ on $(-3/2,-2/K)\cup (2/K,3/2)$ and with support in $(-2,-1/K)\cup(1/K,2)$.  Recall that the multiplier defining
$D_{1/2}^t$ is $|\tau|^{1/2}$. We write
\begin{eqnarray*}
|\tau|^{1/2}&=&|\tau|^{1/2}\phi(\tau/||(\xi,\tau)||^2)+ |\tau|^{1/2}(1-\phi)(\tau/||(\xi,\tau)||^2)\notag\\
&=&\mbox{sgn}(\tau)\frac {||(\xi,\tau)||}{|\tau|^{1/2}}\phi(\tau/||(\xi,\tau)||^2)\frac {\tau}{||(\xi,\tau)||}\notag\\
&&-\sum_{j=1}^n|\tau|^{1/2}\frac {i\xi_j}{|\xi|^2}(1-\phi)(\tau/||(\xi,\tau)||^2)i\xi_j.
\end{eqnarray*}
Hence, introducing the multipliers
\begin{eqnarray*}
m_1(\xi,\tau)&=&\mbox{sgn}(\tau)\frac {||(\xi,\tau)||}{|\tau|^{1/2}}\phi(\tau/||(\xi,\tau)||^2),\notag\\
m_{2,j}(\xi,\tau)&=&-|\tau|^{1/2}\frac {i\xi_j}{|\xi|^2}(1-\phi)(\tau/||(\xi,\tau)||^2),
\end{eqnarray*}
for $j\in\{1,...,n\}$ we can conclude the existence of kernels $L_1$, $L_{2,j}$, corresponding to $m_1$, $m_{2,j}$, such that
\begin{eqnarray*}
D_{1/2}^t=L_1\ast\mathbb D_{n+1}+c\sum_{j=1}^n L_{2,j}\ast\partial_{x_j},
\end{eqnarray*}
where $\ast$ denotes convolution. Choosing $K=K(n)$ large enough we see  that the multipliers $m_1$ and $m_{2,j}$ are bounded, and hence $L_1$ and $L_{2,j}$ are bounded operators
on $L^2(\mathbb R^{n+1},\mathbb C)$. This completes the proof of Lemma \ref{lemsl1+k}.

\section{Proof of  Theorem \ref{th0}}\label{sec7}

Assume that $\mathcal{H}$,  $\mathcal{H}^\ast$,   satisfy \eqref{eq3}-\eqref{eq4} as well as the De Giorgi-Moser-Nash estimates stated in \eqref{eq14+}-\eqref{eq14++}. Assume also that there exists a constant $C $ such that \eqref{keyestint-}
     holds whenever $f\in L^2(\mathbb R^{n+1},\mathbb C)$. To prove Theorem \ref{th0} we need to prove that there exists a constant $c$, depending at most
     on $n$, $\Lambda$,  the De Giorgi-Moser-Nash constants and $C $, such that the inequalities in \eqref{keyestint+a} $(i)-(iv)$ hold. Again, we only have to prove \eqref{keyestint+a} $(i)-(iv)$ for
      $\mathcal{S}_\lambda^{\mathcal{H}}$ as the corresponding results for $\mathcal{S}_\lambda^{\mathcal{H}^\ast}$ follow by analogy. To start the proof, we first note that
     \eqref{keyestint+a} $(i)$ is an immediate consequence of Lemma \ref{lemsl1++} $(i)$ and the assumption in \eqref{keyestint-} $(i)$. Using Lemma
     \ref{lemsl1c}, Lemma \ref{lemsl1+}, and Lemma \ref{lemsl1+k},  we see that \eqref{keyestint+a} $(i)$ and the assumptions in \eqref{keyestint-} imply that
\begin{eqnarray*}
 \sup_{\lambda>0}||\mathbb D\mathcal{S}_{\lambda}^{\mathcal{H}}f||_{2}\leq c||f||_2.
\end{eqnarray*}
This proves \eqref{keyestint+a} $(ii)$. \eqref{keyestint+a} $(iii)$, $(iv)$, now follows immediately form these estimates and Lemma \ref{lemsl1++}.

\section{Proof of Theorem \ref{th2} and Theorem \ref{parabolicm}}\label{sec8}

Assume that $\mathcal{H}=\partial_t-\mbox{div }A\nabla$ satisfies \eqref{eq3}-\eqref{eq4}. Assume in addition that
 $A$ is real and symmetric. Then \eqref{eq14+} and \eqref{eq14++} hold. To prove Theorem \ref{th2} we have to prove that there exists a constant $C $, depending at most
     on $n$, $\Lambda$, such that  \eqref{keyestint-} holds with this $C $. We first focus on the estimate in \eqref{keyestint-} $(ii)$. Consider
 \begin{eqnarray}\label{ker1}
  \psi_{\lambda}(x,t,y,s):=\lambda K_{1,\lambda}(x,t,y,s)=\lambda\partial_\lambda^{2}\Gamma_\lambda(x,t,y,s).
 \end{eqnarray}
 Then, using Lemma \ref{le2+} we see that $\psi_{\lambda}(x,t,y,s)$ satisfies the Calderon-Zygmund bounds
 \begin{eqnarray}\label{CZ1}
|\psi_{\lambda}(x,t,y,s)|\leq c{|\lambda|}(d_\lambda(x,t,y,s))^{-n-3},
\end{eqnarray}
and
\begin{eqnarray}\label{CZ2}
|\mathbb D^h(\psi_{\lambda}(\cdot,\cdot,y,s))(x,t)|&\leq& c{|\lambda|||h||^\alpha}(d_\lambda(x,t,y,s))^{-n-3-\alpha}\notag\\
&\leq& c{||h||^\alpha}(d_\lambda(x,t,y,s))^{-n-2-\alpha},
\end{eqnarray}
for some $\alpha>0$, whenever $2||h||\leq (|x-y|+|t-s|^{1/2})$ or $2||h||\leq |\lambda|$. Our proof of Theorem \ref{th2} is based on the following two theorems proved below.
\begin{theorem}\label{ltb} Assume that $\psi_{\lambda}$  satisfies \eqref{CZ1} and  \eqref{CZ2}. Let
$$\theta_\lambda f(x,t):=\int_{\mathbb R^{n+1}}\psi_{\lambda}(x,t,y,s)f(y,s)\, dyds,$$
whenever $f\in L^2(\mathbb R^{n+1},\mathbb C)$. Suppose that there exists a system $\{b_Q\}$ of functions, $b_Q:\mathbb R^{n+1}\to\mathbb C$, index by parabolic cubes $Q\subseteq\mathbb R^{n+1}$, and a constant $c$, independent of $Q$, such that for each cube $Q$ the following is true.
\begin{eqnarray}\label{testf}
(i)&&\int_{\mathbb R^{n+1}}|b_Q(x,t)|^2\, dxdt\leq c|Q|,\notag\\
(ii)&&\int_0^{l(Q)}\int_{Q}|\theta_\lambda b_Q(x,t)|^2\, \frac {dxdtd\lambda}\lambda\leq c|Q|,\notag\\
(iii)&&c^{-1}|Q|\leq \mbox{Re }\int_{Q} b_Q(x,t)\, dxdt.
\end{eqnarray}
Then there exists a constant $c$ such that
\begin{eqnarray}\label{testres}
|||\theta_\lambda f|||=\biggl (\int_0^\infty\int_{\mathbb R^{n+1}}|\theta_\lambda f(x,t)|^2\, \frac{dxdtd\lambda}\lambda\biggr )^{1/2}\leq c||f||_2,
\end{eqnarray}
whenever $f\in L^2(\mathbb R^{n+1},\mathbb C)$.
\end{theorem}

The proofs of Theorem \ref{ltb} and Theorem \ref{parabolicm} are given below. We here use Theorem \ref{ltb} and Theorem \ref{parabolicm} to complete the proof of
Theorem \ref{th2}.

\begin{proof}[Proof of \eqref{keyestint-} $(ii)$]
We simply have to produce, using Theorem \ref{ltb} and for $\theta_\lambda$ defined using the kernel in \eqref{ker1},  a system $\{b_Q\}$ of functions satisfying
\eqref{testf} $(i)-(iii)$. To do this we let
\begin{eqnarray*}
b_Q(y,s):=|Q|1_QK_-(A_Q^-,y,s),
\end{eqnarray*}
whenever $(y,s)\in\mathbb R^{n+1}$, where $1_Q$ is the indicator function for the cube $Q$ and where $K_-(A_Q^-,y,s)$ is the to
  $\mathcal{H}^\ast=-\partial_t+\mathcal{L}$ associated Poisson kernel, at $A_Q^-:=(x_Q,-l(Q),t_Q)$, defined with respect to $\mathbb R_-^{n+2}$.  Theorem \ref{parabolicm} applies to $K_-(A_Q^-,\cdot,\cdot)$ modulo trivial modifications. To verify that $b_Q$ satisfies
\eqref{testf} $(i)-(iii)$, we first note that $(i)$ is an immediate consequence of Theorem \ref{parabolicm}. Furthermore,
\begin{eqnarray*}
\int_{\mathbb R^{n+1}}b_Q(y,s)\, dyds=|Q|\omega_-^{A_Q^-}(Q)\geq c^{-1}|Q|,
\end{eqnarray*}
by elementary estimates and where $\omega_-^{A_Q^-}$ is the associated parabolic measure at $A_Q^-$ and defined with respect to $\mathbb R_-^{n+2}$. Hence $(iii)$ follows and it only remains to establish $(ii)$. Let $(x,t)\in Q$, $\lambda\in (0,l(Q))$ and note that
\begin{eqnarray*}
\theta_\lambda b_Q(x,t)&=&\int_{\mathbb R^{n+1}}\lambda\partial_\lambda^{2}\Gamma_\lambda(x,t,y,s)b_Q(y,s)\, dyds\notag\\
&=&\lambda|Q|\int_{Q}\partial_\lambda^{2}\Gamma_\lambda(x,t,y,s)K_-(A_Q^-,y,s)\, dyds\notag\\
&=&\lambda|Q|\bigl (\partial_\lambda^{2}\Gamma(x,t,\lambda, x_Q,t_Q,-l(Q))\bigr ),
\end{eqnarray*}
by the definition of ${A_Q^-}$, $K_-(A_Q^-,y,s)$, and as $\partial_\lambda^{2}\Gamma(x,t,\lambda, x_Q,t_Q,-l(Q))$ solves $\mathcal{H}^\ast u=0$ in $\mathbb R^{n+2}_-$.  Using this, and \eqref{CZ1}, we see that $(ii)$ follows by elementary manipulations. Hence, using Theorem \ref{ltb} we can conclude the validity of \eqref{keyestint-} $(ii)$.
\end{proof}

\begin{proof}[Proof of \eqref{keyestint-} $(i)$]
We first note, that we can throughout the proof assume, without loss of generality, that $f\in C_0^\infty(\mathbb R^{n+1},\mathbb R)$. Second, using  Theorem \ref{parabolicm} and the fact that if $\mathcal{H}=\partial_t-\mbox{div }A\nabla$ satisfies \eqref{eq3}-\eqref{eq4}, and if $A$ is real and symmetric, then  the estimates of the non-tangential maximal function by the square function established in \cite{B2} for the heat equation, remain valid for solutions to $\mathcal{H}u=0$. In particular, let $f\in C_0^\infty(\mathbb R^{n+1},\mathbb R)$ and consider $\lambda>0$ fixed. We let $R$ and $r$ be such that
$\lambda\ll r\ll R$ and such  that the support of
$f$ is contained in $Q_{R/4}(0,0)$. Then, using Theorem \ref{parabolicm} and \cite{B2} we see that
 \begin{eqnarray*}
||(\partial_\lambda \mathcal{S}_\lambda f)1_{Q_r(0,0)}||_2^2\leq c|||\lambda\nabla\partial_\lambda\mathcal{S}_\lambda f|||^2+cR^{n+2}|\partial_\lambda \mathcal{S}_{R/2} f(0,0)|^2,
\end{eqnarray*}
for a constant $c$ depending only on $n$, $\Lambda$. However,
 \begin{eqnarray*}
R^{n+2}|\partial_\lambda \mathcal{S}_{R/2} f(0,0)|^2\leq R^{-n-2}||f||_1^{2}.
\end{eqnarray*}
Hence, first letting $R\to\infty$ and then letting $r\to\infty$ we can conclude that
 \begin{eqnarray}\label{testf+cuusea}
||\partial_\lambda \mathcal{S}_\lambda f||_2\leq c|||\lambda\nabla\partial_\lambda\mathcal{S}_\lambda f|||.
\end{eqnarray}
Using \eqref{eq4.43} we see that
 \begin{eqnarray}\label{testf+cuua}
|||\lambda\nabla\partial_\lambda\mathcal{S}_\lambda f|||\leq c|||\lambda\partial_\lambda^2\mathcal{S}_\lambda f|||+c||f||_2.
\end{eqnarray}
\eqref{testf+cuusea}, \eqref{testf+cuua} and \eqref{keyestint-} $(ii)$ now prove \eqref{keyestint-} $(i)$.
\end{proof}

This completes the proof of Theorem \ref{th2} modulo Theorem \ref{ltb} and Theorem \ref{parabolicm}.

\subsection{Proof of Theorem \ref{ltb}}  Though there are several references for this type of argument, see \cite{CJ}, \cite{H1}, \cite{HMc} and the references therein, we will, for completion, include a sketch/proof of the argument in our context. To start with,  as $\psi_{\lambda}$  satisfies \eqref{CZ1} and   \eqref{CZ2} it is well-known, see \cite{CJ}, that to prove \eqref{testres} it suffices to prove the Carleson measure estimate
\begin{eqnarray}\label{testf+}
\sup_{Q\subset\mathbb R^{n+1}}\frac 1 {|Q|}\int_0^{l(Q)}\int_{Q}|\theta_\lambda 1|^2\, \frac {dxdtd\lambda}\lambda\leq c.
\end{eqnarray}
Using assumption $(iii)$ in the statement of Theorem \ref{ltb}, and a by now well-known stopping time argument, see \cite{H1}, one can conclude that
\begin{eqnarray*}
&&\sup_{Q\subset\mathbb R^{n+1}}\frac 1{|Q|}\int_0^{l(Q)}\int_Q|\theta_\lambda 1|^2\frac {dxdtd\lambda}\lambda
\ \leq \ c\sup_{Q\subset\mathbb R^{n+1}}\frac 1{|Q|}\int_0^{l(Q)}\int_Q|(\theta_\lambda 1 )\mathcal{A}_\lambda^Q b_{Q}|\frac {dxdtd\lambda}\lambda,
     \end{eqnarray*}
     where $\mathcal{A}_\lambda^Q$ denotes the dyadic averaging operator induced by $Q$ and introduced in \eqref{dy}. Hence, to prove \eqref{testf+} it suffices to prove that
\begin{eqnarray}\label{ff1}
\int_0^{l(Q)}\int_Q|(\theta_\lambda 1 )\mathcal{A}_\lambda^Q b_{Q}|\frac {dxdtd\lambda}\lambda\leq c|Q|,
\end{eqnarray}
for all $Q\subset\mathbb R^{n+1}$.  We write
\begin{eqnarray*}
(\theta_\lambda 1)\mathcal{A}_\lambda^Qb_{Q}=\mathcal{R}_\lambda^{(1)}b_{Q}+\mathcal{R}_\lambda^{(2)}b_{Q}+\theta_\lambda b_{Q},
\end{eqnarray*}
where
\begin{eqnarray*}
\mathcal{R}_\lambda^{(1)}b_{Q}&:=&(\theta_\lambda 1)(\mathcal{A}_\lambda^Q-\mathcal{A}_\lambda^Q \P_\lambda)b_{Q},\notag\\
\mathcal{R}_\lambda^{(2)}b_{Q}&:=&((\theta_\lambda 1)\mathcal{A}_\lambda^Q \P_\lambda-\theta_\lambda)b_{Q},
\end{eqnarray*}
and where $\P_\lambda$ is a  parabolic approximation of the identity. Using assumption $(ii)$ in the statement of Theorem \ref{ltb} we see that the contribution from the term $\theta_\lambda b_{Q}$ to the Carleson measure in \eqref{ff1} is controlled. Hence we focus on the contributions from
$\mathcal{R}_\lambda^{(1)}b_{Q}$ and $\mathcal{R}_\lambda^{(2)}b_{Q}$.  Note that
\begin{eqnarray*}
\mathcal{R}_\lambda^{(1)}&=&(\theta_\lambda 1)(\mathcal{A}_\lambda^Q-\mathcal{A}_\lambda^Q \P_\lambda)=(\theta_\lambda 1)\mathcal{A}_\lambda^Q(\mathcal{A}_\lambda^Q-\P_\lambda).
\end{eqnarray*}
Using \eqref{CZ1}, \eqref{CZ2}, and a version of Schur's lemma,  we see that
$$||(\theta_\lambda 1)\mathcal{A}_\lambda^Q||_{2\to 2}\leq c.$$
Thus, by Lemma \ref{little3},
\begin{eqnarray*}
\int _{0}^{l(Q)}\int_{Q}| \mathcal{R}_\lambda^{(1)}b_{Q}(x,t)|^2\, \frac {dxdtd\lambda}\lambda &\leq& c \int_0^\infty\int_{\mathbb R^{n+1}}|(\mathcal{A}_\lambda^Q-\P_\lambda)b_{Q}(x,t)|^2\, \frac {dxdtd\lambda}\lambda\notag\\
&\leq &c \int_{\mathbb R^{n+1}}|b_{Q}(x,t)|^2\, {dxdt}\leq c|Q|.
\end{eqnarray*}
It remains to estimate
\begin{eqnarray*}
\int _{0}^{l(Q)}\int_{Q}| \mathcal{R}_\lambda^{(2)}b_{Q}(x,t)|^2\, \frac {dxdtd\lambda}\lambda.
\end{eqnarray*}
However, using \eqref{CZ1}, \eqref{CZ2}, and that $\mathcal{R}_\lambda^{(2)}1=0$, it follows by a well known orthogonality argument, and
assumption $(i)$ in the statement of Theorem \ref{ltb}, that
\begin{eqnarray*}
\int _{0}^{l(Q)}\int_{Q}| \mathcal{R}_\lambda^{(2)}b_{Q}(x,t)|^2\, \frac {dxdtd\lambda}\lambda\leq \int_{\mathbb R^{n+1}}|b_Q(x,t)|^2\, dxdt\leq c|Q|.
\end{eqnarray*}
This completes the proof of Theorem \ref{ltb}.

\subsection{Proof of Theorem \ref{parabolicm}} We can throughout the proof assume, without loss of generality, that $A$ is smooth. Let $(X,t)=(x,x_{n+1},t)$, $(Y,s)=(y,y_{n+1},s)$. Let $f\in C(\mathbb R^{n+1})\cap L^\infty(\mathbb R^{n+1})$ be such that $f(x,t)\to 0$ as $||(x,t)||\to\infty$. Then there exists a unique weak solution to $\mathcal{H}u=(\partial_t+\mathcal{L})u=0$ in $\mathbb R^{n+2}_+$ such that $u\in C(\mathbb R^{n+1}\times[0,\infty))$, $u(x,0,t)=f(x,t)$ whenever $(x,t)\in \mathbb R^{n+1}$. Furthermore, $||u||_{L^\infty(\mathbb R^{n+2}_+)}\leq ||f||_{L^\infty(\mathbb R^{n+1})}$. This can be proved by exhausting $\mathbb R_+^{n+2}$ by bounded domains $\Omega_j=\{(x,x_{n+1},t):\ (x,t)\in Q_j(0,0),\ 0<x_{n+1}<j\}$, $j\in \mathbb Z_+$, and by constructing $u$ as the limit of $\{u_j\}$ where $\mathcal{H}u_j=0$ in $\Omega_j$ and for $u_j$ having appropriate boundary data on the parabolic boundary of $\Omega_j$. By the Riesz representation theorem there exists a family of regular Borel measures $\{\omega^{(X,t)}:\ (X,t)\in \mathbb R^{n+2}_+\}$ on $\mathbb R^{n+1}$, which we call $\mathcal{H}$-caloric or $\mathcal{H}$-parabolic measures, such that
               $$u(X,t)=\int_{\mathbb R^{n+1}}f(y,s)\, d\omega^{(X,t)}(y,s),$$
               whenever $(X,t)\in \mathbb R^{n+2}_+$. Let $G(X,t,Y,s)$ be the Green function associated to the operator $\mathcal{H}$ in $\mathbb R^{n+2}_+$. Then,
\begin{eqnarray*}
(\partial_t+\mathcal{L}_{X,t})G(X,t,Y,s)=\delta_{(0,0)}(X-Y,t-s),
\end{eqnarray*}
and
\begin{eqnarray}\label{G2--}(-\partial_s+\mathcal{L}_{Y,s})G(X,t,Y,s)=\delta_{(0,0)}(X-Y,t-s).
\end{eqnarray}
The following argument, and hence the proof of Theorem \ref{parabolicm}, relies on the assumption in  \eqref{eq4} in a crucial way. Using that $A$ is assumed smooth it follows that the solution to the Dirichlet problem $\mathcal{H}u=0$ in $\mathbb R^{n+2}_+$, $u=f$ on $\mathbb R^{n+1}$, equals
\begin{eqnarray*}
u(X,t)=\int_{\mathbb R^{n+1}} K(X,t,y,s)f(y,s)\, dyds,
\end{eqnarray*}
where
\begin{eqnarray*}
K(X,t,y,s)&:=&\langle \nabla_YG(X,t,Y,s),A(Y)e_{n+1}\rangle|_{y_{n+1}=0}\notag\\
&=&a_{n+1,n+1}(y)\partial_{y_{n+1}}G(X,t,Y,s)|_{y_{n+1}=0}.
\end{eqnarray*}
Using \eqref{eq3} we see that $a_{n+1,n+1}$ is uniformly bounded from below. Let $Q\subset\mathbb R^{n+1}$ be a parabolic cube and let
 $A_Q:=(X_Q,t_Q):=(x_Q,l(Q),t_Q)$, where $(x_Q,t_Q)$ is the center of the cube and $l(Q)$ defines its size.
 We write $Q=\hat Q\times (t_Q-l(Q)^2/2,t_Q+l(Q)^2/2)$ where $\hat Q\subset\mathbb R^{n}$ is a  (elliptic) cube in the space variables only. Then
   \begin{eqnarray}\label{G6}
  \int_{Q}(K(A_Q,y,s))^2\, dyds&=&\int_{t_Q-l(Q)^2/2}^{t_Q+l(Q)^2/2}\int_{\hat Q}(K(X_Q,t_Q,y,s))^2\, dyds\notag\\
  &=&\int_{-l(Q)^2/2}^{l(Q)^2/2}\int_{\hat Q}(K(X_Q,0,y,s))^2\, dyds\notag\\
  &=&\int_{-l(Q)^2/2}^{l(Q)^2/2}\int_{\hat Q}(K(X_Q,0,y,-s))^2\, dyds,
\end{eqnarray}
by the translation invariance in the time-variable due to \eqref{eq4}. Using the Harnack inequality we see that
  \begin{eqnarray}\label{G5}
  (K(X_Q,0,y,-s))^2\leq c K(X_Q,0,y,-s)K(X_Q,16l(Q)^2,y,s),
\end{eqnarray}
whenever $(y,s)\in \hat Q\times [-l(Q)^2/2,l(Q)^2/2]$.  Let $$\phi\in C_0^\infty
(\mathbb R^{n+2}\setminus\bigl(\{(X_Q,0)\}\cup\{(X_Q,16l(Q)^2)\}\bigr)$$ be such that
  \begin{eqnarray}\label{G5ap}
  \phi(y,y_{n+1},s)=1,
\end{eqnarray}
whenever $(y,y_{n+1},s)\in \hat Q\times [-l(Q)/16,l(Q)/16]\times [-l(Q)^2/2,l(Q)^2/2]$, and
  \begin{eqnarray}\label{G5ap+}
  \phi(y,y_{n+1},s)=0,
\end{eqnarray}
whenever $(y,y_{n+1},s)\in \mathbb R^{n+2}\setminus\bigl(2\hat Q\times [-l(Q)/8,l(Q)/8]\times [-l(Q)^2,l(Q)^2]\bigr)$. Furthermore, we choose $\phi$ so that
  \begin{eqnarray}\label{G5ap++}
   \ |\nabla_Y\phi(Y,s)|\leq cl(Q)^{-1},\ |\partial_s\phi(Y,s)|\leq cl(Q)^{-2},
\end{eqnarray}
whenever $(Y,s)\in  \mathbb R^{n+2}$. Let
$\Psi(Y,s):= \phi(Y,s)\partial_{y_{n+1}}v(Y,s)$, where $$v(Y,s):=G(X_Q,0,Y,-s),$$ and let $$\tilde v(Y,s):=G(X_Q,16l(Q)^2,Y,s).$$ Using \eqref{G2--}  we see that
\begin{eqnarray*}\label{G4a}
0&=&\int_{\mathbb R^{n+21}_+} \bigl ((-\partial_s+\mathcal{L}_{Y,s})G(X_Q,16l(Q)^2,Y,s)\bigr)\Psi(Y,s)\, dYds\notag\\
&=&\int_{\mathbb R^{n+2}_+} \bigl ((-\partial_s+\mathcal{L}_{Y,s})\tilde v(Y,s)\bigr)\Psi(Y,s)\, dYds.
\end{eqnarray*}
Using this identity, and integrating by parts, we see that
\begin{eqnarray}\label{G4b}
I&:=&\int_{\mathbb R^{n+1}} \Psi(Y,s)|_{y_{n+1}=0}K(X_Q,16l(Q)^2,y,s)\, dyds\notag\\
&=&\int_{\mathbb R^{n+2}_+} \bigl ((\partial_s+\mathcal{L}_{Y,s})\Psi(Y,s)\bigr )\tilde v(Y,s)\, dYds.
\end{eqnarray}
We will now use the identity in \eqref{G4b} to prove Theorem \ref{parabolicm}. Indeed,
\begin{eqnarray*}
(\partial_s+\mathcal{L}_{Y,s})\Psi&=&\partial_s\Psi-\div(A\nabla_Y\Psi)\notag\\
&=&\partial_{y_{n+1}}v\partial_s\phi-\div((\partial_{y_{n+1}}v)A\nabla_Y\phi)-A\nabla_Y \partial_{y_{n+1}}v\cdot\nabla_Y\phi\notag\\
&&+\phi(\partial_s\partial_{y_{n+1}}v-\div(A\nabla_Y\partial_{y_{n+1}}v)).
\end{eqnarray*}
The key observation is, as $A$ is independent of $y_{n+1}$, that
\begin{eqnarray*}
\partial_s\partial_{y_{n+1}}v-\div(A\nabla_Y\partial_{y_{n+1}}v)&=&\partial_{y_{n+1}}\bigl (\partial_sv-\div(A\nabla_Yv)\bigr )\notag\\
&=&\partial_{y_{n+1}}\bigl (((-\partial_s+\mathcal{L}_Y)G)(X_Q,0,Y,-s)\bigr )=0,
\end{eqnarray*}
on the support of $\phi$. This is due to the presence of the minus sign in front of $s$ in $G(X_Q,0,Y,-s)$. Hence, using \eqref{G4b} and elementary manipulations, we see that
\begin{eqnarray*}
I=I_1+I_2-I_3.
\end{eqnarray*}
where
\begin{eqnarray*}\label{G4f}
I_1&:=&\int_{\mathbb R^{n+2}_+} \partial_{y_{n+1}}G(X_Q,0,Y,-s)(\partial_s\phi(Y,s))G(X_Q,16l(Q)^2,Y,s)\, dYds,\notag\\
I_2&:=&\int_{\mathbb R^{n+2}_+} \partial_{y_{n+1}}G(X_Q,0,Y,-s)(A\nabla_Y\phi)\cdot\nabla_YG(X_Q,16l(Q)^2,Y,s)\, dYds,\notag\\
I_3&:=&\int_{\mathbb R^{n+2}_+} (A\nabla_Y \partial_{y_{n+1}}G(X_Q,0,Y,-s)\cdot\nabla_Y\phi)G(X_Q,16l(Q)^2,Y,s)\, dYds.
\end{eqnarray*}
Recall that $\phi$ satisfies \eqref{G5ap}-\eqref{G5ap++} and let $E=\mathbb R^{n+2}_+\cap\overline{\{(Y,s): \phi(Y,s)\neq 0\}}$. Using this,
\begin{eqnarray*}
|I_1|&\leq&cl(Q)^{-2}\int_{E} |\partial_{y_{n+1}}G(X_Q,0,Y,-s)||G(X_Q,16l(Q)^2,Y,s)|\, dYds,\notag\\
|I_2|&\leq&cl(Q)^{-1}\int_{E} |\partial_{y_{n+1}}G(X_Q,0,Y,-s)||\nabla_YG(X_Q,16l(Q)^2,Y,s)|\, dYds,\notag\\
|I_3|&\leq&cl(Q)^{-1}\int_{E} |\nabla_Y \partial_{y_{n+1}}G(X_Q,0,Y,-s)||G(X_Q,16l(Q)^2,Y,s)|\, dYds.
\end{eqnarray*}
Hence, using energy estimates and Gaussian bounds for the fundamental solution  we deduce
\begin{eqnarray*}
|I|\leq |I_1|+|I_2|+|I_3|\leq c|Q|^{-1}.
\end{eqnarray*}
Using this and \eqref{G4b} we see that
   \begin{eqnarray*}
\int_{-l(Q)^2/2}^{l(Q)^2/2}\int_{\hat Q} K(X_Q,0,y,-s)K(X_Q,16l(Q)^2,y,s)\, dyds\leq  c|Q|^{-1}.
\end{eqnarray*}
Hence, using  \eqref{G6} and \eqref{G5} we can conclude that   \begin{eqnarray}\label{G6ed}
  \int_{Q}(K(A_Q,y,s))^2\, dyds\leq  c|Q|^{-1},
\end{eqnarray}
whenever $Q\subset\mathbb R^{n+1}$ is a parabolic cube, for a constant $c\geq 1 $, depending only on $n$ and $\Lambda$.  Put together  Theorem \ref{parabolicm} follows.\\

     \noindent
{\it Acknowledgement}. The authors thank an anonymous referee for a very careful reading of the paper and for valuable suggestions.

\end{document}